\documentclass{amsart}

\usepackage{amsmath,amssymb,amsthm}
\usepackage{mathtools}
\usepackage{enumitem}
\usepackage[colorlinks=true,linkcolor=blue,citecolor=blue,urlcolor=blue,backref=page]{hyperref}
\renewcommand*{\backrefalt}[4]{\ifcase #1 \or $\uparrow$\,#2\else $\uparrow$\,#2\fi}

\usepackage{tikz}
\usetikzlibrary{cd,calc,matrix,arrows}
\usepackage{orcidlink}

\theoremstyle{plain}
\newtheorem{theorem}{Theorem}[section]
\newtheorem{lemma}[theorem]{Lemma}
\newtheorem{proposition}[theorem]{Proposition}
\newtheorem{corollary}[theorem]{Corollary}

\theoremstyle{definition}
\newtheorem{definition}[theorem]{Definition}
\newtheorem{remark}[theorem]{Remark}
\newtheorem{question}[theorem]{Question}

\newtheorem{theoremalph}{Theorem}

\newcommand{\C}{\mathbb{C}}
\newcommand{\F}{\mathbb{F}}
\newcommand{\Q}{\mathbb{Q}}
\newcommand{\Z}{\mathbb{Z}}
\newcommand{\A}{\mathcal{A}}
\renewcommand{\L}{\mathcal{L}}
\newcommand{\VV}{\mathcal{V}}
\newcommand{\RR}{\mathcal{R}}
\newcommand{\R}{\mathbb{R}}
\newcommand{\TT}{\mathbb{T}}
\newcommand{\CP}{\mathbb{CP}}

\newcommand{\m}{\mathfrak{m}}
\renewcommand{\k}{\Bbbk}
\newcommand{\one}{\mathbf{1}}
\newcommand{\wG}{\widehat{G}}

\newcommand{\Trop}{\operatorname{Trop}}
\newcommand{\Tors}{\operatorname{Tors}}
\newcommand{\res}{\operatorname{res}}
\newcommand{\rank}{\operatorname{rank}}
\newcommand{\tors}{\mathrm{tors}}

\DeclareMathOperator{\Hom}{Hom}
\DeclareMathOperator{\Aut}{Aut}
\DeclareMathOperator{\GL}{GL}
\DeclareMathOperator{\ab}{ab}
\DeclareMathOperator{\gr}{gr}
\DeclareMathOperator{\coker}{coker}
\DeclareMathOperator{\depth}{depth}
\DeclareMathOperator{\Ind}{Ind}

\DeclareMathOperator{\ch}{char}
\DeclareMathOperator{\Gal}{Gal}
\DeclareMathOperator{\CH}{\mathsf{CH}}
\DeclareMathOperator{\rf}{rf}
\DeclareMathOperator{\id}{id}

\newcommand{\surj}{\twoheadrightarrow}
\newcommand{\inj}{\hookrightarrow}
\newcommand\isom{\xrightarrow{ \,\smash{\raisebox{-0.3ex}{\ensuremath{\scriptstyle\simeq}}}\,}}

\begin{document}

\title{Profinite completions and cohomology jump loci}

\author[Alexander~I.~Suciu]{Alexander~I.~Suciu\,\orcidlink{0000-0002-5060-7754}}
\address{Department of Mathematics, Northeastern University,
Boston, MA 02115}
\email{\href{mailto:a.suciu@northeastern.edu}{a.suciu@northeastern.edu}}
\urladdr{\href{https://suciu.sites.northeastern.edu}{suciu.sites.northeastern.edu}}




\begin{abstract}
The profinite completion of the fundamental group of a complex algebraic
variety is its algebraic fundamental group, which does not change under
Galois conjugation; we ask which topological invariants it determines.  
For smooth quasi-projective varieties and compact K\"ahler manifolds, 
an isomorphism of profinite completions matches their cohomology jump 
loci, up to an isogeny in general and, when it is compatible with an
identification of the abelianizations, exactly; this holds in degree one, and 
in all degrees for aspherical varieties whose fundamental groups are good 
in the sense of Serre.  Applied to finite covers, this shows that the tropical
bounds for the Bieri--Neumann--Strebel--Renz invariants are profinite
invariants, and it yields a new proof of a theorem of Hughes, Llosa
Isenrich, Py, Stover, and Vidussi: the BNS invariant of a K\"ahler group is
a profinite invariant.  The profinite completion of a finitely generated
group also determines the lower central series quotients of its verbal
quotients, torsion included, among them the Chen groups.  For hyperplane
arrangements, no arithmetic Zariski pair is distinguished by any of these
invariants, while the profinite completion of an arrangement group is not
determined by the intersection lattice, and does not determine the group.
Our main application concerns Rybnikov's arrangements, the first known
lattice-isomorphic pair with non-isomorphic fundamental groups: already the
pro-$3$ completions of the third nilpotent quotients of their groups are
not isomorphic, so the complements have non-isomorphic algebraic
fundamental groups.
\end{abstract}

\maketitle

\setcounter{tocdepth}{1}
\tableofcontents
\newpage

\section{Introduction}
\label{sect:intro}

\subsection{Profinite invariants}
\label{subsec:profinite-intro}

The profinite completion $\wG$ of a finitely generated group $G$ is the
inverse limit of its finite quotients, and the question of how much of $G$
it remembers goes back to Grothendieck \cite{Grothendieck-mm70}.  A finitely
generated, residually finite group is \emph{profinitely rigid} if it is
determined by its profinite completion among all such groups.  Whether free
groups and surface groups are profinitely rigid is a long-standing question
attributed to Remeslennikov, and the first full-sized examples of
profinitely rigid groups were found only recently \cite{BMRS-ann20}; see
\cite{Reid-icm18} for a survey.  As rigidity is rare, one asks instead which
invariants of $G$ are determined by $\wG$.

For a complex algebraic variety $X$, the profinite completion of $\pi_1(X)$
is the algebraic fundamental group of $X$.  It is therefore unchanged when
an automorphism $\sigma \in \Aut(\C)$ acts on the equations of $X$, although
$\pi_1(X)$ and $\pi_1(X^\sigma)$ need not be isomorphic \cite{Serre-cras64};
see \cite{Bauer-Catanese-Grunewald-inv15, Stover-pams19, Stover-jems26} for
aspherical examples, and Theorem~\ref{thm:arr} below for complements of
hyperplane arrangements.  An invariant of $\pi_1(X)$ that is
determined by its completion thus cannot distinguish $X$ from $X^\sigma$.

Several invariants are known to be profinite: the first Betti numbers of
finite-index subgroups \cite{Reid-lms15}; for $3$-manifolds, fibering over
the circle \cite{JaikinZapirain-gt20}, the Thurston norm
\cite{Boileau-Friedl-amstud20, Liu-invent23}, and the Alexander polynomial
of a knot \cite{Boileau-Friedl-amstud20, Ueki-agt18}; and for K\"ahler
groups, by work of Hughes, Llosa Isenrich, Py, Stover, and Vidussi, the
universal homomorphism to orbisurface groups and the Bieri--Neumann--Strebel
invariant \cite[Thms.~C and~D]{HLPSV-adv26}.

This paper studies two further kinds of profinite invariants.  The first
are the cohomology jump loci of a space $X$ with fundamental group $G$,
which record how the homology of $X$ with coefficients in rank-one local
systems varies with the local system.  At characters of finite order, they
compute the Betti numbers of the finite abelian covers of $X$, which $\wG$
determines in degree one.  For smooth quasi-projective varieties and compact 
K\"ahler manifolds, we show that $\wG$ determines the loci themselves up to
isogeny, in degree one and, under further hypotheses, in all degrees
(Theorem~\ref{thm:loci}), and with them the tropical bounds for the
Bieri--Neumann--Strebel--Renz invariants (Theorem~\ref{thm:sigma}).  The
second are the lower central series quotients of the verbal quotients of
$G$, torsion included (Theorem~\ref{thm:graded}).  For hyperplane
arrangements, these results show that Galois conjugation preserves the
characteristic varieties $\VV^1_s$, the lower central series quotients,
and the Chen groups, and that the profinite completion of an arrangement
group is not determined by the intersection lattice
(Theorem~\ref{thm:arr}).  Our main application concerns the pair of
arrangements with which Rybnikov showed that the fundamental group of an
arrangement complement is not determined by the intersection lattice: even
the algebraic fundamental groups of the two complements are not isomorphic
(Theorem~\ref{thm:ryb-pair}).

\subsection{Jump loci}
\label{subsec:jumploci}

Let $X$ be a connected CW-complex with finite $k$-skeleton, and let
$G = \pi_1(X,x_0)$.  For a field $\k$, a character
$\rho \in \Hom(G,\k^\times)$ determines a rank-one local system $\k_\rho$
on $X$, and the \emph{characteristic varieties} of $X$ are the Zariski
closed sets
\begin{equation}
\label{eq:cv-def}
\VV^q_s(X,\k) = \bigl\{\rho \in \Hom(G,\k^{\times})
: \dim_{\k} H_q(X;\k_\rho) \ge s \bigr\},
\qquad 0 \le q \le k,\ s \ge 1 .
\end{equation}
The \emph{depth} of a character is
$\depth^q_\k(\rho) = \dim_\k H_q(X;\k_\rho)$, so that $\VV^q_s(X,\k)$
consists of the characters of depth at least $s$, and
$\depth^q_\k(\one) = \dim_\k H_q(X;\k)$.  We suppress $\k$ when it is $\C$.

Two features of \eqref{eq:cv-def} are used throughout.  First, the loci are
cut out by determinantal conditions on equations with integer
coefficients, so each $\VV^q_s(X,\k)$ is the set of $\k$-points of an
affine scheme defined over $\Z$; see \cite[\S2.1]{Denham-Suciu-plms14}.
Consequently $\VV^q_s(X,\C)$ is defined over $\Q$, and its set of torsion
points is stable under the action of $\Gal(\overline\Q/\Q)$
permuting roots of unity \cite[Rem.~2.7]{Denham-Suciu-plms14}.  Second, in
degree $q=1$ the loci depend only on $G$ --- indeed only on its maximal
metabelian quotient $G/G''$ --- whereas for $q \ge 2$ they genuinely depend
on the space.  This is why we work with spaces rather than groups from the
outset; it is also the reason the two halves of the argument below behave
so differently in degree one and above it.

We work over $\C$: the Galois stability just noted makes the depth constant
on Galois orbits of torsion characters, which is what allows the Betti
numbers of finite cyclic covers to determine the depths of individual
characters (\S\ref{subsec:inversion}).  Over $\overline{\F}_p$ this fails
(Remark~\ref{rem:char-p}), and over $\R$ there are too few torsion
characters.

The geometric input is a structure theorem: for a smooth quasi-projective
variety $X$, each $\VV^q_s(X,\C)$ is a finite union of torsion-translated
subtori of the character group, by work of Arapura \cite{Arapura-jag97},
Artal, Cogolludo, and Matei \cite{Artal-Cogolludo-Matei-gt13}, and Budur and Wang
\cite{Budur-Wang-ens15}; the same holds for compact K\"ahler manifolds
\cite[\S12]{Suciu-mathann21}, \cite{Wang-mrl16}.  In particular, these
loci are the Zariski closures of their torsion points; we call this
Property~$\mathsf{TT}$ (\S\ref{subsec:transport}).

\subsection{Results}
\label{subsec:results}

An isomorphism $\phi \colon \widehat{G_1} \to \widehat{G_2}$ carries each
torsion character $\rho$ of $G_2$ to a torsion character $\rho^\phi$ of
$G_1$ of the same order (\S\ref{subsec:transport}).  To compare the loci,
which lie in different character groups, one needs more: $\phi$ is
\emph{rigidified} by an isomorphism
$\psi \colon (G_1)_{\ab} \to (G_2)_{\ab}$ if the isomorphism $\phi_{\ab}$
that it induces between the profinite completions of the abelianizations
equals $u \cdot (\psi \otimes \widehat\Z)$ for some
$u \in \widehat\Z^\times$ (Definition~\ref{def:rigid}; this is the
$\widehat\Z^\times$-regularity of Liu \cite{Liu-invent23}).  The
isomorphisms induced by Galois conjugation of arrangements are rigidified,
with $u$ the cyclotomic character (Theorem~\ref{thm:arr}\ref{arr-galois}).

\begin{theoremalph}
\label{thm:loci}
Let $X_1$ and $X_2$ be smooth complex quasi-projective varieties or compact
K\"ahler manifolds, set $G_i = \pi_1(X_i)$, and let
$\phi \colon \widehat{G_1} \to \widehat{G_2}$ be an isomorphism.
\begin{enumerate}[label=\textup{(\roman*)}, itemsep=2pt]
\item \label{loci-depth}
$\depth^1_\C(\rho^\phi) = \depth^1_\C(\rho)$ for every torsion character
$\rho$ of $G_2$.
\item \label{loci-rigid}
If $\phi$ is rigidified by $\psi \colon (G_1)_{\ab} \to (G_2)_{\ab}$,
then
$\psi^\vee(\VV^1_s(X_2,\C)) = \VV^1_s(X_1,\C)$ for all $s \ge 1$.
\item \label{loci-isog}
In general, for every integer $m \ge 1$ there is an injective homomorphism
$L \colon (G_1)_{\ab} \to (G_2)_{\ab}$ with finite cokernel and
$L \otimes \Z/m = \phi_{\ab} \otimes \Z/m$, such that the isogeny
$L^\vee \colon \Hom(G_2,\C^\times) \to \Hom(G_1,\C^\times)$ maps the
irreducible components of $\VV^1_s(X_2,\C)$ bijectively onto those of
$\VV^1_s(X_1,\C)$, preserving dimension, for all $s \ge 1$.  The same $L$
may be chosen to do this simultaneously for the pull-backs of the jump loci
of finitely many pairs of corresponding finite covers.
\item \label{loci-all}
If moreover $X_1$ and $X_2$ are aspherical and $G_1$ and $G_2$ are good,
then \ref{loci-depth}--\ref{loci-isog} hold for the loci $\VV^q_s$ in all
degrees $q$.
\end{enumerate}
\end{theoremalph}

Theorem~\ref{thm:loci} is a special case of Theorem~\ref{thm:main} and its
extension to finite covers, Proposition~\ref{prop:covers}, which hold for
connected CW-complexes of finite type with Property~$\mathsf{TT}$.
Its first input is that $\phi$ matches the Betti numbers of corresponding
finite cyclic covers, a condition equivalent to part~\ref{loci-depth}
(Lemma~\ref{lem:depth}); it holds in degree one for every $\phi$, and in all
degrees for aspherical spaces with good fundamental groups
(Proposition~\ref{prop:CH}).  Its second input is Property~$\mathsf{TT}$.  When
$\phi$ is not rigidified, the passage from torsion points to loci goes
through an approximation of $\phi_{\ab}$, modulo any given integer, by
homomorphisms $(G_1)_{\ab} \to (G_2)_{\ab}$ compatible with the subtori
underlying the components (Proposition~\ref{prop:approx}).  Residual
finiteness is never needed (\S\ref{subsec:verbal-scope}), and
part~\ref{loci-all} applies to fiber-type arrangements and to the
reflection arrangements of all but possibly twelve irreducible complex
reflection groups (\S\ref{sect:good}).  Without Property~$\mathsf{TT}$, the
conclusion may still hold, as it does for knot complements
(Remark~\ref{rem:knots} and Question~\ref{quest:no-T}).

The jump loci of $X$ and of its finite covers also bound the
Bieri--Neumann--Strebel--Renz invariants $\Sigma^q(X;\Z)$, by way of
tropicalization \cite{Suciu-mathann21}.  The authors of \cite{HLPSV-adv26}
ask whether, for groups with isomorphic profinite completions, some linear
isomorphism between the first cohomology groups matches the BNS invariants
\cite[Question~1.8]{HLPSV-adv26}.  For the tropical bounds, it does:

\begin{theoremalph}
\label{thm:sigma}
Let $X_1$, $X_2$, and $\phi$ be as in Theorem~\textup{\ref{thm:loci}}, and
let $m \ge 1$.
\begin{enumerate}[label=\textup{(\roman*)}, itemsep=2pt]
\item \label{sigma-bound}
For every finite set of pairs of corresponding finite covers, there is $L$ as
in Theorem~\textup{\ref{thm:loci}\ref{loci-isog}} such that the dual
isomorphism $L^\vee_\R$, from $H^1(X_2;\R)$ to $H^1(X_1;\R)$, carries the
tropical upper bound for $\Sigma^1(X_2;\Z)$ furnished by each of these
covers \textup{(}Proposition~\textup{\ref{prop:virtual-bound}}\textup{)}
onto the corresponding bound for $\Sigma^1(X_1;\Z)$.  If moreover $X_1$ and
$X_2$ are aspherical with good fundamental groups, the same holds for the
bounds on $\Sigma^q(X_i;\Z)$, for every $q$.
\item \label{sigma-kahler}
If $X_1$ and $X_2$ are compact K\"ahler manifolds, there is $L$ as in
Theorem~\textup{\ref{thm:loci}\ref{loci-isog}} such that
$L^\vee_\R(\Sigma^1(G_2)) = \Sigma^1(G_1)$.
\end{enumerate}
\end{theoremalph}

Part~\ref{sigma-kahler} is due to Hughes, Llosa Isenrich, Py, Stover, and
Vidussi \cite[Thm.~D]{HLPSV-adv26}, with $L$ a specialization in the sense
of Liu, which agrees with $\phi_{\ab}$ modulo~$2$ on the torsion-free
quotients of the abelianizations; ours agrees with $\phi_{\ab}$ modulo any
given integer, torsion included.  Our proof also differs from theirs
(Remark~\ref{rem:hlipsv}): the jump loci of finite covers, which are
twisted jump loci for representations with finite image
\cite{Liu-Suciu-arx26}, replace the twisted Alexander polynomials behind
the property $\mathsf{TAP}_1$ of Hughes and Kielak
\cite{Hughes-Kielak-rmi25}, and Delzant's theorem and Selberg's lemma
replace the profinite invariance of the universal homomorphism.

The applications to arrangements also use the lower central series:
$\wG$ determines the graded pieces $\gr_r(G/W(G))$, torsion included, for
every verbal subgroup $W(G)$ (Theorem~\ref{thm:graded}).  For
$W(G) = \{1\}$ and $G$ residually finite, this is due to Nikolov and
Wykowski \cite[Prop.~2.4]{Nikolov-Wykowski-mathann26}; our proof needs
neither residual finiteness nor the Nikolov--Segal theorem, and
$W(G) = G''$ gives the Chen groups.  The nilpotent quotients
$G/W(G)\gamma_r(G)$ themselves are not determined by $\wG$
(\S\ref{subsec:verbal-scope}), but their profinite completions are, and these
can separate groups whose graded pieces agree in low degrees
(Theorem~\ref{thm:ryb-pair}).

\subsection{Hyperplane arrangements}
\label{subsec:arr-intro}

This study began with questions about hyperplane arrangements raised in
\cite[Ch.~5]{CDFSSTY-pre09}: which invariants of the complement of an
arrangement $\A$ are determined by its intersection lattice, or, for the
fundamental group, by the truncation $L_{\le 2}(\A)$?  The group itself is
not, by Rybnikov's theorem \cite{Rybnikov-faa11, ACCM-aspm07}.  The
resonance varieties and the ranks of the lower central series quotients
are so determined; the torsion in these quotients is not
\cite{Artal-Guerville-ViuSos-expmath20}; for the profinite completion the
question was open, and for the characteristic varieties it still is.
Galois-conjugate
arrangements have isomorphic intersection lattices; they form an
\emph{arithmetic Zariski pair} if no homeomorphism of projective space
carries one onto the other \cite{ACGM-racsam17, Guerville-Balle-top16}.
Theorem~\ref{thm:arr} collects what our results say about these
questions; for applications to linking invariants of line arrangements,
see \cite{Suciu-linking-classes}.

\begin{theoremalph}
\label{thm:arr}
Let $\A$ be an arrangement of hyperplanes in $\C^\ell$, with projectivized
complement $U(\A)$, and set $G(\A) = \pi_1(U(\A))$.
\begin{enumerate}[label=\textup{(\roman*)}, itemsep=1.5pt]
\item \label{arr-galois}
For every $\sigma \in \Aut(\C)$, there is an isomorphism
$\widehat{G(\A)} \to \widehat{G(\A^\sigma)}$ rigidified by the bijection
$H \mapsto H^\sigma$, with unit the cyclotomic character $\chi(\sigma)$.
Consequently, $\A$ and $\A^\sigma$ have the same characteristic varieties
$\VV^1_s$, the same lower central series quotients and Chen groups,
torsion included, and corresponding tropical bounds for their BNSR
invariants; in particular, none of these invariants distinguishes an
arithmetic Zariski pair.
\item \label{arr-agv}
The profinite completion of an arrangement group is not determined by the
intersection lattice: for the $13$-line arrangements $\A^{\pm}$ of Artal
Bartolo, Guerville-Ball\'e, and Viu-Sos
\cite{Artal-Guerville-ViuSos-expmath20},
$L_{\le 2}(\A^+) \cong L_{\le 2}(\A^-)$, but
$\widehat{G(\A^+)} \not\cong \widehat{G(\A^-)}$.
\item \label{arr-acgm}
An arrangement group is not determined by its profinite completion: there
are pairs of Galois-conjugate line arrangements with non-isomorphic groups
\cite{ACGM-racsam17}, and triples whose groups are pairwise non-isomorphic
\cite{Guerville-Balle-mathz22}; the same holds for the maximal residually
finite quotients of these groups.
\end{enumerate}
\end{theoremalph}

In part~\ref{arr-galois}, the equality of the loci $\VV^1_s$ also follows,
and in all degrees, from the absoluteness of jump loci in the sense of
Simpson and of Budur and Wang \cite{Simpson-ens93, Budur-Wang-asens20},
which passes through the Riemann--Hilbert correspondence; the profinite
argument shows that these loci are matched by every isomorphism of
completions rigidified by $H \mapsto H^\sigma$ (\S\ref{subsec:galois}).
In part~\ref{arr-agv}, the groups differ by a $\Z/2$ summand in $\gr_4$,
which Theorem~\ref{thm:graded} makes a profinite invariant; this answers
Question~\ref{quest:profinite-comb} in the negative.
Already the third nilpotent quotients differ, even after pro-$2$
completion, as this torsion is that of their Schur multipliers; the
class-three computation behind Theorem~\ref{thm:ryb-pair} below locates the
difference (Proposition~\ref{prop:agv-class3}).  In
part~\ref{arr-acgm}, the groups were separated in \cite{ACGM-racsam17,
Guerville-Balle-mathz22}, and their completions are isomorphic by
part~\ref{arr-galois}, as noted in \cite{Guerville-Balle-top16,
ACGM-racsam17}; the groups differ by an obstruction over $\Z$ that the
completion evades through the unit $\chi(\sigma)$
(Remark~\ref{rem:acgm-mechanism}).  Their residually finite quotients are
thus not profinitely rigid (Corollary~\ref{cor:acgm-rf}).

Rybnikov's arrangements $\A^{++}$ and $\A^{+-}$ are two realizations,
defined over $\Q(\omega)$ with $\omega = e^{2\pi i/3}$, of a matroid on
thirteen elements, obtained by gluing two copies of the MacLane matroid
along a three-element flat \cite[\S4]{Rybnikov-faa11}.  They are not
Galois conjugate \cite[Ex.~1.3(2)]{Guerville-Balle-arx24}, so
Theorem~\ref{thm:arr}\ref{arr-galois} does not apply, and such 
nonarithmetic pairs are the natural candidates for a profinite separation
\cite[\S1]{Guerville-Balle-arx24}.

\begin{theoremalph}
\label{thm:ryb-pair}
Let $G^{+\pm} = G(\A^{+\pm})$.  The pro-$3$ completions of the nilpotent
quotients $G^{++}/\gamma_4(G^{++})$ and $G^{+-}/\gamma_4(G^{+-})$ are not
isomorphic.  Consequently, $\widehat{G^{++}} \not\cong \widehat{G^{+-}}$;
that is, the complements of $\A^{++}$ and $\A^{+-}$ have non-isomorphic
algebraic fundamental groups.
\end{theoremalph}

Since an isomorphism of groups induces one of pro-$3$ completions,
Theorem~\ref{thm:ryb-pair} recovers Rybnikov's theorem that the quotients
$G^{+\pm}/\gamma_4(G^{+\pm})$ are not isomorphic
\cite[Thm.~4.1]{Rybnikov-faa11}, whereas such a non-isomorphism need not
survive completion, as the pair in Theorem~\ref{thm:arr}\ref{arr-acgm} shows.  
The pair of Theorem~\ref{thm:arr}\ref{arr-agv} is separated by the torsion
in $\gr_4$.  For
Rybnikov's pair, the lower central series quotients and the Chen groups of
$G^{++}$ and $G^{+-}$ agree, torsion included, at least through
degree~$5$ (\S\ref{subsec:rybnikov}), and the separation comes from the
extension class of the third nilpotent quotients, which is not determined
by the graded pieces.

Rybnikov's proof has two steps, both carried out in detail in
\cite{ACCM-aspm07}: no isomorphism of the quotients induces the identity
on $H_1$, and every automorphism of $H_1$ preserving the holonomy relations
is $\pm P_\sigma$, where $P_\sigma$ permutes the meridians by an
automorphism $\sigma$ of the matroid.  Both steps survive pro-$3$
completion.  The second holds over $\widehat\Z_3$, with $\pm 1$ replaced
by a $3$-adic unit $u$ (Lemma~\ref{lem:ryb-rigid}).  The automorphisms of
$\C$ fixing $\Q(\omega)$ induce automorphisms of $\widehat{G^{++}}$ acting
on $H_1 \otimes \widehat\Z$ as multiplication by any prescribed unit
congruent to $1$ modulo~$3$ (Lemma~\ref{lem:galois-units}), and this reduces $u$ to
$\pm 1$.  An exact machine computation, with presentations obtained from two 
independent braid monodromy calculations, then shows that none of the $24$ resulting
maps $\pm P_\sigma$ is induced by a homomorphism of pro-$3$ completions,
the obstruction having order~$3$ (Proposition~\ref{prop:ryb-obstruction});
for primes $p \ne 3$, by contrast, the pro-$p$ completions of the two
quotients are isomorphic.  Heuristically, the prime~$3$ comes from the
MacLane matroid, whose two realizations are complex conjugate: Galois
isomorphisms acting on $H_1$ through units congruent to $2$ modulo~$3$
evade the class-three obstruction to an isomorphism of their groups
inducing the identity on $H_1$, but in $\A^{+-}$ the two copies of the
MacLane matroid would require units in different classes modulo~$3$
(\S\ref{subsec:rybnikov}).

For the characteristic varieties, a lattice-isomorphic pair with different
loci $\VV^1_s$ cannot be arithmetic, by Theorem~\ref{thm:loci}, and must
differ in translated components or in isolated torsion points
(Corollary~\ref{cor:search} and Remark~\ref{rem:through-one}); such
components are studied in \cite{Suciu-isolated-points}.  Whether such a
pair exists remains open (Question~\ref{quest:cv-comb}).
The nonarithmetic pairs of Theorems~\ref{thm:arr}\ref{arr-agv}
and~\ref{thm:ryb-pair} are not such pairs: their loci $\VV^1_s$ agree,
although their completions do not (\S\ref{subsec:search}).

\subsection{Organization}
\label{subsec:organization}

Section~\ref{sect:engine} introduces Property~$\mathsf{TT}_k$ and
condition~$\CH_k$, states the general Theorem~\ref{thm:main}, and proves
its first two parts.  Sections~\ref{sect:degree-one} and~\ref{sect:good}
establish condition~$\CH_k$ (Proposition~\ref{prop:CH}), and 
Section~\ref{sect:verbal} treats the graded pieces of verbal quotients 
(Theorem~\ref{thm:graded}).  Section~\ref{sect:isogeny} deals with isogenies, 
finite covers, and $\Sigma$-invariants, and proves Theorems~\ref{thm:loci}
and~\ref{thm:sigma}.  Section~\ref{sect:arrangements} applies these results
to arrangements and proves Theorems~\ref{thm:arr} and~\ref{thm:ryb-pair},
the latter in \S\ref{subsec:rybnikov}.

\section{From homology of covers to jump loci}
\label{sect:engine}

This section states the general form of Theorem~\ref{thm:loci}, for spaces
satisfying Property $\mathsf{TT}_k$ and isomorphisms of completions
satisfying condition $\CH_k$ (\S\ref{subsec:transport}), and proves its
parts~\ref{main-depth} and~\ref{main-loci}; part~\ref{main-isogeny} is
proved in Section~\ref{sect:isogeny}.  In contrast with Sections~\ref{sect:degree-one} 
and~\ref{sect:good}, which supply condition $\CH_k$, nothing here treats degree 
one differently from the higher degrees.

\subsection{Property \texorpdfstring{$\mathsf{TT}_k$}{TTk} and condition
\texorpdfstring{$\CH_k$}{CHk}}
\label{subsec:transport}

The first property concerns the shape of the loci.

\medskip
\noindent
\textbf{Property $\mathsf{TT}_k$.}
\emph{For all $q \le k$ and $s \ge 1$, the variety $\VV^q_s(X,\C)$ is a
finite union of sets of the form $\rho T$, with
$T \subseteq \Hom(G,\C^\times)$ an algebraic subtorus and $\rho$ a torsion
character.}
\medskip

Property $\mathsf{TT}_k$ holds for every $k$ when $X$ is a smooth
quasi-projective variety.  In degree one, Arapura \cite{Arapura-jag97}
showed that the positive-dimensional components of $\VV^1_1(X,\C)$ are
torsion-translated subtori, and that the isolated points are unitary
characters; Artal, Cogolludo, and Matei
\cite{Artal-Cogolludo-Matei-gt13} showed that the isolated points are
torsion as well, in all depths.  In all degrees and depths the result is
due to Budur and Wang \cite{Budur-Wang-ens15}, sharpening Arapura's theorem
that, when $X$ is compact or has a smooth compactification $\bar X$ with
$b_1(\bar X) = 0$, the components are translated by unitary characters.
In particular it holds for arrangement complements, where the degree-one 
case is \cite[Thm.~5.4]{Denham-Suciu-plms14}.  It also holds for every $k$ 
when $X$ is a compact K\"ahler manifold (Remark~\ref{rem:cover-hyps}\ref{cov-b}).

The second condition, on the isomorphism of completions rather than on the
spaces, carries all the degree-sensitivity of the problem.  Stating it requires
transporting finite quotients along an isomorphism of completions.

For a group $G$ we write $\alpha_G \colon G \to \wG$ for the
canonical morphism to its profinite completion, abbreviated to $\alpha$
when $G$ is clear from the context, and to $\alpha_i$ when $G = G_i$.  
Let $\phi \colon \widehat{G_1} \to \widehat{G_2}$ be an isomorphism of
profinite groups, and let $F$ be a finite group.  By the universal
property of the profinite completion, every homomorphism
$\lambda \colon G_2 \to F$ extends uniquely to a continuous homomorphism
$\widehat\lambda \colon \widehat{G_2} \to F$ with
$\widehat\lambda \circ \alpha_2 = \lambda$.  We may therefore \emph{define}
\begin{equation}
\label{eq:transport}
\lambda^\phi \coloneqq \widehat\lambda \circ \phi \circ \alpha_1
\colon G_1 \longrightarrow F,
\end{equation}
that is, the composite around the outside of the diagram
\[
\begin{tikzcd}[column sep=22pt, row sep=18pt]
G_1 \ar[r, "\alpha_1"] \ar[drr, dashed, "\lambda^\phi"'] 
& \widehat{G_1}  \ar[r, "\phi"]& \widehat{G_2}\phantom{.}
 \ar[d, "\widehat{\lambda}"] \\
& &  F .
\end{tikzcd}
\]

\begin{lemma}
\label{lem:transport}
Let $F$ be a finite group.  Then
$\lambda \mapsto \lambda^\phi$ is a bijection
$\Hom(G_2,F) \to \Hom(G_1,F)$, with inverse
$\nu \mapsto \nu^{\phi^{-1}}$.  It carries epimorphisms to epimorphisms.
\end{lemma}

\begin{proof}[Proof of Lemma~\ref{lem:transport}]
Uniqueness of continuous extensions gives
$\widehat{\lambda^\phi} = \widehat\lambda \circ \phi$, whence
$(\lambda^\phi)^{\phi^{-1}} = \widehat\lambda \circ \phi \circ \phi^{-1}
\circ \alpha_2 = \lambda$, and symmetrically; so the two assignments are
mutually inverse.  If $\lambda$ is onto then so is $\widehat\lambda$, hence
so is $\widehat\lambda \circ \phi$; its kernel is open, and
$\alpha_1(G_1)$ is dense in $\widehat{G_1}$, so
$\alpha_1(G_1) \cdot \ker(\widehat\lambda \circ \phi) = \widehat{G_1}$ and
$\lambda^\phi$ is onto.
\end{proof}

Applying \eqref{eq:transport} with $F = \Z/n$, the cyclic group of order
$n$, transports cyclic quotients,
and hence covers: for $\lambda \colon G_2 \surj \Z/n$ we
write $X_2^{\lambda} \to X_2$ and $X_1^{\lambda^\phi} \to X_1$ for the
associated regular cyclic covers.

\begin{definition}
\label{def:CH}
Let $k \ge 1$.  We say that $\phi \colon \widehat{G_1} \isom \widehat{G_2}$ 
\emph{matches homology of covers through degree $k$}, and write 
$\CH_k(\phi)$, if
\[
\dim_\C H_q\bigl(X_1^{\lambda^\phi};\C\bigr) = 
\dim_\C H_q\bigl(X_2^{\lambda};\C\bigr)
\]
for every $q \le k$, every $n \ge 1$, and every epimorphism
$\lambda \colon G_2 \surj \Z/n$.
\end{definition}

Taking $n=1$ shows that $\CH_k(\phi)$ includes the equality of the
ordinary Betti numbers $b_q(X_1) = b_q(X_2)$ for $q \le k$.

The same operator \eqref{eq:transport} transports torsion characters: a
character $\rho \colon G_2 \to \C^\times$ of finite order takes values in a
finite cyclic subgroup of the group $\mu_\infty$ of roots of unity, so
$\rho^\phi$ is defined by the same formula.
By Lemma~\ref{lem:transport}, the assignment $\rho \mapsto \rho^\phi$ is a 
bijection between the torsion characters of $G_2$ and those of $G_1$, preserving
order.  This is the map that Theorem~\ref{thm:main} shows to preserve depth.

On torsion characters the operator $\rho \mapsto \rho^\phi$ is canonical.
To compare the jump loci themselves, which live in two different character
groups, one needs in addition an identification of those groups, and this is a
genuine extra datum: $\phi$ induces an isomorphism
$\phi_{\ab} \colon (G_1)_{\ab} \otimes \widehat\Z \isom
(G_2)_{\ab} \otimes \widehat\Z$, which need not carry $(G_1)_{\ab}$ onto
$(G_2)_{\ab}$.

\begin{definition}
\label{def:rigid}
An isomorphism $\phi \colon \widehat{G_1} \isom \widehat{G_2}$ is
\emph{rigidified} by an isomorphism
$\psi \colon (G_1)_{\ab} \to (G_2)_{\ab}$ if
$\phi_{\ab} = u\cdot(\psi \otimes \widehat\Z)$ for some unit
$u \in \widehat\Z^{\times}$.
\end{definition}

In that case, $\psi$ induces an isomorphism of algebraic groups
$\psi^\vee \colon \Hom(G_2,\C^\times) \to \Hom(G_1,\C^\times)$, and a short
computation gives
\begin{equation}
\label{eq:rigid}
\rho^\phi = (\psi^\vee\rho)^{u}
\end{equation}
for every torsion character $\rho$ of $G_2$, where $\tau^u = \tau^{u'}$ for
a character $\tau$ of finite order $m$ and any integer $u' \equiv u \bmod m$.
The unit cannot be dispensed with: for $X = \C^\times$ and any
$\sigma \in \Aut(\C)$ whose cyclotomic character $\chi(\sigma)$ is not
$\pm 1$, the automorphism of $\widehat\Z$ induced by $\sigma$ is
multiplication by $\chi(\sigma)$, which is not of the form
$\psi \otimes \widehat\Z$; and for Galois-conjugate arrangements the
isomorphism of completions is rigidified by $H \mapsto H^\sigma$, with
$u = \chi(\sigma)$ (Corollary~\ref{cor:galois}).

We can now state the general form of Theorem~\ref{thm:loci}.

\begin{theorem}
\label{thm:main}
Let $X_1,X_2$ be connected CW-complexes with finite $k$-skeleton,
for some $k\ge 1$, both satisfying Property $\mathsf{TT}_k$, and set
$G_i = \pi_1(X_i)$.  Let $\phi \colon \widehat{G_1} \to \widehat{G_2}$ be
an isomorphism satisfying $\CH_k(\phi)$.
\begin{enumerate}[label=\textup{(\roman*)}]
\item \label{main-depth}
For every torsion character $\rho$ of $G_2$ and every $q \le k$,
\[
\depth^q_\C\bigl(\rho^\phi\bigr) = \depth^q_\C(\rho) .
\]
\item \label{main-loci}
If moreover $\phi$ is rigidified by
$\psi \colon (G_1)_{\ab} \to (G_2)_{\ab}$, in the sense of 
Definition~\textup{\ref{def:rigid}}, then
\[
\psi^\vee\bigl(\VV^q_s(X_2,\C)\bigr) = \VV^q_s(X_1,\C)
\qquad \text{for all } q \le k \text{ and } s \ge 1 .
\]
\item \label{main-isogeny}
In general, for every integer $m \ge 1$ there is an injective homomorphism
$L \colon (G_1)_{\ab} \to (G_2)_{\ab}$ with finite cokernel and
$L \otimes \Z/m = \phi_{\ab} \otimes \Z/m$, such that the homomorphism
$L^\vee \colon \Hom(G_2,\C^\times) \to \Hom(G_1,\C^\times)$, which is
surjective with finite kernel, maps each irreducible component of
$\VV^q_s(X_2,\C)$ onto an irreducible component of $\VV^q_s(X_1,\C)$ of
the same dimension, inducing a bijection between the sets of components,
for all $q \le k$ and $s \ge 1$.  In particular,
$L^\vee$ carries $\VV^q_s(X_2,\C)$ onto $\VV^q_s(X_1,\C)$, and the dual
isomorphism $L^\vee_\R$, from $H^1(X_2;\R)$ to $H^1(X_1;\R)$, carries
$\Trop(\VV^q_s(X_2,\C))$ onto $\Trop(\VV^q_s(X_1,\C))$.
\end{enumerate}
\end{theorem}

Part~\ref{main-depth} uses only condition $\CH_k$, and is in fact
equivalent to it (Lemma~\ref{lem:depth}); Property $\mathsf{TT}_k$ enters in
parts~\ref{main-loci} and~\ref{main-isogeny}, to pass from torsion points to
loci.  Theorem~\ref{thm:loci} is deduced in \S\ref{subsec:covers} from
Theorem~\ref{thm:main}, Proposition~\ref{prop:CH}, and their extension to
finite covers.  Goodness is not the only possible entry point: any argument 
establishing $\CH_k$, for a restricted class of spaces or in a single degree 
$q \ge 2$ (Lemma~\ref{lem:depth} proceeds one degree at a time), feeds into 
Theorem~\ref{thm:main} unchanged.

\subsection{The covering formula, and its inversion}
\label{subsec:inversion}

The input from the covering side is a standard formula expressing the
Betti numbers of a finite regular cyclic cover through the depths of characters.
Formulas of this kind go back to work of Libgober, E.~Hironaka, and
Sakuma in the early 1990s; see \cite{Hironaka-fourier97} and the
references therein, and \cite[Thm.~2.5]{Denham-Suciu-plms14} for the
version in all degrees.

Let $\lambda \colon G \surj \Z/n$ be an epimorphism, let
$X^\lambda \to X$ be the corresponding cover, and let $\k$ be an
algebraically closed field with $\ch\k \nmid n$.  By
Shapiro's lemma, $H_q(X^\lambda;\k) \cong H_q(X;\k[\Z/n])$, where $G$ acts
on the group algebra through $\lambda$; by Maschke's theorem,
$\k[\Z/n]$ splits as the direct sum of the rank-one modules $\k_\chi$, one
for each character $\chi \colon \Z/n \to \k^\times$.  Hence, for $q \le k$,
\begin{equation}
\label{eq:cover-all}
\dim_\k H_q(X^\lambda;\k) = 
\sum_{\chi \in \Hom(\Z/n,\k^\times)} \depth^q_\k(\chi\circ\lambda) .
\end{equation}

If $\ch\k = 0$, the terms of \eqref{eq:cover-all} can be grouped.  Fix an
injection $\iota \colon \Z/n \inj \k^\times$ and put
$\rho = \iota \circ \lambda$.  For each divisor $j$ of $n$, the characters
$\chi\circ\lambda$ of order exactly $j$ are the powers $\rho^{an/j}$ with
$\gcd(a,j) = 1$; there are $\varphi(j)$ of them, $\varphi$ being the Euler
totient function, and they are Galois conjugate.  Since the loci
$\VV^q_s$ are defined over $\Q$, they all have the depth of
$\rho^{n/j}$, and \eqref{eq:cover-all} becomes
\begin{equation}
\label{eq:cover}
\dim_\k H_q(X^\lambda;\k) = \dim_\k H_q(X;\k) + 
\sum_{1 \ne j \mid n} \varphi(j)\, \depth^q_\k\bigl(\rho^{n/j}\bigr),
\end{equation}
the term $j = 1$ being $\depth^q_\k(\one) = \dim_\k H_q(X;\k)$.

\begin{lemma}
\label{lem:depth}
Let $X_1, X_2$ be connected CW-complexes with finite $k$-skeleton, set
$G_i = \pi_1(X_i)$, and let $\phi \colon \widehat{G_1} \to \widehat{G_2}$
be an isomorphism.  Then $\CH_k(\phi)$ holds if and only if
$\depth^q_\C(\rho^\phi) = \depth^q_\C(\rho)$ for every $q \le k$ and
every torsion character $\rho$ of $G_2$, where $\rho^\phi$ is as in
\eqref{eq:transport}.  
\end{lemma}

\begin{proof}
Suppose first that $\CH_k(\phi)$ holds.
Fix $q \le k$ and let $\rho$ be a torsion character of $G_2$ of order
exactly $m$.  Fixing an isomorphism $\iota$ of $\Z/m$ onto the image of
$\rho$, we have $\rho = \iota \circ \lambda$ for a unique
epimorphism $\lambda \colon G_2 \surj \Z/m$.  For each 
divisor $j \mid m$ the character $\rho^{m/j}$ has order
exactly $j$, and arises in the same way from the induced epimorphism
$\lambda_j \colon G_2 \surj \Z/m \surj \Z/j$. Set
\begin{equation}
\label{eq:fqj}
f_q(j) \coloneqq \dim_\C H_q\bigl(X_2^{\lambda_j};\C\bigr)
- \dim_\C H_q(X_2;\C) .
\end{equation}
Formula~\eqref{eq:cover}, applied to $\lambda_j$ with $\k = \C$, gives
$f_q(j) = \sum_{1 \ne i \mid j} \varphi(i)\,\depth^q_\C(\rho^{m/i})$, a
relation that can be inverted directly, without induction.  For $i \mid m$,
put $D_q(i) = \varphi(i)\,\depth^q_\C(\rho^{m/i})$ if $i > 1$, and
$D_q(1) = 0$; note also that $f_q(1) = 0$, since $\lambda_1$ is trivial.
The relation then reads $f_q(j) = \sum_{i \mid j} D_q(i)$ for every
$j \mid m$, the sum now running over \emph{all} divisors of $j$, so
M\"obius inversion gives $D_q(j) = \sum_{i \mid j} \mu(j/i)\, f_q(i)$,
that is,
\begin{equation}
\label{eq:mobius}
\varphi(j)\,\depth^q_\C(\rho^{m/j}) = \sum_{i \mid j} \mu(j/i)\, f_q(i)
\qquad (1 < j \mid m),
\end{equation}
$\mu$ being the classical M\"obius function.  If $m > 1$, taking $j = m$
and dividing by $\varphi(m) \ne 0$ expresses $\depth^q_\C(\rho)$ as an
explicit function of the quantities $f_q(i)$, for $i \mid m$.  If $m = 1$,
then $\rho = \one$ and $\depth^q_\C(\rho) = \dim_\C H_q(X_2;\C)$.

Now apply the same computation to $\rho^\phi = \iota \circ \lambda^\phi$.
By Lemma~\ref{lem:transport}, $\lambda^\phi$ is again an epimorphism, so
$\rho^\phi$ again has order exactly $m$; and since
$\widehat{\lambda^\phi} = \widehat\lambda \circ \phi$, the operator
\eqref{eq:transport} commutes with passing to quotients of $\Z/m$, so
that $(\lambda^\phi)_j = (\lambda_j)^\phi$ for every $j \mid m$.  The
quantities $f_q(i)$ computed for $\rho^\phi$ in $X_1$ are therefore
$\dim_\C H_q(X_1^{(\lambda_i)^\phi};\C) - \dim_\C H_q(X_1;\C)$, and by
$\CH_k(\phi)$, applied with $\lambda = \lambda_i$ and with $n = 1$, they
coincide with those computed for $\rho$ in $X_2$.  If $m > 1$, both
depths are then the same function \eqref{eq:mobius} of the same
quantities; if $m = 1$, both are Betti numbers, matched by $\CH_k(\phi)$
with $n = 1$.

Conversely, suppose that the depths match, and let
$\lambda \colon G_2 \surj \Z/n$ be an epimorphism.  For a character
$\chi \colon \Z/n \to \C^\times$, the continuous extension of
$\chi \circ \lambda$ is $\chi \circ \widehat\lambda$, so
$(\chi \circ \lambda)^\phi = \chi \circ \widehat\lambda \circ \phi
\circ \alpha_1 = \chi \circ \lambda^\phi$.  Applying
\eqref{eq:cover-all} with $\k = \C$ to $\lambda$ and to $\lambda^\phi$,
which is again an epimorphism by Lemma~\ref{lem:transport}, we obtain
\[
\dim_\C H_q\bigl(X_1^{\lambda^\phi};\C\bigr)
= \sum_{\chi} \depth^q_\C\bigl((\chi\circ\lambda)^\phi\bigr)
= \sum_{\chi} \depth^q_\C(\chi\circ\lambda)
= \dim_\C H_q\bigl(X_2^{\lambda};\C\bigr)
\]
for every $q \le k$, the sums running over $\Hom(\Z/n,\C^\times)$.  This
is $\CH_k(\phi)$.
\end{proof}

\begin{remark}
\label{rem:char-p}
The restriction to
\emph{characteristic zero} enters through the covering formula
\eqref{eq:cover}: over $\C$ the characters of a given order factoring
through a cyclic quotient are Galois conjugate, hence of equal depth, and
this is what makes the inversion in Lemma~\ref{lem:depth} possible.  
Over $\overline{\F}_p$, the characters of a given order split into
several Frobenius orbits, whose depths may differ, so the homology of
the covers determines only the total depth of the characters of each
order, not the depth of an individual character; the density step, by
contrast, becomes vacuous there, every $\overline{\F}_p$-point of the
character group being torsion.
\end{remark}

\subsection{Density}
\label{subsec:density}

What remains is an elementary observation about torsion points, which is
where Property $\mathsf{TT}_k$ enters.

\begin{lemma}
\label{lem:density}
Let $T \subseteq \Hom(G,\C^\times)$ be an algebraic subtorus and $\rho$ a
torsion character of $G$.  Then the torsion points of $\rho T$ are Zariski dense
in $\rho T$.
\end{lemma}

\begin{proof}
The torsion points of $T$ are dense in $T$, since $T$ is isomorphic to a
product of copies of $\C^\times$ and the roots of unity are dense in
$\C^\times$.  Multiplication by $\rho$ is an isomorphism of varieties
$T \to \rho T$ carrying torsion points to torsion points, because $\rho$
itself is torsion.
\end{proof}

\begin{proof}[Proof of parts~\textup{\ref{main-depth}} and~\textup{\ref{main-loci}} 
of Theorem~\textup{\ref{thm:main}}]
Part~\ref{main-depth} is Lemma~\ref{lem:depth}.  For part~\ref{main-loci}, fix
$q \le k$ and $s \ge 1$, and set
\[
\Theta^q_s(X_i) \coloneqq \{\, \rho \ \text{a torsion character of } G_i :
\depth^q_\C(\rho) \ge s \,\} .
\]
By Property $\mathsf{TT}_k$ the variety $\VV^q_s(X_2,\C)$ is a finite
union of torsion translates $\rho T$ of subtori, so by Lemma~\ref{lem:density} 
it is the Zariski closure of $\Theta^q_s(X_2)$; and
likewise for $X_1$.  By part~\ref{main-depth}, the bijection
$\rho \mapsto \rho^\phi$ of Lemma~\ref{lem:transport} carries
$\Theta^q_s(X_2)$ onto $\Theta^q_s(X_1)$.  Since $\phi$ is rigidified by
$\psi$, \eqref{eq:rigid} gives $\psi^\vee\rho = (\rho^\phi)^{u^{-1}}$ for
torsion characters $\rho$.  As the cyclotomic character
$\Gal(\overline\Q/\Q) \to \widehat\Z^\times$ is surjective,
$\tau \mapsto \tau^{u^{-1}}$ is the action of a Galois automorphism on
torsion characters, and $\Theta^q_s(X_1)$, the set of torsion points of a
locus defined over $\Q$, is stable under it (\S\ref{subsec:jumploci}).  So
$\psi^\vee$ carries $\Theta^q_s(X_2)$ onto $\Theta^q_s(X_1)$, and, being an
isomorphism of algebraic groups, it carries Zariski closures to Zariski
closures: $\psi^\vee(\VV^q_s(X_2,\C)) = \VV^q_s(X_1,\C)$.
\end{proof}

\begin{remark}
\label{rem:knots}
Property $\mathsf{TT}_k$ is what allows the torsion points, hence the
homology of finite cyclic covers, to determine the loci; it is not
necessary for their profinite invariance in every case.  Let $X$ be the
complement of a knot $K \subset S^3$, so that $G_{\ab} = \Z$ and
$\VV^1_1(X,\C) = \{\one\} \cup \{t \in \C^\times : \Delta_K(t) = 0\}$,
where $\Delta_K$ is the Alexander polynomial.  The roots of $\Delta_K$ need
not be roots of unity --- for the figure-eight knot,
$\Delta_K(t) = t^2 - 3t + 1$ --- so Property $\mathsf{TT}_1$ fails in
general.  Nevertheless, $\Delta_K$, and with it $\VV^1_1(X,\C)$, is
determined by the profinite completion of the knot group: by Boileau and
Friedl \cite[Prop.~4.10]{Boileau-Friedl-amstud20} when no root of
$\Delta_K$ is a root of unity, and by Ueki \cite{Ueki-agt18} in general.
The first argument uses, besides the ranks of the groups $H_1$ of the
finite cyclic covers of $X$, which detect only the torsion points of
$\VV^1_1(X,\C)$, the orders of their torsion subgroups, which $\wG$
determines as well (Lemma~\ref{lem:fi}\ref{fi-det}).
\end{remark}

This suggests the following question, in degree one, where the loci
depend only on the group.

\begin{question}
\label{quest:no-T}
Let $X_1, X_2$ be connected CW-complexes with finite $1$-skeleton, such
that $G_i = \pi_1(X_i)$ have isomorphic profinite completions.  Is there an
isomorphism $\psi \colon (G_1)_{\ab} \to (G_2)_{\ab}$ with
$\psi^\vee(\VV^1_s(X_2,\C)) = \VV^1_s(X_1,\C)$ for all $s \ge 1$?
\end{question}

For knot complements and $s = 1$, the answer is positive, with
$\psi = \id$, by Remark~\ref{rem:knots}.  Under Property
$\mathsf{TT}_1$, Theorem~\ref{thm:main} gives such a $\psi$ when the
isomorphism of completions is rigidified, and an isogeny in general
(Remark~\ref{rem:isogeny-scope}\ref{isog-b}).

\section{Homology of covers in degree one}
\label{sect:degree-one}

Condition $\CH_k$ is supplied by the following proposition.
Recall from Serre \cite[Ch.~I, \S2.6]{Serre-97} that a group $G$ is
\emph{good} if $\alpha_G^{*} \colon H^{*}(\wG; M) \to H^{*}(G; M)$ is 
an isomorphism for every $G$-module $M$ that is finite as a set.

\begin{proposition}
\label{prop:CH}
Let $\phi \colon \widehat{G_1} \to \widehat{G_2}$ be an isomorphism, 
where the groups $G_i = \pi_1(X_i)$ are finitely generated.
\begin{enumerate}[label=\textup{(\roman*)}, itemsep=1.5pt]
\item \label{CH-one}
Condition $\CH_1(\phi)$ holds, always.
\item \label{CH-all}
If both $X_i$ are aspherical and of finite type, and both $G_i$ are good,
then $\CH_k(\phi)$ holds for every $k$.
\end{enumerate}
\end{proposition}

Part~\ref{CH-one} is well known; this section gives the short argument,
since Lemma~\ref{lem:fi} is used repeatedly below, and Section~\ref{sect:good} 
proves part~\ref{CH-all}.  Throughout,
$\alpha = \alpha_G \colon G \to \wG$ is the canonical morphism, and the
abelianization of a profinite group $P$ is $P_{\ab} = P/\overline{[P,P]}$,
its maximal abelian Hausdorff quotient.

\begin{lemma}
\label{lem:fi}
Let $G$ be a finitely generated group.
\begin{enumerate}[label=\textup{(\roman*)}, itemsep=1.5pt]
\item \label{fi-corr}
The assignment $U \mapsto \alpha^{-1}(U)$ is an index-preserving bijection
from the open subgroups of $\wG$ to the finite-index subgroups of
$G$, with inverse $N \mapsto \overline{\alpha(N)}$; it carries normal
subgroups to normal subgroups.
\item \label{fi-ab}
If $N \le G$ has finite index, then $\overline{\alpha(N)} \cong \widehat N$,
and consequently
$\overline{\alpha(N)}_{\ab} \cong N_{\ab} \otimes_\Z \widehat\Z$.
\item \label{fi-det}
An isomorphism $\phi \colon \widehat{G_1} \to \widehat{G_2}$ induces an
index-preserving bijection between the finite-index subgroups of $G_1$ and
those of $G_2$, and isomorphisms of abelian groups
$(N_2)_{\ab} \cong (N_1)_{\ab}$ for corresponding subgroups.
\end{enumerate}
\end{lemma}

\begin{proof}
\ref{fi-corr} is standard.  An open subgroup $U \le \wG$ is closed,
and $\alpha(G)$ is dense, so $\alpha(G)U = \wG$ and hence
$[G : \alpha^{-1}(U)] = [\wG : U]$.  Conversely, a finite-index
subgroup $N \le G$ is open in the profinite topology of $G$, and
$\alpha^{-1}(\overline{\alpha(N)}) = N$.

For \ref{fi-ab}, the point is that the topology induced on $N$ from the
profinite topology of $G$ coincides with the profinite topology of $N$.
Indeed, if $M \le N$ has finite index then $[G:M] = [G:N][N:M] < \infty$,
so $M$ is open in $G$; the other inclusion is trivial.  The two uniform
structures therefore agree and the completions coincide.  Finally, $N$ is
finitely generated (being of finite index in a finitely generated group),
so $\widehat N_{\ab} = \widehat{N_{\ab}} = N_{\ab} \otimes \widehat\Z$.

\ref{fi-det} follows, once one notes that a finitely generated abelian
group $A$ is recovered from $A \otimes \widehat\Z$: writing
$A \cong \Z^r \oplus T$ with $T$ finite, one has
$A \otimes \widehat\Z \cong \widehat\Z^{\,r} \oplus T$, and both $r$ and
$T$ are visible on the right.
\end{proof}

\begin{proof}[Proof of Proposition~\textup{\ref{prop:CH}}\ref{CH-one}]
Let $\lambda_2 \colon G_2 \surj \Z/n$ be an epimorphism, and
$\lambda_1 = \lambda_2^\phi$; write $N_i = \ker \lambda_i$.  Since
$\ker \widehat{\lambda_2} = \overline{\alpha_2(N_2)}$ and
$N_1 = \alpha_1^{-1}\bigl(\phi^{-1}(\ker \widehat{\lambda_2})\bigr)$, the
subgroups $N_1$ and $N_2$ correspond under the bijection of 
Lemma~\ref{lem:fi}\ref{fi-det}.  Since $\pi_1(X_i^{\lambda_i}) = N_i$, we have
$H_1(X_i^{\lambda_i};\Z) = (N_i)_{\ab}$; no asphericity is needed, as
homology in degree one depends only on the fundamental group.  
By Lemma~\ref{lem:fi}\ref{fi-det}, these abelian groups are isomorphic, 
so their ranks agree, which is the assertion.
\end{proof}

\section{Homology of covers in higher degrees: goodness}
\label{sect:good}

Proposition~\ref{prop:CH}\ref{CH-one} is confined to degree one because 
Lemma~\ref{lem:fi} treats only $H_1$ of finite-index subgroups.  In higher
degrees two further inputs are needed: goodness, so that the completion
determines the cohomology of those subgroups, and asphericity, so that this
cohomology is that of the corresponding covers.

\begin{proof}[Proof of Proposition~\textup{\ref{prop:CH}}\ref{CH-all}]
Let $\lambda \colon G \surj \Z/n$ be an epimorphism with kernel $N$, let
$X^\lambda \to X$ be the corresponding cover, and fix a prime $p$.  The
module $M = \F_p[\Z/n]$, on which $G$ acts through $\lambda$, is finite,
and $M = \Ind_N^G \F_p$ as a $G$-module.  Hence
\begin{equation}
\label{eq:good-chain}
H^q(X^\lambda;\F_p) \cong H^q(N;\F_p) \cong H^q(G;M) \cong
H^q(\wG;M) ,
\end{equation}
by asphericity of $X$, which makes $X^\lambda$ a $K(N,1)$, by Shapiro's
lemma, and by goodness of $G$, respectively; here $\wG$ acts on $M$
through $\widehat\lambda$.

Now apply \eqref{eq:good-chain} to $\lambda \colon G_2 \surj \Z/n$ and to
$\lambda^\phi \colon G_1 \surj \Z/n$ (Lemma~\ref{lem:transport}).  The
$\widehat{G_1}$-module structure on $M$, through
$\widehat{\lambda^\phi} = \widehat\lambda\circ\phi$, is the pullback along
$\phi$ of the $\widehat{G_2}$-module structure, through $\widehat\lambda$,
so $\phi$ induces an isomorphism
$H^q(\widehat{G_2}; M) \cong H^q(\widehat{G_1}; M)$.  Therefore
\[
H^q(X_2^{\lambda};\F_p) \cong H^q(X_1^{\lambda^\phi};\F_p)
\qquad\text{for every prime } p \text{ and every } q .
\]

It remains to pass from $\F_p$ to $\C$.  By the universal coefficient
theorem,
\[
\dim_{\F_p} H^q(X^\lambda;\F_p) = \rank H_q(X^\lambda;\Z) + t_p^q +
t_p^{q-1} ,
\]
where $t_p^{j}$ is the number of cyclic summands of $p$-power order in
$H_j(X^\lambda;\Z)$.  Since $X^\lambda$ is of finite type, these numbers
vanish for all but finitely many $p$, and so the rank of
$H_q(X^\lambda;\Z)$ is the minimum over all primes $p$ of
$\dim_{\F_p} H^q(X^\lambda;\F_p)$.  Hence
$\dim_\C H_q(X_1^{\lambda^\phi};\C) = \dim_\C H_q(X_2^{\lambda};\C)$ for
every $q$, which is $\CH_k(\phi)$ for every $k$.
\end{proof}

The argument is the standard way in which goodness transfers cohomological
information from $\wG$ to $G$ and its finite-index subgroups; compare
\cite{Grunewald-JaikinZapirain-Zalesskii-duke08} and \cite[\S7]{Reid-lms15}.

Both hypotheses hold for fiber-type arrangements.  These are aspherical by
Falk--Randell \cite{Falk-Randell-inv85}, and their groups, iterated
semidirect products of finitely generated free groups, are good.  Indeed,
finitely generated free groups are good
\cite[Prop.~3.7]{Grunewald-JaikinZapirain-Zalesskii-duke08}, and goodness
passes to an extension $1 \to H \to G \to Q \to 1$ with $H$ and $Q$ good,
provided $H$ is finitely generated and $H^q(H;M)$ is finite for all $q$
and all finite $G$-modules $M$ \cite[Ch.~I, \S2.6,
Ex.~2(c)]{Serre-97}, \cite[Lem.~3.3]{Grunewald-JaikinZapirain-Zalesskii-duke08};
the finiteness condition holds whenever $H$ is of type $FP_\infty$.  For
the pure braid groups, this is
\cite[Prop.~3.5]{Grunewald-JaikinZapirain-Zalesskii-duke08}.  Hence:

\begin{corollary}
\label{cor:fibertype}
Let $\A_1$ and $\A_2$ be fiber-type arrangements, and let
$\phi \colon \widehat{G(\A_1)} \to \widehat{G(\A_2)}$ be an isomorphism.
Then Theorem~\textup{\ref{thm:main}} applies to $\phi$ in every degree: for all
$q$ and $s$, the loci $\VV^q_s(\A_1,\C)$ and $\VV^q_s(\A_2,\C)$ correspond
under an isogeny as in part~\textup{\ref{main-isogeny}}, and under $\psi^\vee$
if $\phi$ is rigidified by $\psi$.
\end{corollary}

Fiber-type is far from the largest class where goodness is known.
Marin \cite[Prop.~1]{Marin-aspm12} proves goodness of the pure braid group
$P$ and braid group $B$ of $W$, for \emph{every} irreducible complex
reflection group $W$ in the Shephard--Todd classification except
possibly twelve exceptional cases ($G_{23}$, $G_{24}$, $G_{27}$--$G_{31}$,
$G_{33}$--$G_{37}$).  For $W = G(de,e,n)$ with $d \ne 1$ the reflection
arrangement is fiber-type; for $d = 1$ (in particular, for $W$ of type
$D_n$, whose reflection arrangement is not fiber-type for $n \ge 4$), Marin
exhibits a fibration over a fiber-type base whose fiber is an affine
curve, so that the extension criterion above applies, with $H$ free of
finite rank.  All rank-$2$ complex reflection arrangements are fiber-type, 
disposing of $G_4$--$G_{22}$. Three more of the
exceptional cases share a braid group with an infinite-family member
($G_{25}$ with $G(1,1,4)$, $G_{32}$ with $G(1,1,5)$, $G_{26}$ with
$G(2,1,3)$, the last misprinted as $G(2,2,3)$ in \cite{Marin-aspm12}),
leaving the twelve above open (Marin's Conjecture~1).
Since complex reflection arrangement complements are aspherical
\cite{Bessis-ann15}, Corollary~\ref{cor:fibertype} holds verbatim
for any two arrangements each of which is either fiber-type or the
reflection arrangement of one of these groups.

For general arrangements, the question is open.

\begin{question}[{\cite[Ch.~5]{CDFSSTY-pre09}}]
\label{quest:good}
Are arrangement groups good?
\end{question}

Simplicial arrangements show that asphericity can be available where
goodness is not known.  By a theorem of Deligne \cite{Deligne-inv72}, the
complement of a simplicial arrangement is aspherical; this class is large,
and properly contains the reflection arrangements of finite Coxeter
groups.  On that Coxeter subclass, goodness is known in types $A_n$,
$B_n$, and $I_2(m)$, whose reflection arrangements are fiber-type, and in
type $D_n$, by Marin; for these, Proposition~\ref{prop:CH}\ref{CH-all} applies
unconditionally.  The remaining Coxeter groups, of the exceptional types
$H_3$, $F_4$, $H_4$, $E_6$, $E_7$, $E_8$, are the groups $G_{23}$,
$G_{28}$, $G_{30}$, $G_{35}$, $G_{36}$, $G_{37}$ of the Shephard--Todd
classification, all among Marin's twelve open cases.  Within the simplicial
class, the open test cases for Question~\ref{quest:good} are therefore
these six Coxeter arrangements and the \emph{non-Coxeter} simplicial
arrangements that are not fiber-type.

The natural companion of Question~\ref{quest:good} is residual finiteness,
although the results above do not require it (\S\ref{subsec:verbal-scope}).  
For smooth projective varieties, Toledo \cite{Toledo-ihes93} gave the first 
examples of fundamental groups that are not residually finite, answering a 
question of Serre.  For arrangements, residual finiteness holds whenever
the group is residually torsion-free nilpotent, as for fiber-type
arrangements \cite{Falk-Randell-inv85, Falk-Randell-conm88}; 
Marin \cite[Prop.~2]{Marin-aspm12} establishes it for many complex 
reflection arrangements, by that mechanism when $W = G(de,e,n)$ 
with $d \ne 1$ or $W$ has rank $2$, and by \emph{linearity} when $W$ 
is a Coxeter or Shephard group.  In general the question is open; see
\cite[Problem~1.9]{Koberda-Suciu-ccm20} for affine line arrangements.

\begin{question}[\cite{CDFSSTY-pre09}]
\label{quest:res-finite}
Let $G$ be an arrangement group.  Is $G$ residually finite?
\end{question}

\section{Verbal quotients and their graded pieces}
\label{sect:verbal}

This section proves Theorem~\ref{thm:graded}.  It is independent of 
Sections~\ref{sect:engine}--\ref{sect:good} but for the elementary 
Lemma~\ref{lem:fi}\ref{fi-det}: no covers, no jump loci, and no 
hypothesis on the spaces.  The only external input is Mal'cev's 
theorem on the profinite topology of a finitely generated nilpotent group.

\subsection{Closed verbal series}
\label{subsec:closed-series}

Recall that a set $W$
of words (elements of a free group on $x_1,x_2,\dots$) determines, in
any group $G$, the \emph{verbal subgroup} $W(G)$ generated by the values
$w(g_1,\dots,g_n)$ with $w \in W$ and $g_i \in G$.  The two cases used
below are $W = \varnothing$, with $W(G) = \{1\}$, and
$W = \{[[x_1,x_2],[x_3,x_4]]\}$, with $W(G) = G''$, so that $G/W(G)$ is the
maximal metabelian quotient of $G$.  Verbal subgroups are
normal, indeed fully invariant, and are carried onto verbal subgroups by
epimorphisms.  The lower central series
is itself of this kind: $\gamma_r(G) = V_r(G)$ for
$V_r = \{[x_1,\dots,x_r]\}$ the left-normed commutator of weight $r$.  
For a topological group $P$ we write
\begin{equation}
\label{eq:closed-lcs}
\Delta^W_r(P) \coloneqq \overline{W(P)\,\gamma_r(P)}
= \overline{(W \cup V_r)(P)} ,
\end{equation}
a closed normal subgroup, so that $P/\Delta^W_r(P)$ is again profinite
when $P$ is; being defined by words and closure, $\Delta^W_r(P)$ is
carried onto its counterpart by every isomorphism of topological groups.  For 
$W = \varnothing$ this is the closed lower central
series $\Gamma_r(P) = \overline{\gamma_r(P)}$.

\begin{lemma}
\label{lem:topgp}
Let $P$ be a Hausdorff topological group and let $H \le P$ be a dense subgroup.
Then $\overline{V(H)} = \overline{V(P)}$ for every set of words $V$.
\end{lemma}

\begin{proof}
One inclusion is immediate from $V(H) \subseteq V(P)$.  For
the other, let $w \in V$ be a word in $n$ letters.  The map
$P^n \to P$ given by $(p_1,\dots,p_n) \mapsto w(p_1,\dots,p_n)$ is
continuous, 
and $P^n = \overline{H^n}$ in the product topology, so
$w(P^n) \subseteq \overline{w(H^n)} \subseteq \overline{V(H)}$.  The
right-hand side is a closed subgroup, being the closure of a subgroup, and
it therefore contains the subgroup generated by all such values, which is
$V(P)$; hence it contains $\overline{V(P)}$.
\end{proof}

\subsection{The graded pieces}
\label{subsec:verbal-proof}

\begin{theorem}
\label{thm:verbal}
Let $G$ be a finitely generated group, let $W$ be a set of words, and write
$P = \wG$, with canonical morphism $\alpha \colon G \to P$ and
$\Delta_r = \Delta^W_r(P)$ as in \eqref{eq:closed-lcs}.  Then
\[
\Delta_r/\Delta_{r+1} \cong \gr_r\bigl(G/W(G)\bigr) \otimes \widehat\Z
\qquad \text{for every } r \ge 1 .
\]
\end{theorem}

\begin{proof}[Proof of Theorem~\textup{\ref{thm:verbal}}]
Write $\overline G = G/W(G)$ and, for $r \ge 1$,
\[
N_r \coloneqq W(G)\,\gamma_r(G) = (W \cup V_r)(G) ,
\]
so that $N_1 = G$, while $\gamma_r(\overline G) = N_r/W(G)$ and
$\gr_r(\overline G) = N_r/N_{r+1}$.  Note that $[N_r,G] \subseteq N_{r+1}$,
since $W(G)$ is normal and $[\gamma_r(G),G] = \gamma_{r+1}(G)$.

\vspace*{1pt}
\emph{Step 1: $\Delta_r = \overline{\alpha(N_r)}$.}  Apply 
Lemma~\ref{lem:topgp} with $V = W \cup V_r$ and the dense subgroup
$H = \alpha(G)$, using $V(\alpha(G)) = \alpha(V(G)) = \alpha(N_r)$.

\vspace*{1pt}
\emph{Step 2: $P/\Delta_{r+1} \cong \widehat{Q_r}$, where
$Q_r = G/N_{r+1}$.}  The group $Q_r$ is finitely generated, nilpotent of
class $\le r$, and satisfies $W(Q_r) = 1$.  By Step 1, the composite
$G \to P \to P/\Delta_{r+1}$ kills $N_{r+1}$, so factors through $Q_r$,
with dense image in the profinite group $P/\Delta_{r+1}$.  It remains to
check the universal property of the profinite completion: let
$f \colon Q_r \to F$ with $F$ finite, and let $\widehat f \colon P \to F$ be
the continuous extension of the composite $G \to Q_r \to F$.  Then
$\widehat f(P) = \overline{f(Q_r)} = f(Q_r)$, as $F$ is finite; this is a
quotient of $Q_r$, hence nilpotent of class $\le r$ and killed by $W$.
Therefore
$\widehat f\bigl(W(P)\gamma_{r+1}(P)\bigr) =
W(\widehat f(P))\,\gamma_{r+1}(\widehat f(P)) = 1$, and since
$\ker \widehat f$ is closed, $\widehat f$ kills $\Delta_{r+1}$ and so
factors through $P/\Delta_{r+1}$.  Uniqueness of the factorization follows
from density.

\vspace*{1pt}
\emph{Step 3: the graded pieces.}  Fix $r \ge 1$ and put $Q = Q_r$ and
$A = N_r/N_{r+1} = \gr_r(\overline G)$, a subgroup of $Q$ which is central
by $[N_r,G] \subseteq N_{r+1}$, with $Q/A = Q_{r-1}$.  Step 2, applied at
$r$ and at $r-1$, gives $P/\Delta_{r+1} \cong \widehat Q$ and
$P/\Delta_{r} \cong \widehat{Q/A}$, compatibly with the projections by the
uniqueness in the universal property; hence
\[
\Delta_r/\Delta_{r+1} =
\ker\bigl(\widehat Q \longrightarrow \widehat{Q/A}\bigr) .
\]
Now $Q$ is finitely generated nilpotent, hence polycyclic, hence a solvable
minimax group; and $A$ is closed in the profinite topology of $Q$, since
$Q/A$ is finitely generated nilpotent and therefore residually finite.  By
Mal'cev's theorem in the form of \cite[Lem.~4.7.6]{Ribes-Zalesskii-10},
the profinite topology of $Q$ induces on $A$ its own full profinite
topology; equivalently
\cite[Lem.~3.2.6 and Prop.~3.2.5]{Ribes-Zalesskii-10}, the sequence
$1 \to \widehat A \to \widehat Q \to \widehat{Q/A} \to 1$ is exact.  The
kernel above is therefore $\widehat A = A \otimes \widehat\Z$, the group 
$A$ being finitely generated abelian.
\end{proof}

\begin{theorem}
\label{thm:graded}
Let $G_1,G_2$ be finitely generated groups with
$\widehat{G_1} \cong \widehat{G_2}$, and let $W$ be any set of words.
Then, for every $r \ge 1$,
\[
\gr_r\bigl(G_1/W(G_1)\bigr) \cong \gr_r\bigl(G_2/W(G_2)\bigr)
\]
as abelian groups.
\end{theorem}

\begin{proof}[Proof of Theorem~\textup{\ref{thm:graded}}]
An isomorphism $\phi \colon \widehat{G_1} \to \widehat{G_2}$ of profinite
groups carries $\Delta^W_r(\widehat{G_1})$ onto $\Delta^W_r(\widehat{G_2})$
for every $r$, these subgroups being intrinsic.
By Theorem~\ref{thm:verbal}, it follows that
\[
\gr_r\bigl(G_1/W(G_1)\bigr) \otimes \widehat\Z \cong
\gr_r\bigl(G_2/W(G_2)\bigr) \otimes \widehat\Z .
\]
Since $G_i$ is finitely generated, so is the abelian group
$\gr_r(G_i/W(G_i))$, which is therefore recovered from its tensor product
with $\widehat\Z$, as in Lemma~\ref{lem:fi}\ref{fi-det}.
\end{proof}

\begin{corollary}
\label{cor:lcs-profinite}
Let $G_1,G_2$ be finitely generated groups with
$\widehat{G_1} \cong \widehat{G_2}$.  Then $\gr_r(G_1) \cong \gr_r(G_2)$ for
every $r \ge 1$.  That is, $\wG$ determines the lower central series
quotients $\gr_r(G) = \gamma_r(G)/\gamma_{r+1}(G)$, torsion included.
\end{corollary}

\begin{proof}
Theorem~\ref{thm:graded} with $W = \varnothing$.
\end{proof}

For residually finite $G_i$, Corollary~\ref{cor:lcs-profinite} is
\cite[Prop.~2.4]{Nikolov-Wykowski-mathann26}.  Theorem~\ref{thm:graded}
dispenses with residual finiteness and with the Nikolov--Segal theorem used
there (\S\ref{subsec:verbal-scope}), and allows an arbitrary set of words
$W$; the Chen groups, treated next, are the second case used in 
Section~\ref{sect:arrangements}.

For the second case, recall that the \emph{Alexander invariant} of $G$ is the
abelian group $B(G) = G'/G''$, viewed as a module over
$\Lambda = \Z[G_{\ab}]$ by conjugation, and that the \emph{Chen groups} of
$G$ are the lower central series quotients of its maximal metabelian
quotient $G/G''$.  By a classical result of Massey, the two are the 
same data at the graded level: for $j \ge 0$ one has
$\gr_{j+2}(G/G'') \cong \m^{j} B(G)/\m^{j+1} B(G)$, where
$\m \subseteq \Lambda$ is the augmentation ideal;
see, e.g., \cite{Suciu-pisa24}. In particular, the truncations $B(G)/\m^{k}B(G)$ 
used in \cite{ACCM-aspm07,ACGM-racsam17} have the Chen groups 
as their graded pieces \cite[\S3]{ACGM-racsam17}.

\begin{corollary}
\label{cor:chen-profinite}
Let $G_1,G_2$ be finitely generated groups with
$\widehat{G_1} \cong \widehat{G_2}$.  Then
$\gr_r(G_1/G_1'') \cong \gr_r(G_2/G_2'')$ for every $r \ge 1$.  That is,
$\wG$ determines the Chen groups of $G$, torsion included, and with
them the graded pieces of every truncation of the Alexander invariant.
\end{corollary}

\begin{proof}
Theorem~\ref{thm:graded} with $W = \{[[x_1,x_2],[x_3,x_4]]\}$, for which
$W(G) = G''$.
\end{proof}

\subsection{What the argument does and does not use}
\label{subsec:verbal-scope}

The proof of Theorem~\ref{thm:verbal} works with the closed series
$\Delta^W_r$ rather than with $W(P)\gamma_r(P)$.  This makes it independent
of the theorem of Nikolov and Segal \cite{Nikolov-Segal-abb07}, which shows
that $\gamma_r(\wG)$ is closed when $\wG$ is topologically finitely
generated, so that $\Gamma_r(\wG) = \gamma_r(\wG)$.  The corresponding step
rests on that theorem in \cite[\S3.2]{Marin-aspm12}, for $r=2$, and in
\cite[\S2]{Nikolov-Wykowski-mathann26}, for all $r$.  Step~3 passes through
the nilpotent quotients $Q_r$ rather than through $W(G)$, because it needs
the ambient profinite topology to induce the full profinite topology on each
subgroup it completes.  This holds for subgroups of finite index 
(Lemma~\ref{lem:fi}\ref{fi-ab}) and, by Mal'cev's theorem, for central subgroups of
finitely generated nilpotent groups, the only other ones completed there;
it fails for $G'$ in general: for $G$ free of rank two, the induced topology
on $G'$ has countably many open subgroups, and the full one uncountably many
\cite[Ex.~3.1.3(1)]{Ribes-Zalesskii-10}, so that one obtains only a
surjection $\widehat{G'} \surj \overline{\alpha(G')}$, although $G'$ is
closed in the profinite topology of $G$.

Nor does the argument use residual finiteness of $G$, which one might
expect to be a standing hypothesis, since $\wG$ is also the completion of
$G/\ker\alpha$; arrangement groups are not known to be residually finite
(Question~\ref{quest:res-finite}).  The reason is the following lemma.

\begin{lemma}
\label{lem:rf-kernel}
Let $G$ be a finitely generated group, and let $K$ be the kernel of
$\alpha \colon G \to \wG$, that is, the intersection of all
finite-index subgroups of $G$.  Then $K \le G''$, and
$K \le W(G)\gamma_r(G)$ for every set of words $W$ and every $r \ge 1$.
Consequently the group $G^{\rf} \coloneqq G/K$ is residually finite,
$\widehat{G^{\rf}} = \wG$, and the projection
$G \surj G^{\rf}$ induces isomorphisms
\[
G/G'' \isom G^{\rf}/(G^{\rf})''
\quad\text{and}\quad
G/W(G)\gamma_r(G) \isom
G^{\rf}/W(G^{\rf})\gamma_r(G^{\rf}) .
\]
\end{lemma}

\begin{proof}
The subgroup $K$ lies in every normal subgroup $N$ of $G$ for which $G/N$ is
residually finite.  This applies to $N = W(G)\gamma_r(G)$, the quotient
being finitely generated and nilpotent, and to $N = G''$, since finitely
generated metabelian groups are residually finite, by a theorem of P.~Hall
\cite{Hall-plms59}.  Every finite quotient of $G$ factors through
$G^{\rf}$, so $\widehat{G^{\rf}} = \wG$ and
$G^{\rf}$ is residually finite.  Finally, epimorphisms carry verbal
subgroups onto verbal subgroups, so $(G^{\rf})'' = G''/K$ and
$W(G^{\rf})\gamma_r(G^{\rf}) = W(G)\gamma_r(G)/K$, since $K$
lies in both.
\end{proof}

The loci $\VV^1_s$ depend only on $G/G''$ (\S\ref{subsec:jumploci}), and
the groups $\gr_r(G/W(G))$ only on the quotients $G/W(G)\gamma_{r+1}(G)$;
by Lemma~\ref{lem:rf-kernel}, these quotients do not change when $G$ is
replaced by $G^{\rf}$.  So Theorem~\ref{thm:main} and
Proposition~\ref{prop:CH} in degree one, and Theorem~\ref{thm:graded} in
full, are statements about $G^{\rf}$.  In higher degrees, the relevant
hypotheses are goodness and asphericity, not residual finiteness.  The
generality is not vacuous: the smooth projective varieties of Toledo
\cite{Toledo-ihes93}, whose fundamental groups are not residually finite,
satisfy Property~$\mathsf{TT}_k$ for every $k$ \cite{Budur-Wang-ens15}, so
Theorem~\ref{thm:main} and Proposition~\ref{prop:CH}\ref{CH-one} apply to
them.

What Theorem~\ref{thm:verbal} recovers are the graded pieces of the
nilpotent quotients $Q_r = G/W(G)\gamma_{r+1}(G)$, not the quotients
themselves, and this cannot be improved.  By Pickel's theorem
\cite{Pickel-tams71}, $\widehat{Q_r}$ determines $Q_r$ only up to finitely
many possibilities, and for the pair of Corollary~\ref{cor:acgm} already
the quotients by $\gamma_4$ differ (Remark~\ref{rem:acgm-mechanism}),
although the completions are isomorphic.  A finitely generated abelian
group, by contrast, is recovered from its profinite completion, and this is
what the proof of Theorem~\ref{thm:graded} uses.  For the same reason,
Rybnikov's separation of the quotients $G/\gamma_4$ \cite{Rybnikov-faa11}
does not by itself yield a profinite separation; that requires the further
arguments of \S\ref{subsec:rybnikov}.

\section{Isogenies, finite covers, and \texorpdfstring{$\Sigma$}{Sigma}-invariants}
\label{sect:isogeny}

This section proves part~\ref{main-isogeny} of Theorem~\ref{thm:main}, 
extends it to the jump loci of any finite family of finite covers at once
(Proposition~\ref{prop:covers}), and applies the result to the
Bieri--Neumann--Strebel--Renz invariants.  For a finitely generated 
group $G$ we write $A_G = G_{\ab}$ and
$\TT_G = \Hom(G,\C^\times) = \Hom(A_G,\C^\times)$; for a finitely generated
abelian group $A$ we write $\widehat A = A \otimes \widehat\Z$ and $\Tors A$
for its torsion subgroup.  A \emph{torsion-translated subtorus} of $\TT_G$
is a subset $\rho T$, with $T$ a connected algebraic subgroup and $\rho$ a
character of finite order; a Zariski closed subset of $\TT_G$ is
\emph{torsion-toral} if it is a finite union of such.  Property
$\mathsf{TT}_k$ says that $\VV^q_s(X,\C)$ is torsion-toral for $q \le k$. 
We abbreviate $\VV^q = \VV^q_1$.

\subsection{Subtori and their torsion points}
\label{subsec:subtori}

Let $A$ be a finitely generated abelian group.  If $K \le A$ is a subgroup
with $A/K$ torsion-free, then $T_K \coloneqq \Hom(A/K,\C^\times)$ is a
connected algebraic subgroup of $\Hom(A,\C^\times)$ of dimension
$\rank(A/K)$, and every connected algebraic subgroup arises
in this way from a unique such $K$, namely the common kernel of its
elements.  A character of $A$ of finite order extends uniquely to a
continuous character of $\widehat A$, and vanishes on $K$ if and 
only if its extension vanishes on the closure
$\widehat K = K \otimes \widehat\Z$ of $K$ in $\widehat A$ (the map
$K \otimes \widehat\Z \to \widehat A$ is injective, $\widehat\Z$ being flat
over $\Z$).  Writing $\mu_\infty \subset \C^\times$ for the group of roots
of unity, we thus have
\begin{equation}
\label{eq:torsion-subtorus}
T_K \cap \Hom(A,\mu_\infty) =
\bigl\{ \chi \in \Hom_{\mathrm{cont}}(\widehat A,\mu_\infty) :
\chi(\widehat K) = 1 \bigr\} ,
\end{equation}
and $T_K \cap \Hom(A,\mu_\infty) \cong (\Q/\Z)^{\rank(A/K)}$
is a divisible group.  We set
$\Trop(\rho T_K) \coloneqq \Hom(A/K,\R) \subseteq \Hom(A,\R)$, and, for
$V$ torsion-toral, $\Trop(V)$ is the union of the $\Trop(C)$ over the
irreducible components $C$ of $V$; by \cite[Cor.~2.4]{Suciu-mathann21},
this is the tropicalization of $V$ in the sense of
\cite{Suciu-mathann21}.

\begin{lemma}
\label{lem:div-coset}
Let $V \subseteq \Hom(A,\C^\times)$ be torsion-toral, with irreducible
components $C_1,\dots,C_r$.  Let $S \le \Hom(A,\mu_\infty)$ be a divisible
subgroup, and $x \in \Hom(A,\mu_\infty)$.  If $xS \subseteq V$, then
$xS \subseteq C_j$ for some $j$.
\end{lemma}

\begin{proof}
The Zariski closure $\overline{S}$ is an algebraic subgroup, and its
identity component $\overline{S}^{\,0}$ has finite index in it, so
$S \cap \overline{S}^{\,0}$ has finite index in $S$.  A divisible abelian
group has no proper subgroups of finite index (if $[S:S'] = n$, then
$S = S^n \subseteq S'$), so $S \subseteq \overline{S}^{\,0}$, whence
$\overline{S}$ is connected, hence irreducible.  The closure of $xS$ is
$x\overline{S}$, an irreducible closed subset of $V$, and therefore lies in
one of the $C_j$.
\end{proof}

\begin{lemma}
\label{lem:match}
Let $A_1, A_2$ be finitely generated abelian groups, let
$V_i \subseteq \Hom(A_i,\C^\times)$ be torsion-toral, and let
$\vartheta \colon \Hom(A_2,\mu_\infty) \to \Hom(A_1,\mu_\infty)$ be an
isomorphism of groups such that
$\vartheta(V_2 \cap \Hom(A_2,\mu_\infty)) = V_1 \cap \Hom(A_1,\mu_\infty)$.
\begin{enumerate}[label=\textup{(\roman*)}, itemsep=1.5pt]
\item \label{match-bij}
For each irreducible component $C$ of $V_2$ there is a unique irreducible
component $C'$ of $V_1$ with $\vartheta(C_{\tors}) = C'_{\tors}$,
where the subscript denotes the set of torsion points, and $C \mapsto C'$
is a bijection between the sets of components.
\item \label{match-dim}
If $C = \rho T_K$ and $C' = \rho' T_{K'}$, then $\dim C = \dim C'$ and
\[
\vartheta\bigl(T_K \cap \Hom(A_2,\mu_\infty)\bigr) =
T_{K'} \cap \Hom(A_1,\mu_\infty) .
\]
\end{enumerate}
\end{lemma}

\begin{proof}
Let $C = \rho T_K$ with $\rho$ of finite order.  Then
$C_{\tors} = \rho\,(T_K \cap \Hom(A_2,\mu_\infty))$ is a coset of
a divisible group, hence so is $\vartheta(C_{\tors})$, and it lies in
$V_1$.  By Lemma~\ref{lem:div-coset}, $\vartheta(C_{\tors}) \subseteq
C'_{\tors}$ for some component $C'$ of $V_1$.  Applying the same
argument to $\vartheta^{-1}$ gives a component $C''$ of $V_2$ with
$C_{\tors} \subseteq \vartheta^{-1}(C'_{\tors}) \subseteq
C''_{\tors}$.  Since the torsion points of a torsion-translated
subtorus are Zariski dense in it (Lemma~\ref{lem:density}), $C \subseteq
C''$, so $C = C''$, the decomposition into components being irredundant;
hence $\vartheta(C_{\tors}) = C'_{\tors}$.  By density again,
$C'$ is determined by $C'_{\tors}$, which gives uniqueness, and the
inverse map $C' \mapsto C$ is obtained from $\vartheta^{-1}$ in the same way.
This proves \ref{match-bij}.  

For \ref{match-dim}, the group of
translations $T_K \cap \Hom(A_2,\mu_\infty)$ is the set of quotients
$xy^{-1}$ with $x,y \in C_{\tors}$, and likewise for $C'$; since
$\vartheta$ is a homomorphism, it carries one group of translations onto the
other.  Finally, $\dim C$ is recovered from the divisible group
$T_K \cap \Hom(A_2,\mu_\infty) \cong (\Q/\Z)^{\dim C}$ as the dimension
over $\F_p$ of its $p$-torsion, for any prime $p$.
\end{proof}

\begin{proposition}
\label{prop:approx}
Let $A_1, A_2$ be finitely generated abelian groups and 
$\Phi \colon \widehat{A_1} \to \widehat{A_2}$ an isomorphism of 
profinite groups, and let $\Phi^\vee \colon \Hom(A_2,\mu_\infty) \to
\Hom(A_1,\mu_\infty)$ be the induced isomorphism,
$\Phi^\vee\chi = \widehat\chi \circ \Phi|_{A_1}$.  Let
$\{(V^\alpha_2, V^\alpha_1)\}_{\alpha \in I}$ be a finite family of pairs
of torsion-toral subsets $V^\alpha_i \subseteq \Hom(A_i,\C^\times)$ with
$\Phi^\vee(V^\alpha_2 \cap \Hom(A_2,\mu_\infty)) =
V^\alpha_1 \cap \Hom(A_1,\mu_\infty)$ for all $\alpha$.  Then, for every
integer $m_0 \ge 1$, there is an injective homomorphism
$L \colon A_1 \to A_2$ with finite cokernel and
$L \otimes \Z/m_0 = \Phi \otimes \Z/m_0$, such that:
\begin{enumerate}[label=\textup{(\roman*)},itemsep=1.5pt]
\item \label{approx-isog}
$L^\vee \colon \Hom(A_2,\C^\times) \to \Hom(A_1,\C^\times)$ is surjective,
with finite kernel.
\item \label{approx-comp}
For each $\alpha$, the isogeny $L^\vee$ maps each irreducible component $C$ of
$V^\alpha_2$ onto the component $C'$ of $V^\alpha_1$ attached to it by
Lemma~\textup{\ref{lem:match}} \textup{(}with $\vartheta = \Phi^\vee$\textup{)};
in particular $L^\vee(V^\alpha_2) = V^\alpha_1$.
\item \label{approx-trop}
The dual isomorphism $L^\vee_\R \colon \Hom(A_2,\R) \to \Hom(A_1,\R)$,
given by $\chi \mapsto \chi \circ L$, carries $\Trop(V^\alpha_2)$ onto
$\Trop(V^\alpha_1)$ for each $\alpha$.
\end{enumerate}
\end{proposition}

\begin{proof}
\emph{Step 1: the lattices attached to components.}  Let $\mathcal{C}$ be
the (finite) set of irreducible components of the sets $V^\alpha_2$
with $\alpha \in I$.  For $C \in \mathcal{C}$ fix a torsion point $\rho_C \in C$
and write $C = \rho_C T_{K_C}$, and let $C' = \rho'_C T_{K'_C}$ be the
component attached to $C$ by Lemma~\ref{lem:match} (applied to the 
pair in which $C$ occurs).  Since $\Phi^\vee(C_{\tors}) =
C'_{\tors}$, we may take $\rho'_C = \Phi^\vee\rho_C$.  The isomorphism 
$\Phi^\vee$ carries the annihilator of $\widehat{K_C}$ in
$\Hom_{\mathrm{cont}}(\widehat{A_2},\mu_\infty)$ onto the annihilator of
the closed subgroup $\Phi^{-1}(\widehat{K_C})$ of $\widehat{A_1}$, and, by
Lemma~\ref{lem:match}\ref{match-dim} and \eqref{eq:torsion-subtorus}, onto
the annihilator of $\widehat{K'_C}$.  As a closed subgroup of a profinite
abelian group is determined by its annihilator in the Pontryagin dual,
\begin{equation}
\label{eq:lattice-match}
\Phi\bigl(\widehat{K'_C}\bigr) = \widehat{K_C}
\qquad \text{for all } C \in \mathcal{C} .
\end{equation}

\emph{Step 2: approximation.}  Let $\mathcal{H} \le \Hom(A_1,A_2)$ be the
subgroup of
homomorphisms $L$ with $L(K'_C) \subseteq K_C$ for all $C \in \mathcal{C}$.
It is the kernel of the homomorphism
$\Hom(A_1,A_2) \to \bigoplus_{C} \Hom(K'_C, A_2/K_C)$.  Since $A_1$ is
finitely presented and $\widehat\Z$ is flat over $\Z$, tensoring with
$\widehat\Z$ identifies $\Hom(A_1,A_2) \otimes \widehat\Z$ with
$\Hom_{\widehat\Z}(\widehat{A_1},\widehat{A_2})$ and
$\mathcal{H} \otimes \widehat\Z$
with the kernel of the corresponding map to
$\bigoplus_C \Hom_{\widehat\Z}(\widehat{K'_C}, \widehat{A_2/K_C})$.  By
flatness again, $\widehat{K_C}$ is the kernel of
$\widehat{A_2} \to \widehat{A_2/K_C}$, so \eqref{eq:lattice-match} says that
$\Phi \in \mathcal{H} \otimes \widehat\Z$.  Choose an integer $m \ge 2$
divisible by $m_0$, by the exponent of $\Tors A_2$, and by the order of
each $\rho_C$ with $C \in \mathcal{C}$.  Since
$(\mathcal{H} \otimes \widehat\Z)/m(\mathcal{H}\otimes\widehat\Z) =
\mathcal{H}/m\mathcal{H}$, there is $L \in \mathcal{H}$ with
$L - \Phi \in m\,(\mathcal{H} \otimes \widehat\Z)$; in particular,
$(L - \Phi)(\widehat{A_1}) \subseteq m\widehat{A_2}$, so
$L \otimes \Z/m = \Phi \otimes \Z/m$, and the same holds with $m_0$ in
place of $m$.

\vspace*{1pt}
\emph{Step 3: $L$ is injective with finite cokernel.}  As $\Phi$ is an
isomorphism, $A_1$ and $A_2$ have the same rank $b$.  The maps $L$ and
$\Phi$ induce $\bar L \colon A_1/\Tors \to A_2/\Tors$ and
$\bar\Phi \colon \widehat{A_1}/\Tors \to \widehat{A_2}/\Tors$, with
$\bar L \equiv \bar\Phi$ modulo $m$.  In bases, $\det \bar\Phi \in
\widehat\Z^\times$, so $\det \bar L \in \Z$ reduces to a unit of $\Z/m$; as
$m \ge 2$, this forces $\det \bar L \ne 0$, and $\bar L$ is injective with finite
cokernel.  On torsion, $L$ and $\Phi$ both map $\Tors A_1$ to
$\Tors A_2 = \Tors \widehat{A_2}$, and for $t \in \Tors A_1$ the difference
$L(t) - \Phi(t)$ lies in $\Tors A_2 \cap m\widehat{A_2}$.  Writing
$\widehat{A_2} = \widehat{F} \oplus \Tors A_2$ with $\widehat F \cong
\widehat\Z^b$ torsion-free, an element $my$ is torsion only if the
$\widehat F$-component of $y$ vanishes, so
$\Tors A_2 \cap m\widehat{A_2} = m\Tors A_2 = 0$, the exponent of
$\Tors A_2$ dividing $m$.  Hence $L|_{\Tors A_1} = \Phi|_{\Tors A_1}$ is
injective.  Together with the injectivity of $\bar L$, this shows that
$L$ is injective; its cokernel is finite since the ranks agree.  As
$\C^\times$ is divisible, $L^\vee$ is then surjective, with kernel dual to
$\coker L$; this is \ref{approx-isog}.

\vspace*{1pt}
\emph{Step 4: components.}  Let $C \in \mathcal{C}$.  Since
$L(K'_C) \subseteq K_C$, the map $L$ induces
$L_C \colon A_1/K'_C \to A_2/K_C$, a homomorphism of free abelian groups of
the same rank $\dim C = \dim C'$ (Lemma~\ref{lem:match}\ref{match-dim}).
By \eqref{eq:lattice-match}, $\Phi$ induces an isomorphism
$\widehat{A_1/K'_C} \to \widehat{A_2/K_C}$, congruent to $L_C$ modulo $m$;
the argument of Step 3 shows that $\det L_C \ne 0$, so $L_C$ is injective
with finite cokernel, and its dual map from $\Hom(A_2/K_C,\C^\times)$ to
$\Hom(A_1/K'_C,\C^\times)$ is surjective.  That
is, $L^\vee(T_{K_C}) = T_{K'_C}$.  Moreover, $\rho_C$ takes values in
$\mu_m$ and so factors through $A_2/mA_2$, where $L$ and $\Phi$ agree;
hence $L^\vee\rho_C = \Phi^\vee\rho_C = \rho'_C$.  Therefore
$L^\vee(C) = \rho'_C\, T_{K'_C} = C'$, and \ref{approx-comp} follows from
the bijectivity in Lemma~\ref{lem:match}\ref{match-bij}.

\vspace*{1pt}
\emph{Step 5: tropicalizations.}  $L^\vee_\R$ is an isomorphism since
$L \otimes \Q$ is, and it maps $\Trop(C) = \Hom(A_2/K_C,\R)$ onto
$\Trop(C') = \Hom(A_1/K'_C,\R)$ since $L_C \otimes \R$ is an isomorphism.
Taking unions over components gives \ref{approx-trop}.
\end{proof}

\subsection{Jump loci up to isogeny}
\label{subsec:proof-isogeny}

\begin{lemma}
\label{lem:transport-ab}
For every torsion character $\rho$ of $G_2$, regarded as a character of
$(G_2)_{\ab}$, we have $\rho^\phi = \phi_{\ab}^\vee\rho$, in the notation of
Proposition~\textup{\ref{prop:approx}} with $\Phi = \phi_{\ab}$.
\end{lemma}

\begin{proof}
The continuous extension $\widehat\rho \colon \widehat{G_2} \to \mu_m$
factors through the maximal abelian quotient of $\widehat{G_2}$, which is
$(G_2)_{\ab} \otimes \widehat\Z$ (Lemma~\ref{lem:fi}\ref{fi-ab} with
$N = G_2$), and $\phi$ induces $\phi_{\ab}$ on these quotients.  Hence
$\rho^\phi = \widehat\rho \circ \phi \circ \alpha_1$ is the composite of
$G_1 \to (G_1)_{\ab}$ with $\widehat{\rho}_{\ab} \circ \phi_{\ab}$, which is
$\phi_{\ab}^\vee\rho$.
\end{proof}

\begin{proof}[Proof of Theorem~\textup{\ref{thm:main}\ref{main-isogeny}}]
Choose $s_0$ with $\VV^q_s(X_i,\C) = \emptyset$ whenever $q \le k$ and
$s > s_0$, for $i = 1,2$, as is possible since $X_i$ has finite $k$-skeleton.  
By Property $\mathsf{TT}_k$, the sets $\VV^q_s(X_i,\C)$ with $q \le k$ and
$1 \le s \le s_0$ are torsion-toral.  By part~\ref{main-depth} and 
Lemma~\ref{lem:transport-ab}, the map $\phi_{\ab}^\vee$ carries the torsion points
of $\VV^q_s(X_2,\C)$ onto those of $\VV^q_s(X_1,\C)$, exactly as in the
proof of part~\ref{main-loci}.  Proposition~\ref{prop:approx}, applied with
$\Phi = \phi_{\ab}$ and $m_0 = m$, gives the assertion.
\end{proof}

\begin{remark}
\label{rem:isogeny-scope}
\begin{enumerate}[label=(\alph*), leftmargin=*]
\item \label{isog-a}
The isogeny $L^\vee$ is in general not injective, so the preimage of
$\VV^q_s(X_1,\C)$ under $L^\vee$ may be larger than $\VV^q_s(X_2,\C)$.
Combinatorial features of the loci are nevertheless preserved by the
bijection of Lemma~\ref{lem:match} induced by $\rho \mapsto \rho^\phi$.
These include the number of components of each dimension, the number of
those passing through $\one$, and the orders of the torsion points on
each component.
\item \label{isog-b}
We do not know whether $L$ can always be chosen to be an isomorphism when
$\phi$ is not rigidified; compare Question~\ref{quest:no-T}.
\end{enumerate}
\end{remark}

The following lemma shows that nothing is lost by formulating
rigidification on the torsion-free quotients of the abelianizations, as is
done for $3$-manifold groups (Remark~\ref{rem:rigidify}).

\begin{lemma}
\label{lem:free-part}
Let $\phi \colon \widehat{G_1} \to \widehat{G_2}$ be an isomorphism, and
write $\bar A_i = (G_i)_{\ab}/\Tors$ and $\bar\phi_{\ab}$ for the
isomorphism $\widehat{\bar A_1} \to \widehat{\bar A_2}$ induced by
$\phi_{\ab}$.  If $\bar\phi_{\ab} = u\cdot(\bar\psi \otimes \widehat\Z)$ for
an isomorphism $\bar\psi \colon \bar A_1 \to \bar A_2$ and a unit
$u \in \widehat\Z^\times$, then $\phi$ is rigidified, in the sense of
Definition~\textup{\ref{def:rigid}}, by an isomorphism
$\psi \colon (G_1)_{\ab} \to (G_2)_{\ab}$ lifting $\bar\psi$, with the same
unit $u$.
\end{lemma}

\begin{proof}
Choose splittings $(G_i)_{\ab} = F_i \oplus E_i$ with $F_i$ free and
$E_i = \Tors (G_i)_{\ab}$, so that
$(G_i)_{\ab} \otimes \widehat\Z = \widehat{F_i} \oplus E_i$ and
$\widehat{F_i} \cong \widehat{\bar A_i}$.  Since $\phi_{\ab}$ maps $E_1$
onto $E_2$ and $\widehat{F_2}$ is torsion-free, it has the triangular form
$\phi_{\ab}(f,e) = \bigl(u\,(\bar\psi\otimes\widehat\Z)(f),\, c(f) + d(e)\bigr)$,
with $c \colon \widehat{F_1} \to E_2$ continuous and
$d = \phi_{\ab}|_{E_1}$ an isomorphism.  Choose $u' \in \Z$ with
$uu' \equiv 1$ modulo the exponent of $E_2$, and define
$\psi(f,e) = \bigl(\bar\psi(f),\, u'(c(f) + d(e))\bigr)$ for $f \in F_1$
and $e \in E_1$.  This is an isomorphism, being triangular with invertible
diagonal blocks, and $u\cdot(\psi \otimes \widehat\Z) = \phi_{\ab}$, since
$c$ is the continuous extension of $c|_{F_1}$ and $uu'$ acts as the
identity on $E_2$.
\end{proof}

\begin{remark}
\label{rem:rigidify}
Rigidification is the \emph{$\widehat\Z^\times$-regularity} introduced by
Liu \cite{Liu-invent23} for $3$-manifold groups; following Boileau and
Friedl \cite{Boileau-Friedl-amstud20}, $\phi$ is \emph{regular} if one can
take $u = \pm 1$, that is, if $\phi_{\ab} = \psi \otimes \widehat\Z$ after
replacing $\psi$ by $-\psi$ if need be.  Liu proved that every isomorphism
between the profinite completions of the fundamental groups of two
finite-volume hyperbolic $3$-manifolds is $\widehat\Z^\times$-regular, and
regularity was then established by Xu \cite{Xu-arx25} in the cusped case
and by Hanany \cite{Hanany-arx26} in the closed case.  Liu formulates the
condition on the torsion-free quotients $(G_i)_{\ab}/\Tors$, which by
Lemma~\ref{lem:free-part} loses nothing.
\end{remark}

\subsection{Finite covers}
\label{subsec:covers}

Let $\phi \colon \widehat{G_1} \to \widehat{G_2}$ be an isomorphism.  Given
a finite-index subgroup $N_2 \le G_2$, let $N_1 \le G_1$ be the subgroup
corresponding to it under Lemma~\ref{lem:fi}\ref{fi-det}, that is,
$N_1 = \alpha_1^{-1}(\phi^{-1}(U_2))$ with
$U_2 = \overline{\alpha_2(N_2)}$; we say that $(N_1,N_2)$ is a
\emph{corresponding pair}.  By Lemma~\ref{lem:fi}\ref{fi-corr},
$\overline{\alpha_1(N_1)} = \phi^{-1}(U_2)$, and by Lemma~\ref{lem:fi}\ref{fi-ab} 
the inclusions induce isomorphisms
$\widehat{N_i} \cong \overline{\alpha_i(N_i)}$, compatible with the maps
$\alpha_{N_i} = \alpha_i|_{N_i}$.  Hence $\phi$ restricts to an isomorphism
$\phi_N \colon \widehat{N_1} \to \widehat{N_2}$.

\begin{lemma}
\label{lem:restrict}
For every torsion character $\rho$ of $G_2$, we have
$(\rho^\phi)|_{N_1} = (\rho|_{N_2})^{\phi_N}$.
\end{lemma}

\begin{proof}
By uniqueness of continuous extensions, the extension of $\rho|_{N_2}$ to
$\widehat{N_2} = U_2$ is $\widehat\rho|_{U_2}$.  Hence
$(\rho|_{N_2})^{\phi_N} = \widehat\rho|_{U_2} \circ \phi|_{\phi^{-1}(U_2)}
\circ \alpha_1|_{N_1} = (\widehat\rho \circ \phi \circ \alpha_1)|_{N_1}$.
\end{proof}

Let $X$ be a connected CW-complex with finite $k$-skeleton and set 
$G = \pi_1(X)$; let $N \le G$ have finite index, and let $p \colon Y \to X$
be the corresponding cover.  Restriction of characters is a homomorphism of
algebraic groups $p^{*} \colon \TT_G \to \TT_N$, dual to the map
$\iota \colon A_N \to A_G$ induced by inclusion; its differential at
$\one$ is the restriction map
$\res \colon H^1(X;\C) \to H^1(Y;\C)$, which is injective.  We set
\begin{equation}
\label{eq:vqsN}
\VV^q_s(X;N) \coloneqq (p^{*})^{-1}\bigl(\VV^q_s(Y,\C)\bigr)
\subseteq \TT_G ,
\qquad
\VV^{\le q}(X;N) \coloneqq \bigcup_{i \le q} \VV^i_1(X;N) .
\end{equation}
We also write $\VV^{\le q}(X,\C) = \bigcup_{i \le q} \VV^i_1(X,\C)$;
for $N = G$ the cover is $X$ itself, and $\VV^{\le q}(X;G) = \VV^{\le q}(X,\C)$.

\begin{remark}
\label{rem:twisted-loci}
If $N$ is a normal subgroup, then $\VV^q_1(X;N) = \bigcup_L \VV^q(X,L)$, 
the union over the irreducible complex representations $L$ of $G/N$ of the
\emph{twisted homology jump loci} of \cite{Liu-Suciu-arx26}, by
\cite[Prop.~2.6]{Liu-Suciu-arx26}.  The loci \eqref{eq:vqsN} are thus the
twisted jump loci attached to representations with finite image.
\end{remark}

\begin{lemma}
\label{lem:preimage}
Let $V \subseteq \TT_N$ be torsion-toral.
\begin{enumerate}[label=\textup{(\roman*)}, itemsep=1.5pt]
\item \label{pre-toral}
$(p^{*})^{-1}(V)$ is torsion-toral.
\item \label{pre-trop}
For each irreducible component $C'$ of $(p^{*})^{-1}(V)$ there is a
component $C$ of $V$ with $p^{*}(C') \subseteq C$, and then
$\res(\Trop(C')) \subseteq \Trop(C)$.
\end{enumerate}
\end{lemma}

\begin{proof}
Let $C = \rho T$ be a component of $V$, and let
$q \colon \TT_G \to \TT_N/T$ be the composite of $p^{*}$ with the quotient
map; then $(p^{*})^{-1}(C) = q^{-1}(\bar\rho)$, where $\bar\rho$, the
image of $\rho$, has finite order.  Suppose this fiber is nonempty.  The
image $Q = q(\TT_G)$ is a closed subgroup, hence diagonalizable, and the
surjection $\TT_G \to Q$ is onto on torsion points: dually, the character
group of $Q$ embeds in $A_G$, and $\Hom(-,\mu_\infty)$ is exact, $\mu_\infty$
being divisible.  So $q^{-1}(\bar\rho) = x \ker q$ with $x$ of finite order.
The diagonalizable group $\ker q$ is the product of a torus and a finite
group, so it is the union of finitely many cosets of $(\ker q)^0$, each
containing a torsion point.  This proves \ref{pre-toral}, and shows that
each component $C'$ of $(p^{*})^{-1}(V)$ is of the form
$x'(\ker q)^0$ for some component $C$ of $V$, with $p^{*}(C') \subseteq C$.
Write $(\ker q)^0 = T_{K'}$ with $K' \le A_G$ and $T = T_K$ with
$K \le A_N$.  Then $p^{*}(T_{K'}) \subseteq T_K$ means that every character
vanishing on $K'$ vanishes on $\iota(K)$, so $\iota(K) \subseteq K'$, the
quotient $A_G/K'$ being torsion-free.  Consequently
$\res = \iota^\vee_\R$ maps $\Trop(C') = \Hom(A_G/K',\R)$ into
$\Hom(A_N/K,\R) = \Trop(C)$, which is \ref{pre-trop}.
\end{proof}

\begin{proposition}
\label{prop:covers}
Let $X_1, X_2$ be connected CW-complexes with finite $k$-skeleton,
$G_i = \pi_1(X_i)$, and $\phi \colon \widehat{G_1} \to \widehat{G_2}$ an
isomorphism.  Let $\mathcal{N}$ be a finite set of corresponding pairs
$(N_1,N_2)$, with covers $Y_i \to X_i$, such that both $Y_i$ satisfy
Property $\mathsf{TT}_k$ and $\phi_N$ satisfies $\CH_k(\phi_N)$ for each
pair.  Then, for every integer $m \ge 1$, there is an injective
homomorphism $L \colon (G_1)_{\ab} \to (G_2)_{\ab}$ with finite cokernel
and $L \otimes \Z/m = \phi_{\ab} \otimes \Z/m$, such that, for all
$(N_1,N_2) \in \mathcal{N}$, all $q \le k$, and all $s \ge 1$, the isogeny $L^\vee$
maps each irreducible component of $\VV^q_s(X_2;N_2)$ onto an irreducible
component of $\VV^q_s(X_1;N_1)$ of the same dimension, inducing a
bijection between components, and $L^\vee_\R$ carries
$\Trop(\VV^q_s(X_2;N_2))$ onto $\Trop(\VV^q_s(X_1;N_1))$.
\end{proposition}

\begin{proof}
By Lemma~\ref{lem:preimage}\ref{pre-toral} and Property $\mathsf{TT}_k$
for the spaces $Y_i$, the sets $\VV^q_s(X_i;N_i)$ are torsion-toral, and only
finitely many of them are nonempty.  A torsion character $\rho$ of $G_2$
lies in $\VV^q_s(X_2;N_2)$ if and only if
$\depth^q_\C(\rho|_{N_2}) \ge s$, the depth being computed on $Y_2$.  
Lemmas~\ref{lem:restrict} and~\ref{lem:depth}, applied to $\phi_N$, give
$\depth^q_\C((\rho^\phi)|_{N_1}) = \depth^q_\C(\rho|_{N_2})$.  Hence
$\rho \mapsto \rho^\phi$, which is $\phi_{\ab}^\vee$ by Lemma~\ref{lem:transport-ab}, 
carries the torsion points of $\VV^q_s(X_2;N_2)$ onto those of $\VV^q_s(X_1;N_1)$, 
and Proposition~\ref{prop:approx} applies.
\end{proof}

Taking $(G_1,G_2) \in \mathcal{N}$ recovers part~\ref{main-isogeny} of
Theorem~\ref{thm:main}.

\begin{remark}
\label{rem:cover-hyps}
\begin{enumerate}[label=(\alph*), leftmargin=*]
\item \label{cov-a}
$\CH_1(\phi_N)$ always holds, by Proposition~\ref{prop:CH}\ref{CH-one} applied to
$\phi_N$.
\item \label{cov-b}
Property $\mathsf{TT}_k$ holds for all $k$ for every finite cover of a
smooth quasi-projective variety, the cover being again smooth and
quasi-projective \cite{Budur-Wang-ens15}, and for every
finite cover of a compact K\"ahler manifold, the cover being again compact
K\"ahler; for the latter, the components of the jump loci are translated
subtori by the work recalled in \cite[\S12]{Suciu-mathann21}, and each
contains a torsion point by \cite{Wang-mrl16}.
\item \label{cov-c}
If $X_1$ and $X_2$ are aspherical and of finite type, with $G_1$ and $G_2$
good, then $\CH_k(\phi_N)$ holds for all $k$ by Proposition~\ref{prop:CH}\ref{CH-all}, 
since finite covers of aspherical spaces are
aspherical and finite-index subgroups of good groups are good
\cite[Lem.~3.2]{Grunewald-JaikinZapirain-Zalesskii-duke08}.
\item \label{cov-d}
Even when $\phi$ is rigidified, as for Galois-conjugate arrangements, we do
not know whether the restrictions $\phi_N$ are; Proposition~\ref{prop:covers} 
does not require it.
\end{enumerate}
\end{remark}

\begin{proof}[Proof of Theorem~\textup{\ref{thm:loci}}]
Smooth quasi-projective varieties and compact K\"ahler manifolds are of
finite type, and they and their finite covers satisfy Property
$\mathsf{TT}_k$ for every $k$, by Remark~\ref{rem:cover-hyps}\ref{cov-b}.
Condition $\CH_1$ holds for $\phi$ and for its restrictions $\phi_N$ by
Proposition~\ref{prop:CH}\ref{CH-one}, and $\CH_k$ holds for every $k$
under the hypotheses of part~\ref{loci-all}, by Proposition~\ref{prop:CH}\ref{CH-all} 
and Remark~\ref{rem:cover-hyps}\ref{cov-c}.
Parts~\ref{loci-depth} and~\ref{loci-rigid} now follow from Theorem~\ref{thm:main}\ref{main-depth} and~\ref{main-loci}, and part~\ref{loci-isog}, including the assertion on finite covers, from
Proposition~\ref{prop:covers}; for part~\ref{loci-all}, apply the same
results with $k$ arbitrary.
\end{proof}

\subsection{\texorpdfstring{$\Sigma$}{Sigma}-invariants}
\label{subsec:sigma}

For a connected CW-complex $X$ with finite $q$-skeleton and
$G = \pi_1(X)$, let $S(X)$ be the sphere of rays in
$H^1(X;\R)\setminus\{0\}$, and let $\Sigma^q(X;\Z) \subseteq S(X)$ be the
Bieri--Neumann--Strebel--Renz (BNSR) invariants; recall that
$\Sigma^1(X;\Z) = \Sigma^1(G;\Z) = -\Sigma^1(G)$, where $\Sigma^1(G)$ is
the BNS invariant.  We use two facts from \cite{Suciu-mathann21}.  First,
by a theorem of Farber, Geoghegan, and Sch\"utz
\cite[Thm.~5.6]{Suciu-mathann21},
\begin{equation}
\label{eq:fgs}
\chi \in \Sigma^q(X;\Z) \iff H_i(X;\widehat{\Z G}_{-\chi}) = 0
\ \text{ for all } i \le q ,
\end{equation}
where $\widehat{\Z G}_{-\chi}$ is the Novikov--Sikorav completion.
Second, by \cite[Thm.~6.5]{Suciu-mathann21},
\begin{equation}
\label{eq:trop-bound}
\Sigma^q(X;\Z) \subseteq S\bigl(\Trop(\VV^{\le q}(X,\C))\bigr)^{c} .
\end{equation}

\begin{lemma}
\label{lem:fi-sigma}
Let $p \colon Y \to X$ be a finite connected cover, with
$N = \pi_1(Y) \le G$.  For $0 \ne \chi \in H^1(X;\R)$, we have
$\chi \in \Sigma^q(X;\Z)$ if and only if
$\res\chi \in \Sigma^q(Y;\Z)$.
\end{lemma}

\begin{proof}
Let $g_1,\dots,g_r$ be representatives of the left cosets of $N$ in $G$.
Since $\chi(g_j n) = \chi(g_j) + \chi(n)$, the finiteness condition defining
the Novikov--Sikorav completion can be checked coset by coset, so every
element of $\widehat{\Z G}_{-\chi}$ can be written uniquely as
$\sum_j g_j x_j$ with $x_j \in \widehat{\Z N}_{-\res\chi}$.  Thus the map
$\Z G \otimes_{\Z N} \widehat{\Z N}_{-\res\chi} \to \widehat{\Z G}_{-\chi}$
given by $g \otimes x \mapsto gx$ is an isomorphism of left $\Z G$-modules, 
and Shapiro's lemma gives
$H_i(X;\widehat{\Z G}_{-\chi}) \cong H_i(Y;\widehat{\Z N}_{-\res\chi})$
for all $i$.  As $\res$ is injective, $\res\chi \ne 0$, and the claim
follows from \eqref{eq:fgs}.
\end{proof}

\begin{proposition}
\label{prop:virtual-bound}
Let $N \le G$ have finite index, and suppose the cover $Y$ satisfies
Property $\mathsf{TT}_q$.  Then
\[
\Sigma^q(X;\Z) \subseteq S\bigl(\Trop(\VV^{\le q}(X;N))\bigr)^{c} .
\]
\end{proposition}

\begin{proof}
Let $0 \ne \chi \in \Trop(\VV^{\le q}(X;N))$.  By 
Lemma~\ref{lem:preimage}\ref{pre-trop}, the class $\res\chi$ 
lies in $\Trop(\VV^{\le q}(Y,\C))$, and it is nonzero.  Applying
\eqref{eq:trop-bound} to $Y$ gives $\res\chi \notin \Sigma^q(Y;\Z)$, and
Lemma~\ref{lem:fi-sigma} then gives $\chi \notin \Sigma^q(X;\Z)$.
\end{proof}

For $N = G$, this is the bound \eqref{eq:trop-bound}.  For $q = 1$ one may
replace $\VV^{\le 1}(X;N)$ by $\VV^1(X;N)$, since
$\VV^0(X;N) = \ker p^{*}$ is finite.  The next corollary shows that these
bounds are profinite invariants.

\begin{corollary}
\label{cor:qp-bound}
Let $\phi \colon \widehat{\pi_1(X_1)} \to \widehat{\pi_1(X_2)}$ be an
isomorphism, where $X_1$ and $X_2$ are smooth quasi-projective varieties
\textup{(}for instance, complements of hyperplane arrangements\textup{)} or
compact K\"ahler manifolds, and let $\mathcal{N}$ be a finite set of
corresponding pairs.
\begin{enumerate}[label=\textup{(\roman*)}, itemsep=1.5pt]
\item \label{qp-one}
There is $L$ as in Proposition~\textup{\ref{prop:covers}} with
\[
L^\vee_\R\bigl(\Trop(\VV^1(X_2;N_2))\bigr) = \Trop(\VV^1(X_1;N_1))
\quad\text{for all } (N_1,N_2) \in \mathcal{N} ;
\]
thus the upper bounds for $\Sigma^1(X_i;\Z)$
of Proposition~\textup{\ref{prop:virtual-bound}} correspond under
$L^\vee_\R$.
\item \label{qp-all}
If moreover the $X_i$ are aspherical and their fundamental groups are
good, the same holds for $\Trop(\VV^{\le q}(X_i;N_i))$ and
$\Sigma^q(X_i;\Z)$, for every $q$.
\end{enumerate}
\end{corollary}

\begin{proof}
By Remark~\ref{rem:cover-hyps}\ref{cov-a}--\ref{cov-c}, 
Proposition~\ref{prop:covers} applies with $k = 1$ in \ref{qp-one}, 
and with every $k$ in \ref{qp-all}; for \ref{qp-all}, apply it to the finite 
family of sets $\VV^i_1(X_j;N_j)$ with $i \le q$, whose unions are the
$\VV^{\le q}(X_j;N_j)$.
\end{proof}

For $N = G$, whether the bound in Corollary~\ref{cor:qp-bound}\ref{qp-one} 
is sharp for arrangement complements is
asked in \cite[Question~13.7]{Suciu-mathann21}.  For compact K\"ahler
manifolds, it becomes sharp once finitely many covers are included, as we
now show.

\subsection{K\"ahler groups}
\label{subsec:kahler}

Let $X$ be a compact K\"ahler manifold with $G = \pi_1(X)$.  Let
$\mathcal{E}_X$ be the set of equivalence classes of orbifold fibrations
$f \colon X \to (C_f,\mu_f)$ onto hyperbolic orbifold curves, two
fibrations being equivalent if they differ by an isomorphism of the base
orbifolds; this set is finite \cite[Thm.~2]{Delzant-gafa08}.  For
$f \in \mathcal{E}_X$, write
$f_\# \colon G \surj \Gamma_f$ for the induced epimorphism onto
$\Gamma_f = \pi_1^{\mathrm{orb}}(C_f,\mu_f)$, and
$V_f = f_\#^{*}H^1(\Gamma_f;\R) = f^{*}H^1(C_f;\R)$; this subspace depends
only on the class of $f$.  Delzant's theorem \cite[Thm.~1.1]{Delzant-mathann10}
(see also \cite[Thm.~6.5]{Liu-Suciu-arx26}) states that
\begin{equation}
\label{eq:delzant}
\Sigma^1(G) = S\Bigl(\bigcup\nolimits_{f \in \mathcal{E}_X} V_f\Bigr)^{c} ;
\end{equation}
in particular $\Sigma^1(G) = -\Sigma^1(G) = \Sigma^1(X;\Z)$.

\begin{lemma}
\label{lem:kahler-visible}
Let $X$ be a compact K\"ahler manifold, $G = \pi_1(X)$.
\begin{enumerate}[label=\textup{(\roman*)}, itemsep=1.5pt]
\item \label{kv-upper}
$\Trop(\VV^1(X;N)) \subseteq \bigcup_{f \in \mathcal{E}_X} V_f$ for every
finite-index subgroup $N \le G$.
\item \label{kv-lower}
There is a finite set $\mathcal{N}_X$ of finite-index subgroups of $G$ such
that
\[
\bigcup\nolimits_{N \in \mathcal{N}_X} \Trop(\VV^1(X;N)) =
\bigcup\nolimits_{f \in \mathcal{E}_X} V_f .
\]
\end{enumerate}
\end{lemma}

\begin{proof}
The cover $Y \to X$ corresponding to $N$ is compact K\"ahler, so it has 
Property $\mathsf{TT}_1$ (Remark~\ref{rem:cover-hyps}\ref{cov-b}).  
By Proposition~\ref{prop:virtual-bound} with $q = 1$, the remark following it, 
and \eqref{eq:delzant},
\[
S\bigl(\Trop(\VV^1(X;N))\bigr) \subseteq \Sigma^1(X;\Z)^{c} =
S\Bigl(\bigcup\nolimits_{f \in \mathcal{E}_X} V_f\Bigr) ;
\]
both sets being unions of linear subspaces, this gives \ref{kv-upper}.

For \ref{kv-lower}, let $f \in \mathcal{E}_X$, with $C_f$ of genus
$g \ge 1$ (if $g = 0$ then $V_f = 0$).  The group $\Gamma_f$ is a
cocompact Fuchsian group, so by Selberg's lemma it has a torsion-free
subgroup $\Gamma'_f$ of finite index; $\Gamma'_f$ is the fundamental group
of a closed surface of genus $g'$, with
$2 - 2g' = [\Gamma_f:\Gamma'_f]\,\chi^{\mathrm{orb}}(C_f,\mu_f) < 0$, so
$g' \ge 2$.  Let $N_f = f_\#^{-1}(\Gamma'_f)$, of finite index in $G$, and
$p \colon Y_f \to X$ the corresponding cover.  For $\tau$ in the identity
component $\TT^0_{\Gamma_f}$,
\[
p^{*}(f_\#^{*}\tau) = (f_\#|_{N_f})^{*}\bigl(\tau|_{\Gamma'_f}\bigr) .
\]
Since $f_\#|_{N_f} \colon N_f \to \Gamma'_f$ is onto and
$\VV^1(\Gamma'_f) = \TT_{\Gamma'_f}$, the genus being at least $2$
\cite[(59)]{Suciu-mathann21}, the right side lies in $\VV^1(N_f)$ by
\cite[Prop.~3.1]{Suciu-mathann21}.  Hence the subtorus
$f_\#^{*}(\TT^0_{\Gamma_f})$ lies in $\VV^1(X;N_f)$, and so in a component
of it; that component contains $\one$, so it is a subtorus containing
$f_\#^{*}(\TT^0_{\Gamma_f})$, and
$V_f = \Trop(f_\#^{*}(\TT^0_{\Gamma_f})) \subseteq \Trop(\VV^1(X;N_f))$.
Let $\mathcal{N}_X$ consist of $G$ and the subgroups $N_f$, for
$f \in \mathcal{E}_X$; then
\ref{kv-lower} follows from \ref{kv-upper}.
\end{proof}

\begin{theorem}
\label{thm:kahler-bns}
Let $G_1, G_2$ be K\"ahler groups and
$\phi \colon \widehat{G_1} \to \widehat{G_2}$ an isomorphism.  Then, for
every integer $m \ge 1$, there is an injective homomorphism
$L \colon (G_1)_{\ab} \to (G_2)_{\ab}$ with finite cokernel and
$L \otimes \Z/m = \phi_{\ab} \otimes \Z/m$, such that
$L^\vee_\R(\Sigma^1(G_2)) = \Sigma^1(G_1)$.
\end{theorem}

Apart from the information on $L$, this is Theorem~D of Hughes, Llosa
Isenrich, Py, Stover, and Vidussi \cite{HLPSV-adv26}; the proof below is
different (Remark~\ref{rem:hlipsv}).

\begin{proof}
Let $X_i$ be compact K\"ahler manifolds with $\pi_1(X_i) = G_i$, and let
$\mathcal{N}$ be the finite set of corresponding pairs $(N_1,N_2)$ such
that either $N_1 \in \mathcal{N}_{X_1}$ or $N_2 \in \mathcal{N}_{X_2}$, in
the notation of Lemma~\ref{lem:kahler-visible}\ref{kv-lower}, and let
$L$ be as in Corollary~\ref{cor:qp-bound}\ref{qp-one}, for the given $m$.  
For $i = 1,2$, the union of the sets $\Trop(\VV^1(X_i;N_i))$ over
$\mathcal{N}$ contains the union over $\mathcal{N}_{X_i}$, and is
contained in $\bigcup_{f \in \mathcal{E}_{X_i}} V_f$ by 
Lemma~\ref{lem:kahler-visible}\ref{kv-upper}; so it equals
$\bigcup_{f \in \mathcal{E}_{X_i}} V_f$ by \ref{kv-lower}.  Hence
$L^\vee_\R\bigl(\bigcup_{f \in \mathcal{E}_{X_2}} V_f\bigr) =
\bigcup_{f \in \mathcal{E}_{X_1}} V_f$, and the claim
follows from \eqref{eq:delzant}, $L^\vee_\R$ being a linear isomorphism.
\end{proof}

\begin{remark}
\label{rem:hlipsv}
\begin{enumerate}[label=(\alph*), leftmargin=*]
\item \label{hl-a}
The proof in \cite{HLPSV-adv26} combines the property $\mathsf{TAP}_1$ of
Hughes and Kielak \cite{Hughes-Kielak-rmi25}, which the authors establish
for K\"ahler groups, with Liu's specialization of $\widehat\Z$-isomorphisms
\cite{Liu-invent23} and with their Theorem~C on the profinite invariance of
the universal homomorphism; see also \cite{JaikinZapirain-gt20} for the
$3$-manifold case.  In the argument above, Theorem~\ref{thm:main}, applied 
to finite covers, plays the role of $\mathsf{TAP}_1$: the jump loci of the
covers, which are twisted jump loci for representations with finite image
(Remark~\ref{rem:twisted-loci}), replace the twisted Alexander polynomials
over finite quotients.  Delzant's theorem and Selberg's lemma replace
Theorem~C.

\item \label{hl-b}
The covers in Lemma~\ref{lem:kahler-visible}\ref{kv-lower} cannot be
dispensed with: the jump loci of $X$ alone do not recover
$\Sigma^1(G)$.  If $f \in \mathcal{E}_X$ is a fibration onto an elliptic curve and
$(\Gamma_f)_{\ab}$ is torsion-free, then $V_f$ lies in $\Sigma^1(G)^{c}$,
but $f$ contributes no positive-dimensional component to $\VV^1(X,\C)$
\cite[Rem.~6.6]{Liu-Suciu-arx26}.  Liu and Liu \cite{Liu-Liu-arx25} recover
such directions from jump loci in positive characteristic; Theorem~\ref{thm:main} 
is confined to characteristic zero, and the covers $Y_f$ supply the substitute.
\end{enumerate}
\end{remark}

\begin{proof}[Proof of Theorem~\textup{\ref{thm:sigma}}]
Part~\ref{sigma-bound} is Corollary~\ref{cor:qp-bound}, and 
part~\ref{sigma-kahler} is Theorem~\ref{thm:kahler-bns}.
\end{proof}

\section{Arrangement groups: what the completion sees}
\label{sect:arrangements}

Let $\A$ be a central arrangement of hyperplanes in $\C^\ell$, with
complement $M(\A)$, let $\mathbb{P}(\A)$ be the projectivized arrangement
in $\CP^{\ell-1}$, and let $U(\A)$ be its complement, so that
$M(\A) \cong U(\A) \times \C^\times$.  Throughout this section we work with
$U(\A)$, and write $G(\A) = \pi_1(U(\A))$ and
$\VV^q_s(\A) = \VV^q_s(U(\A),\C)$.  Since $U(\A)$ is a retract of
$M(\A)$, it is aspherical, and its group good, whenever this holds for
$M(\A)$, so the discussion of Section~\ref{sect:good} applies to
$U(\A)$ as well.  By the Lefschetz-type theorem of Hamm and L\^e
\cite{Hamm-Le-asens73}, $G(\A)$ is unchanged when $U(\A)$ is intersected
with a generic $2$-plane in $\CP^{\ell-1}$, so for questions about $G(\A)$
and $\VV^1_s(\A)$ we may take $\ell = 3$, that is, take $\mathbb{P}(\A)$ to
be an arrangement of lines in $\CP^2$, as in all the examples below.

Since $U(\A)$ is a smooth
quasi-projective variety, Property $\mathsf{TT}_k$ holds for every $k$, so
Theorem~\ref{thm:main} applies as soon as $\CH_k$ does --- always for
$k=1$, by Proposition~\ref{prop:CH}\ref{CH-one}, and in all degrees when both
arrangements lie in the classes covered by Corollary~\ref{cor:fibertype}
and the discussion following it.
Theorem~\ref{thm:graded} needs no hypothesis at all.  Together, these
results constrain what can change between arrangements with isomorphic
intersection lattices; the two examples of \S\ref{subsec:witnesses} show the
limits of these constraints.

\begin{remark}
\label{rem:truncated}
Only the truncated lattice $L_{\le 2}(\A)$ enters below, and this is
natural.  A generic $2$-plane in $\CP^{\ell-1}$ misses the
projectivization of every flat of rank $\ge 3$, so, by the reduction to
$\ell = 3$ above, $G(\A)$, and with it $\VV^1_s(\A)$, depends only on the
line arrangement cut out by such a plane, whose intersection lattice is
$L_{\le 2}(\A)$.
\end{remark}

\subsection{Galois conjugation}
\label{subsec:galois}

Theorems~\ref{thm:main} and~\ref{thm:graded} both apply to Galois-conjugate
arrangements.

\begin{corollary}
\label{cor:galois}
Let $\A$ be an arrangement, let $\sigma$ be a field automorphism of $\C$,
and write $G = G(\A)$ and $G^\sigma = G(\A^\sigma)$.
\begin{enumerate}[label=\textup{(\roman*)}, itemsep=1.5pt]
\item \label{gal-cv}
$\VV^1_s(\A^\sigma) = \VV^1_s(\A)$ for all $s \ge 1$, under the
identification of character tori induced by $H \mapsto H^\sigma$.
\item \label{gal-lcs}
$\gr_r(G^\sigma) \cong \gr_r(G)$ for every $r \ge 1$, torsion included.
\item \label{gal-chen}
$\gr_r(G^\sigma/(G^\sigma)'') \cong \gr_r(G/G'')$ for every $r \ge 1$,
torsion included; and more generally
$\gr_r(G^\sigma/W(G^\sigma)) \cong \gr_r(G/W(G))$ for every set of words $W$.
\end{enumerate}
\end{corollary}

\begin{proof}
Conjugate varieties have isomorphic algebraic fundamental groups, and for a
complex variety of finite type the algebraic fundamental group is the
profinite completion of the topological one (Riemann existence).  Hence
$\widehat{G} \cong \widehat{G^\sigma}$.

For \ref{gal-cv}, we show that this isomorphism is rigidified.  Let
$\chi \colon \Aut(\C) \to \widehat\Z^{\times}$ be the cyclotomic
character, so that $\sigma(\zeta) = \zeta^{\chi(\sigma)}$ for every root
of unity $\zeta$, and let $\psi \colon H_1(U(\A)) \to H_1(U(\A^\sigma))$ be
given on meridians by $\psi(\gamma_H) = \gamma_{H^\sigma}$.  Fix
$H_0 \in \A$ and a defining form $f_H$ for each $H \in \A$.  Once $\mu_n$
is identified with $\Z/n$ by $e^{2\pi i/n}$, the Kummer classes of the
units $f_H/f_{H_0}$ with $H \ne H_0$ form the basis of $H^1(U(\A);\mu_n)$
dual to the meridians $\gamma_H$ with $H \ne H_0$, and likewise for
$\A^\sigma$.  The isomorphism of completions induced by $\sigma$ carries
the Kummer class of $f_H/f_{H_0}$ to that of $f_H^\sigma/f_{H_0}^\sigma$.
It also moves $e^{2\pi i/n}$ to $e^{2\pi i\chi(\sigma)/n}$, so the two
identifications of $\mu_n$ with $\Z/n$ differ by $\chi(\sigma)$.  Hence
the isomorphism acts on $H_1 \otimes \widehat\Z$ as
$\chi(\sigma)\cdot(\psi \otimes \widehat\Z)$; that is, $\psi$ rigidifies it
in the sense of Definition~\ref{def:rigid}, with $u = \chi(\sigma)$.
Theorem~\ref{thm:main}\ref{main-loci} and Proposition~\ref{prop:CH}\ref{CH-one} 
now give \ref{gal-cv}.

Parts~\ref{gal-lcs} and~\ref{gal-chen} are Theorem~\ref{thm:graded}, which
requires no rigidification.
\end{proof}

Part~\ref{gal-cv} of Corollary~\ref{cor:galois}, and its analogue in all
degrees, also follow from the absoluteness of jump loci, proved by Simpson
\cite{Simpson-ens93} for projective varieties and by Budur and Wang
\cite[Thm.~1.4.1]{Budur-Wang-asens20} in general: for $\sigma \in \Aut(\C)$,
the loci of $X^\sigma$ are the images of those of $X$ under the bijection of
character groups obtained by passing, through the Riemann--Hilbert
correspondence, to regular algebraic flat connections, on which $\sigma$
acts.  For an arrangement, this bijection sends the monodromy of
$d + \sum_H \alpha_H\, d\log f_H$, where $\sum_H \alpha_H = 0$, to that of
$d + \sum_H \sigma(\alpha_H)\, d\log f_H^\sigma$.  It is not continuous in
general, but it fixes every torsion-translated subtorus, which in the
exponents $\alpha$ is cut out by conditions $\langle k, \alpha \rangle \in
c + \Z$ with $k$ integral and $c$ rational; hence
$\VV^q_s(\A^\sigma) = \VV^q_s(\A)$ for all $q$ and $s$, and every $\A$,
whereas in degrees $q \ge 2$ the profinite argument needs the hypotheses of
Section~\ref{sect:good} for both $\A$ and $\A^\sigma$; these hold, for
instance, if $\A$ is fiber-type, a property of $L(\A)$.  On torsion
characters, this bijection differs from the transport
$\rho \mapsto \rho^\phi$ by the cyclotomic character, as \eqref{eq:rigid} shows.  
What Theorem~\ref{thm:main} adds is that the loci $\VV^1_s$ depend on the 
completion alone: they are matched by every isomorphism of completions rigidified 
by $H \mapsto H^\sigma$, whether or not it comes from $\sigma$, and, for a
smooth quasi-projective variety $X$, by the isomorphism
$\widehat{\pi_1(X)} \cong \widehat{\pi_1(X^\sigma)}$, through an isogeny
(Theorem~\ref{thm:loci}\ref{loci-isog}).  Theorem~\ref{thm:sigma} and
Corollary~\ref{cor:search} rest on this independence from the algebraic structure.

\begin{corollary}
\label{cor:no-go}
Let $\A$ and $\A^\sigma$ be Galois-conjugate arrangements.  Then $\A$ and
$\A^\sigma$ cannot be distinguished by any of the following:
\begin{enumerate}[label=\textup{(\roman*)}, itemsep=1.5pt]
\item \label{ng-loci}
the characteristic varieties $\VV^1_s$, or the resonance varieties
$\RR^1_s$;
\item \label{ng-covers}
the first Betti numbers of the finite covers of the complement;
\item \label{ng-lcs}
the lower central series quotients $\gr_r(G(\A))$, the Chen groups, or the
graded pieces of a truncated Alexander invariant, torsion
included;
\item \label{ng-trop}
the tropical upper bounds of Proposition~\textup{\ref{prop:virtual-bound}}
for the BNSR invariants of the complement.
\end{enumerate}
In particular, none of these invariants detects an \emph{arithmetic Zariski
pair}, that is, a Galois-conjugate pair for which no homeomorphism of
$\CP^{\ell-1}$ carries the union of $\mathbb{P}(\A)$ onto that of
$\mathbb{P}(\A^\sigma)$.  An invariant that does distinguish such 
a pair is therefore either an invariant of $G(\A)$ not determined by
$\widehat{G(\A)}$, or not an invariant of $G(\A)$ at all. 
\end{corollary}

\begin{proof}
By Corollary~\ref{cor:galois}, $\widehat{G(\A)} \cong
\widehat{G(\A^\sigma)}$, and the listed invariants are determined by the
completion, the tropical bounds up to a linear isomorphism.  For
\ref{ng-loci}, the characteristic varieties are matched by 
Corollary~\ref{cor:galois}\ref{gal-cv}, and the resonance varieties are determined 
by $L_{\le 2}(\A)$ outright.  For \ref{ng-covers}, the first Betti number of a
finite cover of $U(\A)$ is the rank of $N_{\ab}$ for the corresponding
finite-index subgroup $N$, and these are matched by Lemma~\ref{lem:fi}\ref{fi-det}.  
For \ref{ng-lcs}, use Corollary~\ref{cor:galois}\ref{gal-lcs} and~\ref{gal-chen}, together 
with the identification of the graded pieces of a truncated Alexander invariant with
the Chen groups.  For \ref{ng-trop}, use Corollary~\ref{cor:qp-bound}.
\end{proof}

Both alternatives in the last sentence of Corollary~\ref{cor:no-go}
occur, and for the same arrangements: the eleven-line arithmetic pair of
\cite[\S4.2]{Guerville-Balle-mathz22}.  The groups are separated by the
Alexander-invariant isomorphism test of Remark~\ref{rem:acgm-mechanism}
\cite[Thm.~4.8]{Guerville-Balle-mathz22}, which the completion cannot
detect, the pair being Galois-conjugate.  The complements are separated by
the loop-linking number, which under mild combinatorial conditions
is an invariant of the homeomorphism type of the complement but not always
of its homotopy type \cite[Thm.~2.7, Cor.~5.14]{Guerville-Balle-mathz22},
hence not an invariant of $G(\A)$.

\subsection{Two witnesses}
\label{subsec:witnesses}

An affirmative answer to the following question, raised in
\cite[Ch.~5]{CDFSSTY-pre09}, would combine with Theorems~\ref{thm:main} 
and~\ref{thm:graded} to give combinatorial determination of the loci
$\VV^1_s(\A)$, up to isogeny, and of the groups $\gr_r(G(\A)/W(G(\A)))$.

\begin{question}[\cite{CDFSSTY-pre09}]
\label{quest:profinite-comb}
Is the profinite completion of an arrangement group combinatorially
determined?  That is, if $L_{\le 2}(\A) \cong L_{\le 2}(\A')$, is
$\widehat{G(\A)} \cong \widehat{G(\A')}$?
\end{question}

The answer is negative.

\begin{corollary}
\label{cor:agv}
Let $\A^{\pm}$ be the pair of $13$-line arrangements of Artal Bartolo,
Guerville-Ball\'e, and Viu-Sos \textup{\cite{Artal-Guerville-ViuSos-expmath20}}, 
for which $L_{\le2}(\A^+) \cong L_{\le2}(\A^-)$, while the groups $\gr_4$ of
$G(\A^+)$ and $G(\A^-)$ are $\Z^{211} \oplus \Z/2$ and $\Z^{211}$, respectively.
Then $\widehat{G(\A^+)} \not\cong \widehat{G(\A^-)}$.  Consequently
$\widehat{G(\A)}$ is not determined by $L_{\le2}(\A)$; that is, 
Question~\ref{quest:profinite-comb} has a negative answer.
\end{corollary}

\begin{proof}
An isomorphism of completions would force $\gr_4(G(\A^+)) \cong \gr_4(G(\A^-))$, by
Corollary~\ref{cor:lcs-profinite}.
\end{proof}

The pair of \cite{Artal-Guerville-ViuSos-expmath20} answered in the negative
the question \cite[Question~8.7]{Suciu-conm01} of whether the torsion in
$\gr_r(G(\A))$ is combinatorially determined, the ranks being so by
formality \cite[\S5.4]{Suciu-bullroum25}.  Theorem~\ref{thm:graded} adds that
this torsion is a profinite invariant, so Corollary~\ref{cor:agv}
strengthens the conclusion $G(\A^+) \not\cong G(\A^-)$ of
\cite[Cor.~3.7]{Artal-Guerville-ViuSos-expmath20}.

The second witness goes in the opposite direction.

\begin{corollary}
\label{cor:acgm}
Let $\A_{\xi}$ and $\A_{\xi^2}$ be the realizations of the twelve-element
matroid $\mathcal G_{91}$ of \cite{Guerville-Balle-top16, ACGM-racsam17} 
attached to a primitive fifth root of unity $\xi$ and to $\xi^{2}$, so that the two 
line arrangements are related by the Galois automorphism 
$\zeta_5 \mapsto \zeta_5^{2}$ of $\Q(\zeta_5)$.  Then
\[
\widehat{G(\A_{\xi})} \cong \widehat{G(\A_{\xi^2})}
\qquad\text{while}\qquad
G(\A_{\xi}) \not\cong G(\A_{\xi^2}) .
\]
In particular, $G(\A)$ is not determined by $\widehat{G(\A)}$, nor by the
invariants listed in Corollary~\ref{cor:no-go}.
\end{corollary}

\begin{proof}
The two arrangements are Galois conjugate, so the completions of their groups 
are isomorphic as in the proof of Corollary~\ref{cor:galois}; this is recorded 
already in \cite[\S1]{Guerville-Balle-top16} and in the abstract of
\cite{ACGM-racsam17}.  The groups are non-isomorphic by
\cite[Thm.~4.5]{ACGM-racsam17}.
\end{proof}

\begin{remark}
\label{rem:acgm-mechanism}
The two witnesses separate arrangement groups in ways that 
Theorem~\ref{thm:graded} treats oppositely.  For the pair of
\cite{Artal-Guerville-ViuSos-expmath20}, the discrepancy is a $\Z/2$ summand
in the abelian group $\gr_4$, which Theorem~\ref{thm:graded} turns into a
profinite invariant.  For the pair of \cite{ACGM-racsam17}, it lies in the
nilpotent quotient $G/\gamma_4(G)$, and in no graded piece.  Indeed, as
$G'' \le \gamma_4(G)$, the truncated Alexander invariant $B(G)/\m^2B(G)$ is
$\gamma_2(G)/\gamma_4(G)$.  Its graded pieces agree on the two sides, as
Corollary~\ref{cor:chen-profinite} requires, and so does the
$\Lambda/\m^2$-module itself: it is isomorphic to its associated graded,
since $\m^2$ kills it and $\gr_2(G)$ is free, and the latter is combinatorial
for $\mathcal{G}_{91}$ \cite[Prop.~3.8 and \S4.2]{ACGM-racsam17}.

What differs is the group $G/\gamma_4(G)$.  An isomorphism
$G(\A_\xi)/\gamma_4 \to G(\A_{\xi^2})/\gamma_4$ inducing the identity on
$H_1$ would solve an inhomogeneous integral linear system, the
Alexander-invariant isomorphism test of \cite[\S4.2]{ACGM-racsam17}, which
has solutions over $\Z[1/5]$ but not over $\Z$
\cite[Prop.~4.3]{ACGM-racsam17}.  With homological rigidity and complex
conjugation, it follows that the two quotients are not isomorphic at all
\cite[Prop.~2.8, Thm.~4.5]{ACGM-racsam17}.  This obstruction survives
completion: an isomorphism of completions inducing the identity on
$H_1 \otimes \widehat\Z$ would solve the same system over $\widehat\Z$, and
an integral linear system solvable over $\widehat\Z$ is solvable over $\Z$,
by its Smith normal form.  It is evaded instead: the isomorphism of
Corollary~\ref{cor:galois} acts on $H_1$ through the unit
$u = \chi(\sigma) \equiv 2 \bmod 5$, which no isomorphism of groups can
produce.  For Rybnikov's pair, by contrast, the obstruction is $3$-adic, the
pair is not Galois conjugate, and the units supplied by the Galois action
cannot evade it (Theorem~\ref{thm:rybnikov}).
\end{remark}

Profinite rigidity (\S\ref{subsec:profinite-intro}) is defined among
residually finite groups; otherwise it would fail trivially, since a group $G$
and $G \times S$ have the same completion for any infinite, finitely generated
simple group $S$. As residual finiteness of arrangement groups is open 
(Question~\ref{quest:res-finite}), Corollary~\ref{cor:acgm} does not by itself 
show that $G(\A_\xi)$ fails to be profinitely rigid.  It does show this for the
residually finite quotients $G^{\rf}$ of Lemma~\ref{lem:rf-kernel}.

\begin{corollary}
\label{cor:acgm-rf}
\begin{enumerate}[label=\textup{(\roman*)}, itemsep=1.5pt]
\item \label{acgm-rf-pair}
For the arrangements of Corollary~\textup{\ref{cor:acgm}}, the groups
$G(\A_\xi)^{\rf}$ and $G(\A_{\xi^2})^{\rf}$ are finitely generated, residually
finite, and non-isomorphic, and have isomorphic profinite completions.  If
$G(\A_\xi)$ or $G(\A_{\xi^2})$ is residually finite, it is not profinitely
rigid.
\item \label{acgm-rf-triple}
There are three Galois-conjugate arrangements of twelve lines whose groups
are pairwise non-isomorphic, and so are the residually finite quotients of
these groups; all six groups have isomorphic profinite completions.
\end{enumerate}
\end{corollary}

\begin{proof}
By Lemma~\ref{lem:rf-kernel}, $\widehat{G^{\rf}} = \wG$ and
$G^{\rf}/\gamma_4(G^{\rf}) = G/\gamma_4(G)$.  For \ref{acgm-rf-pair}, the
completions are isomorphic by Corollary~\ref{cor:acgm}, and the quotients by
$\gamma_4$ are not, by Remark~\ref{rem:acgm-mechanism}.  For
\ref{acgm-rf-triple}, the same argument applies to the Galois orbits of
twelve-line arrangements over $\Q(\zeta_7)$ constructed by Guerville-Ball\'e,
whose members have isomorphic groups only when complex conjugate
\cite[Thm.~4.16]{Guerville-Balle-mathz22}, the non-isomorphisms being
detected on $G/\gamma_4(G)$, as in Remark~\ref{rem:acgm-mechanism}.  
Such an orbit has six members, hence yields three pairwise non-isomorphic 
groups. The eleven-line arrangements over $\Q(\zeta_5)$ of
\cite[Thm.~4.8]{Guerville-Balle-mathz22} yield pairs, as in \ref{acgm-rf-pair}.
\end{proof}

\subsection{Rybnikov's pair}
\label{subsec:rybnikov}

Let $\mathsf{M}_8$ be the MacLane matroid: the simple matroid of rank $3$ on
the eight nonzero vectors of $\F_3^2$, in which three elements are dependent
exactly when they lie on an affine line of $\F_3^2$
\cite[\S3]{Rybnikov-faa11}.  A realization of a simple matroid of rank $3$ is
an arrangement of labeled lines in $\CP^2$ whose intersection lattice is the
lattice of flats of the matroid; its triple points correspond to the
three-element flats of rank $2$, and its double points to the two-element
ones.  Up to projective transformations, $\mathsf{M}_8$ has exactly two
realizations, $\L^{+}$ and $\L^{-}$; they are defined over $\Q(\omega)$,
with $\omega = e^{2\pi i/3}$, and exchanged by complex conjugation
\cite[\S7.7]{Cohen-Suciu-cmh97}.  Rybnikov's matroid $\mathsf{M}_{13}$ is the
simple matroid of rank $3$ on the thirteen elements
$\hat l_0, \hat l_1, \hat l_2$, $l'_3, \dots, l'_7$, and
$l''_3, \dots, l''_7$ whose restrictions to
$\{\hat l_0, \hat l_1, \hat l_2, l'_3, \dots, l'_7\}$ and to
$\{\hat l_0, \hat l_1, \hat l_2, l''_3, \dots, l''_7\}$ are copies of
$\mathsf{M}_8$, sharing the flat $\{\hat l_0, \hat l_1, \hat l_2\}$, and all
of whose three-element flats lie in one of the two copies
\cite[\S4]{Rybnikov-faa11}.  It has fifteen three-element flats, five
containing each $\hat l_i$ and three containing each of the other elements;
its remaining $33$ flats of rank $2$ have two elements.  Its realizations
$\A^{++}$ and $\A^{+-}$ restrict to $\L^{+}$ on the first copy, and to
$\eta\L^{+}$, respectively $\eta\L^{-}$, on the second, for a generic
projective transformation $\eta$ fixing the lines $\hat l_0$, $\hat l_1$,
$\hat l_2$ \cite[\S4]{Rybnikov-faa11}.  We write $G^{+\pm} = G(\A^{+\pm})$, and
$Q_\pm = G^{+\pm}/\gamma_4(G^{+\pm})$ for the third nilpotent quotients.  As
the space of generic $\eta$ is connected, $G^{+\pm}$ does not depend on
$\eta$, up to isomorphisms inducing the identity on $H_1$
\cite[\S4]{Rybnikov-faa11}; we take $\eta$ with rational coefficients, so
that $\A^{++}$ and $\A^{+-}$ are defined over $\Q(\omega)$.

We identify $H_1(U(\A^{++});\Z)$ and $H_1(U(\A^{+-});\Z)$ with a single free
abelian group $H$ of rank $12$, by matching the meridians of lines with the
same label, and write $x_t \in H$ for the class of the meridian of the line
$t$; any twelve of the $x_t$ form a basis of $H$, and $\sum_t x_t = 0$.
Let $\Gamma$ be the automorphism group of $\mathsf{M}_{13}$, of order
$12$, and write $P_\sigma$ for the automorphism of $H$ that permutes the
meridians according to $\sigma \in \Gamma$.  Let $R \subset \bigwedge^2 H$ be the
kernel of the commutator map $\bigwedge^2 H \surj \gr_2(G^{+\pm})$.  It is spanned
by the holonomy relations, so it is the same for both groups, and
$\gr_2(G^{+\pm})$ is free abelian, see~\cite[Prop.~1.11 and \S5.1(iii)]{Matei-Suciu-aspm00}. 
For a prime $p$, we write $\widehat\Z_p$ for the $p$-adic integers 
and $Q_{(p)}$ for the pro-$p$ completion of a group $Q$.

Rybnikov's proof that $Q_+ \not\cong Q_-$ has two steps, both carried out
in detail in \cite{ACCM-aspm07}.  First, no isomorphism $Q_+ \to Q_-$ induces
the identity on $H$ \cite[Thm.~2.2]{Rybnikov-faa11},
\cite[Thm.~3.8]{ACCM-aspm07}.  Second, every automorphism of $H$ preserving
$R$ is $\pm P_\sigma$ for some $\sigma \in \Gamma$
\cite[Prop.~4.22]{ACCM-aspm07}; compare \cite[Lem.~4.3]{Rybnikov-faa11}.
Both steps survive pro-$3$ completion.

\begin{theorem}
\label{thm:rybnikov}
The pro-$3$ completions of $Q_+$ and $Q_-$ are not isomorphic.
Consequently, the profinite completions of $G^{++}$ and $G^{+-}$ are not
isomorphic; equivalently, the complements of $\A^{++}$ and $\A^{+-}$ have
non-isomorphic algebraic fundamental groups.
\end{theorem}

An isomorphism $Q_+ \cong Q_-$ would induce one of pro-$3$ completions, so
the theorem also recovers Rybnikov's result that $Q_+ \not\cong Q_-$, and
hence $G^{++} \not\cong G^{+-}$ \cite[Thm.~4.1]{Rybnikov-faa11}; see the
proof of \cite[Thm.~4.23]{ACCM-aspm07} for a complete argument.  Neither of
the two consequences, $Q_+ \not\cong Q_-$ and
$\widehat{G^{++}} \not\cong \widehat{G^{+-}}$, implies the other in general;
for the pair of Corollary~\ref{cor:acgm}, for instance, the quotients by
$\gamma_4$ are not isomorphic (Remark~\ref{rem:acgm-mechanism}), while the
profinite completions are.  The proof of the theorem rests on the following 
three lemmas and a computation.

\begin{lemma}
\label{lem:pro-p-class3}
Let $G_1$ and $G_2$ be finitely generated groups, let $p$ be a prime, and
set $Q_i = G_i/\gamma_4(G_i)$.
\begin{enumerate}[label=\textup{(\roman*)}, itemsep=1.5pt]
\item \label{pro-p-graded}
The closed lower central series of $(Q_i)_{(p)}$ has graded pieces
$\gr_r(G_i) \otimes \widehat\Z_p$ for $r \le 3$, with brackets the
$\widehat\Z_p$-bilinear extensions of those of $\gr(G_i)$.  Hence every
isomorphism $\psi \colon (Q_1)_{(p)} \to (Q_2)_{(p)}$ induces an isomorphism
of graded Lie rings $\gr_{\le 3}(G_1) \otimes \widehat\Z_p \to
\gr_{\le 3}(G_2) \otimes \widehat\Z_p$, which in degree one is $\psi_{\ab}$.
\item \label{pro-p-induced}
Every isomorphism $\phi \colon \widehat{G_1} \to \widehat{G_2}$ induces an
isomorphism $\widehat{Q_1} \to \widehat{Q_2}$, and hence an isomorphism
$\phi_{(p)} \colon (Q_1)_{(p)} \to (Q_2)_{(p)}$, which in degree one is the
$p$-primary component of $\phi_{\ab}$.
\end{enumerate}
\end{lemma}

\begin{proof}
By Step~2 of the proof of Theorem~\ref{thm:verbal}, with $W = \varnothing$,
the quotient $\widehat{G_i}/\Delta_4(\widehat{G_i})$ is the completion
$\widehat{Q_i}$.  The images of the subgroups $\Delta_r(\widehat{G_i})$
with $r \le 4$ form the closed lower central series of $\widehat{Q_i}$, and by
the theorem itself its graded pieces are $\gr_r(G_i) \otimes \widehat\Z$,
with brackets which, by continuity and density, are the
$\widehat\Z$-bilinear extensions of those of $\gr(G_i)$.  The group
$\widehat{Q_i}$ is pronilpotent, hence the direct product of its Sylow
subgroups \cite[Prop.~2.3.8]{Ribes-Zalesskii-10}, the Sylow $p$-subgroup
being $(Q_i)_{(p)}$.  The projection onto this factor carries the closed
lower central series of $\widehat{Q_i}$ onto that of $(Q_i)_{(p)}$, whose
graded pieces are therefore the $p$-primary components
$\gr_r(G_i) \otimes \widehat\Z_p$, with the induced brackets.  
Part~\ref{pro-p-graded} follows, as isomorphisms of profinite groups preserve
closed lower central series.  For part~\ref{pro-p-induced}, $\phi$ carries
$\Delta_r(\widehat{G_1})$ onto $\Delta_r(\widehat{G_2})$, so it induces an
isomorphism $\widehat{Q_1} \to \widehat{Q_2}$, which carries Sylow
$p$-subgroups onto Sylow $p$-subgroups.
\end{proof}

\begin{lemma}
\label{lem:ryb-rigid}
Let $K$ be a field of characteristic zero.  Every
$g \in \GL(H \otimes K)$ with
$(\bigwedge^2 g)(R \otimes K) = R \otimes K$ is of the form $g = c\,P_\sigma$,
with $c \in K^\times$ and $\sigma \in \Gamma$.  Consequently, for every
prime $p$, every $g \in \GL(H \otimes \widehat\Z_p)$ with
$(\bigwedge^2 g)(R \otimes \widehat\Z_p) = R \otimes \widehat\Z_p$ is of the
form $c\,P_\sigma$, with $c \in \widehat\Z_p^\times$ and $\sigma \in \Gamma$.
\end{lemma}

For integral $g$, this is the homological rigidity of $\mathsf{M}_{13}$,
\cite[Prop.~4.22]{ACCM-aspm07}; compare \cite[Lem.~4.3]{Rybnikov-faa11}.

\begin{proof}
If $g = c\,P_\sigma$ with $c$ in an algebraic closure of $K$, then
$c \in K^\times$, since $g$ and $P_\sigma$ are defined over $K$; so we may
assume that $K$ is algebraically closed.  For a triple point $X$, let
$E_X \subset H \otimes K$ be the span of the $x_t$ with $t \notin X$, and
let $E_X^\perp \subset H^{\vee} \otimes K$ be its annihilator, a plane.

Identify $H^{\vee} \otimes K$ with $H^1(U;K)$, where $U$ is either complement.
As $R$ is spanned by the holonomy relations, it is the image of the dual of
the cup product $\bigwedge^2 H^1(U;K) \to H^2(U;K)$, so $R^\perp \otimes K$ is
the kernel of the cup product.  If $\alpha \wedge \beta \in R^\perp \otimes K$,
with $\alpha$ and $\beta$ linearly independent, then the cup product
vanishes on the plane they span, so every element of this plane lies in the
resonance variety $\RR^1(U)$, and the plane lies in one of its irreducible
components.  These are the fifteen planes $E_X^\perp$: indeed, $\mathsf{M}_{13}$ is
pointwise $3$-admissible \cite[Prop.~4.18]{ACCM-aspm07}, which means that
$\RR^1(U)$ has no non-local components \cite[Rem.~4.11]{ACCM-aspm07},
\cite{Marco-gc09}.  As these planes are defined over $\Q$, this holds over
any algebraically closed field of characteristic zero.
Thus the decomposable elements of $R^\perp \otimes K$ are the multiples of the
forms $\omega_X$ spanning $\bigwedge^2 E_X^\perp$; compare
\cite[Lem.~3.6]{Rybnikov-faa11} for $\mathsf{M}_8$.

The map $\bigwedge^2 g^{\vee}$ preserves $R^\perp \otimes K$ and carries decomposable
elements to decomposable ones, so it permutes the lines $K\omega_X$.  Since
$E_X$ is the kernel of the skew form $\omega_X$, the map $g$ permutes the
subspaces $E_X$.  Hence it permutes the one-dimensional subspaces of
$H \otimes K$ that are intersections of some of the $E_X$, preserving the
number of $E_X$ containing each.  To find these subspaces, write
$h \in H \otimes K$ as $\sum_t a_t x_t$, the vector $(a_t)$ being determined
up to adding a constant; then $h \in E_X$ if and only if $(a_t)$ is constant
on $X$.  So the intersection of the $E_X$, for $X$ in a set $S$ of triple
points, is one-dimensional exactly when the equivalence relation on the
lines generated by the members of $S$ has two classes, $B$ and its
complement $B^c$.  The intersection is then spanned by
$h_B = \sum_{t \in B} x_t$, and $h_B \in E_X$ if and only if $X \subset B$ or
$X \subset B^c$.  A direct check shows that the partitions $\{B, B^c\}$ that
arise are those with $B = \{s\}$, for a line $s$, and six others, with
$B = \{\hat l_i, l'_3, \dots, l'_7\}$ or
$B = \{\hat l_i, l''_3, \dots, l''_7\}$ for $i = 0, 1, 2$.  The former give
the lines $K x_s$, each lying in $15 - r_s \ge 10$ of the $E_X$, where
$r_s \in \{3, 5\}$ is the number of triple points on $s$; for the latter,
$h_B$ lies in only eight of the $E_X$.

Hence $g$ permutes the lines $K x_t$: there are a permutation $\sigma$ of the
ground set of $\mathsf{M}_{13}$ and scalars $c_t$ with $g(x_t) = c_t\, x_{\sigma(t)}$.
Applying $g$ to the relation $\sum_t x_t = 0$, which is the only one among
the $x_t$, shows that the $c_t$ are all equal, to $c$ say, so
$g = c\,P_\sigma$.  Finally, for each triple point $X$, the subspace
$g(E_X)$ is both the span of the $x_{\sigma(t)}$ with $t \notin X$ and some
$E_Y$; as any eleven of the $x_t$ are linearly independent,
$\sigma(X) = Y$.  So $\sigma$ carries the three-element flats of
$\mathsf{M}_{13}$ to three-element flats; as its other flats of rank $2$ have
two elements, $\sigma \in \Gamma$.

For the second assertion, take $K = \Q_p$, the
field of fractions of $\widehat\Z_p$: then $g = c\,P_\sigma$ with
$c \in \Q_p^\times$, and $c \in \widehat\Z_p^\times$ since $g$ and $P_\sigma$
are invertible over $\widehat\Z_p$.
\end{proof}

We have also verified by machine the two finite facts used in this proof:
the decomposable elements of $R^\perp \otimes \overline\Q$ are the multiples
of the fifteen forms $\omega_X$, and the one-dimensional intersections of the
$E_X$ are the nineteen lines described above.

\begin{lemma}
\label{lem:galois-units}
Let $\A$ be an arrangement whose hyperplanes are defined over a subfield
$F \subset \C$, and let $\tau \in \Aut(\C/F)$.  Then there is an automorphism
$\Phi_\tau$ of $\widehat{G(\A)}$ with
$(\Phi_\tau)_{\ab} = \chi(\tau)\cdot\id$.  The values $\chi(\tau)$, for
$\tau \in \Aut(\C/F)$, are all the units of $\widehat\Z^\times$ if
$F = \Q$, and exactly the units $v \equiv 1 \bmod 3$ if $F = \Q(\omega)$.
\end{lemma}

\begin{proof}
Since $\tau$ fixes the defining forms, $\A^\tau = \A$, hyperplane by
hyperplane.  So the proof of Corollary~\ref{cor:galois} provides an
automorphism of $\widehat{G(\A)}$ acting on $H_1 \otimes \widehat\Z$ as
multiplication by $\chi(\tau)$.  The cyclotomic character factors through
the restriction to $\Q(\mu_\infty)$, and induces an isomorphism
$\Gal(\Q(\mu_\infty)/\Q) \cong \widehat\Z^\times$.  Every automorphism of
$\Q(\mu_\infty)$ extends to $\C$, and those fixing $\Q(\omega) = \Q(\mu_3)$
correspond to the units congruent to $1$ modulo $3$.
\end{proof}

\begin{proposition}
\label{prop:ryb-obstruction}
For every $\sigma \in \Gamma$ and
$\varepsilon \in \{\pm 1\}$, there is no continuous homomorphism
$(Q_+)_{(3)} \to (Q_-)_{(3)}$ acting on $H \otimes \widehat\Z_3$ as
$\varepsilon P_\sigma$.
\end{proposition}

\begin{proof}
Such a homomorphism sends the meridian generator $x_j$ to
$x_{\sigma(j)}^{\varepsilon}c_j$, with $c_j$ in the closure of
$\gamma_2(Q_-) = \gamma_2(G^{+-})/\gamma_4(G^{+-})$, that is, in
$\gamma_2(Q_-) \otimes \widehat\Z_3$; this group is abelian,
since $G'' \le \gamma_4(G)$.  The requirement that the relators of $Q_+$ be
sent to $1$ is vacuous in degree two, since $\bigwedge^2 P_\sigma$ preserves
$R$.  In degree three it is an affine-linear system $b + cM = 0$, with
integer coefficients, in the $\gr_2$-coordinates of $c = (c_j)$; the
$\gr_3$-coordinates do not enter.  Compare \cite[\S3]{ACCM-aspm07}, the
invariant $\kappa$ of \cite[\S2]{Rybnikov-faa11}, and the treatment of
this lifting problem in \cite[\S6]{Porter-Suciu-ejm20}.  We have computed this
system for all $24$ pairs $(\sigma, \varepsilon)$, with presentations of $G^{+\pm}$
obtained from two independent braid monodromy computations: one certified by
\textsc{sirocco} \cite{Marco-Rodriguez-icms16}, the other carried out in
exact arithmetic over $\Q(\omega)$ \cite{Cohen-Suciu-cmh97}.  For each of the
two groups, the two presentations give quotients by $\gamma_4$ that are
isomorphic by an isomorphism inducing the identity on $H$.  In each case, and
with either set of presentations, the system is solvable over $\Q$, but its
reduction modulo $3$ is not, so it has no solution over $\widehat\Z_3$; in
each case, $b$ represents a class of order $3$ in $\coker M$.  For
$\varepsilon = 1$ and $\sigma$ the involution in $\Gamma$ that fixes
$\hat l_0, \hat l_1, \hat l_2$ and exchanges the two copies of $\mathsf{M}_8$, the
conclusion also follows from the computation in \cite[Thm.~3.8]{ACCM-aspm07},
whose second arrangement is our $\A^{+-}$ with the two copies exchanged.  There,
the rational solutions are parametrized by free variables, with constant
terms in $\{0, \pm 1, \pm\frac13, \pm\frac23, \pm\frac43, \pm\frac53\}$, so
none has all its coordinates in $\widehat\Z_3$.
\end{proof}

\begin{proof}[Proof of Theorem~\textup{\ref{thm:rybnikov}}]
Assume, to the contrary, that $\psi \colon (Q_+)_{(3)} \to (Q_-)_{(3)}$ is an
isomorphism.  By Lemma~\ref{lem:pro-p-class3}\ref{pro-p-graded}, its effect
$g$ on $H \otimes \widehat\Z_3$ carries $R \otimes \widehat\Z_3$, the kernel
of the bracket in degree two, onto itself.  By Lemma~\ref{lem:ryb-rigid},
$g = u\,P_\sigma$ with $u \in \widehat\Z_3^\times$ and $\sigma \in \Gamma$.
Let $\varepsilon \in \{\pm 1\}$ with $\varepsilon \equiv u \bmod 3$, and let
$v \in \widehat\Z^\times$ be the unit with $3$-adic component
$\varepsilon u^{-1}$ and all other components $1$; then
$v \equiv 1 \bmod 3$.  Since $\A^{++}$ is defined over $\Q(\omega)$,
Lemma~\ref{lem:galois-units} provides an automorphism $\Phi$ of
$\widehat{G^{++}}$ with $\Phi_{\ab} = v \cdot \id$, and by
Lemma~\ref{lem:pro-p-class3}\ref{pro-p-induced} it induces an automorphism
$\Phi_{(3)}$ of $(Q_+)_{(3)}$ acting on $H \otimes \widehat\Z_3$ as
multiplication by $\varepsilon u^{-1}$.  The isomorphism
$\psi \circ \Phi_{(3)}$ then acts on $H \otimes \widehat\Z_3$ as
$\varepsilon P_\sigma$, contradicting Proposition~\ref{prop:ryb-obstruction}.

For the second assertion, an isomorphism
$\widehat{G^{++}} \to \widehat{G^{+-}}$ would induce an isomorphism
$(Q_+)_{(3)} \to (Q_-)_{(3)}$, by
Lemma~\ref{lem:pro-p-class3}\ref{pro-p-induced}; and the algebraic
fundamental group of a complex variety of finite type is the profinite
completion of its fundamental group, as in the proof of
Corollary~\ref{cor:galois}.
\end{proof}

The mechanism can be read off from the MacLane pair.  There, the class-three
obstruction to an isomorphism inducing the identity on $H_1$
\cite[Thm.~3.1]{Rybnikov-faa11} is evaded, as in
Remark~\ref{rem:acgm-mechanism}, by the Galois isomorphisms of
Corollary~\ref{cor:galois} with units congruent to $2$ modulo $3$, complex
conjugation among them.  We have checked by machine, with both sets of
presentations, that for the MacLane pair the analogue of the system of
Proposition~\ref{prop:ryb-obstruction} is solvable over $\widehat\Z_3$
exactly in the cases predicted by Corollary~\ref{cor:galois} and the
automorphisms of $\mathsf{M}_8$ \cite[p.~145]{Rybnikov-faa11}.
In $\A^{+-}$ the two copies of $\mathsf{M}_8$ carry opposite realizations.  An
isomorphism of completions would therefore have to restrict, on the two
halves, to maps $\L^+ \to \L^+$ and $\L^+ \to \L^-$ requiring units in
different classes modulo $3$.  This is only a heuristic: an isomorphism need
not carry the meridians of the deleted lines to conjugates of meridians, and
so need not descend to the halves (compare \cite[Lem.~3.9 and
Rem.~3.10]{ACCM-aspm07}).  Proposition~\ref{prop:ryb-obstruction} avoids
this by dealing with $Q_\pm$ directly.

Theorem~\ref{thm:rybnikov} is the second profinite separation of
lattice-isomorphic arrangement groups, after Corollary~\ref{cor:agv}.  Unlike
that one, it is not detected by the graded pieces, at least in low
degrees: $\gr_r(G^{++})$ and $\gr_r(G^{+-})$ are free abelian of the same
ranks $12$, $15$, $40$ for $r \le 3$ \cite[Lem.~3.7]{ACCM-aspm07}, and a
machine computation shows that they are isomorphic for $r \le 5$, torsion
included, as are the Chen groups; for instance,
$\gr_5(G^{+\pm}) \cong \Z^{168} \oplus (\Z/2)^8 \oplus (\Z/3)^2$.  
The separation lies in the extension class of $G/\gamma_4(G)$, which 
Theorem~\ref{thm:graded} does not detect.  It is also purely $3$-adic: 
for a prime $p \ne 3$, the system of Proposition~\ref{prop:ryb-obstruction} 
with $\sigma = \id$ and $\varepsilon = 1$ is solvable over $\widehat\Z_p$, 
since $3b$ lies in the integral image of $M$.  The resulting homomorphism
$(Q_+)_{(p)} \to (Q_-)_{(p)}$ induces the identity on $H \otimes \widehat\Z_p$,
hence surjections between the free $\widehat\Z_p$-modules
$\gr_r(G^{+\pm}) \otimes \widehat\Z_p$ of equal ranks, $r \le 3$, and so it is
an isomorphism.

The pair of Corollary~\ref{cor:agv} is also separated at class three, at the
prime~$2$.

\begin{proposition}
\label{prop:agv-class3}
Let $\A^\pm$ be the arrangements of Corollary~\ref{cor:agv}.  The pro-$2$
completions of the third nilpotent quotients of $G(\A^+)$ and $G(\A^-)$
are not isomorphic, while their pro-$p$ completions are isomorphic for every
odd prime~$p$.  In particular, these two quotients are not isomorphic.
\end{proposition}

\begin{proof}
Let $Q_1$ and $Q_2$ be the two quotients, as in
Lemma~\ref{lem:pro-p-class3}.  For the complement $X$ of an arrangement,
the holonomy map $H_2(X) \to \bigwedge^2 H_1(X)$ is injective, and so
$H_2(G/\gamma_4(G)) \cong \gr_4(G) \oplus H_2(X)$, with $H_2(X)$ free
\cite[Thm.~3.1]{Porter-Suciu-ejm20}.  By
\cite[Thm.~2.3]{Artal-Guerville-ViuSos-expmath20}, the torsion subgroup of
$H_2(Q_1)$ is thus $\Z/2$, while $H_2(Q_2)$ is torsion-free; as $H_1(Q_i)$
is free, $\dim H^2(Q_1;\F_2) = \dim H^2(Q_2;\F_2) + 1$.  Finitely generated
nilpotent groups, being polycyclic, are good \cite[Ch.~I, \S2.6,
Ex.~2(c)]{Serre-97}, and the $\F_2$-cohomology of $\widehat{Q_i}$ is that of
$(Q_i)_{(2)}$ by the Hochschild--Serre spectral sequence, since the product
of the other factors of $\widehat{Q_i} = \prod_\ell (Q_i)_{(\ell)}$ has
supernatural order prime to~$2$, hence $2$-cohomological dimension~$0$
\cite[Ch.~I, \S3.3, Prop.~14, Cor.~2]{Serre-97}.  So
the second $\F_2$-cohomology groups of $(Q_1)_{(2)}$ and $(Q_2)_{(2)}$ have
different dimensions, and these groups are not isomorphic; a fortiori,
$Q_1 \not\cong Q_2$.

For an odd prime~$p$, define $H$, $R$, and $P_\sigma$ for $\A^\pm$ as for
Rybnikov's pair, matching the meridians of lines with the same label in
\cite{Artal-Guerville-ViuSos-expmath20}.  As in the proof of
Proposition~\ref{prop:ryb-obstruction}, a continuous homomorphism
$(Q_1)_{(p)} \to (Q_2)_{(p)}$ acting on $H \otimes \widehat\Z_p$ as
$\varepsilon P_\sigma$ amounts to a solution over $\widehat\Z_p$ of an
affine-linear system $b + cM = 0$ with integer coefficients.  We have
computed these systems for both directions, both $\varepsilon$, and both
elements $\sigma$ of the automorphism group $\{\id, \sigma_0\}$ of the
matroid, where $\sigma_0 = (1\,2)(6\,7)(8\,11)(9\,12)(10\,13)$, with
presentations of $G(\A^\pm)$ obtained from the real picture
\cite{Randell-inv82} and by \textsc{sirocco} \cite{Marco-Rodriguez-icms16},
which agree at class three, up to isomorphisms inducing the identity on
$H$, with each other and with that of
\cite{Artal-Guerville-ViuSos-expmath20}.  In each case, $M$ has elementary
divisors $1$, with multiplicity $251$, and $2$, and $b$ represents a class
of order~$2$ in $\coker M$; so the system is solvable over $\Z[1/2]$ but not
modulo~$2$.  For $\sigma = \id$ and $\varepsilon = 1$, this gives a
homomorphism $(Q_1)_{(p)} \to (Q_2)_{(p)}$ inducing the identity on
$H \otimes \widehat\Z_p$, which is an isomorphism by the argument given
above for Rybnikov's pair, since $\gr_r(G(\A^\pm))$ is free abelian of rank
$12$, $23$, and $76$ for $r \le 3$
\cite[Thm.~2.3]{Artal-Guerville-ViuSos-expmath20}.
\end{proof}

The computation also explains the separation at the prime~$2$, and gives a
second proof of it.  The analogue of Lemma~\ref{lem:ryb-rigid} holds, for
every field $K$ of characteristic zero, with $\sigma \in \{\id, \sigma_0\}$;
we have verified it by machine, as its proof does not carry over to the
nineteen components of $\RR^1(\A^\pm)$: the local components of the
thirteen multiple points, and six coming from braid subarrangements.

The Lie algebra of the stabilizer of
$R \otimes \Q$ consists of the scalars.  An element of the stabilizer
permutes the nineteen components of $\RR^1$, hence their annihilators in
$H \otimes K$ and all intersections of these.  The $412$ one-dimensional
intersections form, with the nineteen annihilators, an incidence structure
whose automorphism group has order~$2$ and maps meridian lines to meridian
lines.  So the element permutes the lines $K x_t$, by $\id$ or $\sigma_0$,
and, as in the proof of Lemma~\ref{lem:ryb-rigid}, it is $c\,P_\sigma$.  An
isomorphism $\psi \colon (Q_1)_{(2)} \to (Q_2)_{(2)}$ would therefore act on
$H \otimes \widehat\Z_2$ as $u\,P_\sigma$, with $u \in \widehat\Z_2^\times$,
as in the proof of Theorem~\ref{thm:rybnikov}.  The lines of $\A^+$ are
defined over $\Q$, so Lemma~\ref{lem:galois-units} provides an automorphism
$\Phi$ of $\widehat{G(\A^+)}$ with $\Phi_{\ab} = v \cdot \id$, where
$v \in \widehat\Z^\times$ has $2$-adic component $u^{-1}$ and all other
components~$1$; by Lemma~\ref{lem:pro-p-class3}\ref{pro-p-induced}, it
induces an automorphism $\Phi_{(2)}$ of $(Q_1)_{(2)}$ acting on
$H \otimes \widehat\Z_2$ as multiplication by $u^{-1}$.  Then
$\psi \circ \Phi_{(2)}$ acts on $H \otimes \widehat\Z_2$ as $P_\sigma$,
which the unsolvability modulo~$2$ rules out.  Thus the extensions of
$G(\A^\pm)/\gamma_3(G(\A^\pm))$ by $\gr_3(G(\A^\pm))$ have isomorphic ends but
incompatible $k$-invariants, at the prime~$2$ only, with an obstruction of
order~$2$; at that prime, the incompatibility is visible in the Schur
multipliers.  For Rybnikov's pair, the corresponding incompatibility at the
prime~$3$ (Proposition~\ref{prop:ryb-obstruction}) leaves the Schur
multipliers unchanged, since $\gr_4(G^{++}) \cong \gr_4(G^{+-})$, and only
the class-three computation detects it; we do not know why the two pairs
behave differently.

\begin{proof}[Proof of Theorems~\textup{\ref{thm:arr}} and~\textup{\ref{thm:ryb-pair}}]
In Theorem~\ref{thm:arr}, part~\ref{arr-galois} is Corollaries~\ref{cor:galois} 
and~\ref{cor:no-go}, part~\ref{arr-agv} is Corollary~\ref{cor:agv}, and part~\ref{arr-acgm} 
is Corollaries~\ref{cor:acgm} and~\ref{cor:acgm-rf}.  Theorem~\ref{thm:ryb-pair} is 
 Theorem~\ref{thm:rybnikov}.
\end{proof}

\subsection{Characteristic varieties and the intersection lattice}
\label{subsec:search}

The resonance varieties of an arrangement are determined by its
intersection lattice.  For the characteristic varieties, the question is open.

\begin{question}[{\cite[Question~6.1]{Suciu-topappl02}, \cite[Problem~3.15]{Suciu-toulouse14}}]
\label{quest:cv-comb}
Is $\VV^1_s(\A)$ determined by $L_{\le 2}(\A)$?  That is, if
$L_{\le 2}(\A_1) \cong L_{\le 2}(\A_2)$, is
$\VV^1_s(\A_1) = \VV^1_s(\A_2)$ for all $s$, under the identification of
character tori induced by the resulting bijection of hyperplanes?
\end{question}

A \emph{counterexample} to Question~\ref{quest:cv-comb} is a pair of
lattice-isomorphic arrangements for which these loci differ.

\begin{remark}
\label{rem:through-one}
The components of $\VV^1_s(\A)$ through $\one$ are combinatorially
determined: by formality and the tangent cone theorem
\cite[Thm.~A]{Dimca-Papadima-Suciu-duke09}, they are the subtori
$\exp(L)$ for $L$ a component of $\RR^1_s(\A)$, and $\RR^1_s(\A)$ is
determined by the Orlik--Solomon algebra, each of its components arising
from a multinet on a subarrangement \cite{Falk-Yuzvinsky-comp07}.  So a
counterexample must differ in translated components or in isolated torsion
points.
\end{remark}

The following corollary turns Theorem~\ref{thm:main} into a criterion for
separating lattice-isomorphic arrangements, and thereby constrains the
possible counterexamples.

\begin{corollary}
\label{cor:search}
Let $\A_1$ and $\A_2$ be arrangements.
\begin{enumerate}[label=\textup{(\roman*)}, itemsep=1.5pt]
\item \label{search-rig}
Suppose $L_{\le 2}(\A_1) \cong L_{\le 2}(\A_2)$, and
$\VV^1_s(\A_1) \ne \VV^1_s(\A_2)$ for some $s$, the two being compared
under the identification of character tori induced by the resulting
bijection of hyperplanes.  Then
$\widehat{G(\A_1)} \not\cong \widehat{G(\A_2)}$ by any isomorphism
rigidified by that bijection.
\item \label{search-isog}
Suppose that, for some $s$ and $d$, the varieties $\VV^1_s(\A_1)$ and
$\VV^1_s(\A_2)$ have different numbers of irreducible components of
dimension~$d$, or different numbers of irreducible components of
dimension~$d$ passing through $\one$.  Then
$\widehat{G(\A_1)} \not\cong \widehat{G(\A_2)}$.
\end{enumerate}
\end{corollary}

\begin{proof}
Property $\mathsf{TT}_1$ holds for arrangement complements, and $\CH_1$
holds by Proposition~\ref{prop:CH}\ref{CH-one}.  Part~\ref{search-rig} is then
Theorem~\ref{thm:main}\ref{main-loci}, and part~\ref{search-isog} follows from
Theorem~\ref{thm:main}\ref{main-isogeny} and Remark~\ref{rem:isogeny-scope}\ref{isog-a}.
\end{proof}

By Corollaries~\ref{cor:galois} and~\ref{cor:search}, a counterexample
must be a pair whose profinite
completions admit no isomorphism rigidified by the bijection of hyperplanes
underlying a lattice isomorphism, and no isomorphism at all if the two
loci differ in their numbers of components of some dimension; in
particular, it cannot be arithmetic.  This excludes the MacLane pair 
and the twelve-line arrangements of \cite{Guerville-Balle-top16}, the 
latter including a pair with non-isomorphic groups (Corollary~\ref{cor:acgm}).  
It does not exclude Rybnikov's original pair \cite{Rybnikov-faa11}, which is 
nonarithmetic \cite[Ex.~1.3(2)]{Guerville-Balle-arx24} and whose completions 
are not isomorphic (Theorem~\ref{thm:rybnikov}); the pair of
Corollary~\ref{cor:agv}, whose completions are not isomorphic either; or the
nonarithmetic pairs produced by Guerville-Ball\'e's splitting-polygon method
\cite{Guerville-Balle-arx24}.

Rybnikov's pair is not a counterexample, however: a complete machine
computation, exact over $\Q$ and with both sets of presentations from the
proof of Proposition~\ref{prop:ryb-obstruction}, shows that
$\VV^1_s(\A^{++}) = \VV^1_s(\A^{+-})$ for all $s$, under the identity
labeling of the lines.  In both cases, $\VV^1$ is the union of the fifteen
local components, one for each three-element flat of $\mathsf{M}_{13}$, each
of dimension $2$, and $\VV^1_s = \{\one\}$ for $2 \le s \le 12$; in
particular, $\VV^1$ is determined by $\mathsf{M}_{13}$ here.  So the
characteristic varieties do not separate Rybnikov's pair, while the profinite
completions do (Theorem~\ref{thm:rybnikov}).

The same holds for the pair of Corollary~\ref{cor:agv}: a complete
computation of the same kind, with the presentations from the proof of
Proposition~\ref{prop:agv-class3}, shows that
$\VV^1_s(\A^+) = \VV^1_s(\A^-)$ for all $s$, under every lattice isomorphism.  
So the characteristic varieties do not separate this pair either, while the
profinite completions do, already through the
pro-$2$ completions of the third nilpotent quotients
(Proposition~\ref{prop:agv-class3}).

\subsection*{Acknowledgments}
I thank Sam Hughes for bringing related work to my attention, which led to a 
sharpening and extension of some of the results herein.

Parts of this work were carried out with the assistance of Anthropic's Claude (Opus 5.5), 
including its Claude Code agent. It was used to implement and cross-check by independent 
methods the machine computations in Section~\ref{sect:arrangements}, verify citations 
against original sources, and assist in writing up and polishing parts of the text and proofs 
developed by the author. All computational scripts, references, and mathematical exposition 
have been carefully checked by the author, who takes full responsibility for the contents 
of this paper.

\renewcommand{\MR}[1]{\href{https://mathscinet.ams.org/mathscinet-getitem?mr=#1}{MR~#1}}
\newcommand{\DOIhref}[2]{\href{https://doi.org/#1}{#2}}


\begin{thebibliography}{10}

\bibitem{Arapura-jag97}
D.~Arapura, \emph{Geometry of cohomology support loci for local systems.
  \textup{I}}, J. Algebraic Geom. \textbf{6} (1997), no.~3, 563--597.
  \MR{1487227}

\bibitem{ACCM-aspm07}
E.~Artal~Bartolo, J.~Carmona~Ruber, J.~I. Cogolludo~Agust\'{i}n, and M.~A.
  Marco~Buzun\'{a}riz, \DOIhref{10.2969/aspm/04310001}{\emph{Invariants of
  combinatorial line arrangements and {R}ybnikov's example}}, Singularity
  Theory and its Applications, Adv. Studies in Pure Math., vol.~43, Math. Soc.
  Japan, 2007, pp.~1--34. \MR{2313406}

\bibitem{ACGM-racsam17}
E.~Artal~Bartolo, J.~I. Cogolludo-Agust\'{i}n, B.~Guerville-Ball\'{e}, and
  M.~Marco-Buzun\'{a}riz, \DOIhref{10.1007/s13398-016-0298-y}{\emph{An
  arithmetic {Z}ariski pair of line arrangements with non-isomorphic
  fundamental group}}, Rev. R. Acad. Cienc. Exactas F\'{i}s. Nat. Ser. A Mat.
  RACSAM \textbf{111} (2017), no.~2, 377--402. \MR{3623846}

\bibitem{Artal-Cogolludo-Matei-gt13}
E.~Artal~Bartolo, J.~I. Cogolludo~Agust\'{i}n, and D.~Matei,
  \DOIhref{10.2140/gt.2013.17.273}{\emph{Characteristic varieties of
  quasi-projective manifolds and orbifolds}}, Geom. Topol. \textbf{17} (2013),
  no.~1, 273--309. \MR{3035328}

\bibitem{Artal-Guerville-ViuSos-expmath20}
E.~Artal~Bartolo, B.~Guerville-Ball\'{e}, and J.~Viu-Sos,
  \DOIhref{10.1080/10586458.2018.1428131}{\emph{Fundamental groups of real
  arrangements and torsion in the lower central series quotients}}, Exp. Math.
  \textbf{29} (2020), no.~1, 28--35. \MR{4067904}

\bibitem{Bauer-Catanese-Grunewald-inv15}
I.~Bauer, F.~Catanese, and F.~Grunewald,
  \DOIhref{10.1007/s00222-014-0531-2}{\emph{Faithful actions of the absolute
  {G}alois group on connected components of moduli spaces}}, Invent. Math.
  \textbf{199} (2015), no.~3, 859--888. \MR{3314516}

\bibitem{Bessis-ann15}
D.~Bessis, \DOIhref{10.4007/annals.2015.181.3.1}{\emph{Finite complex
  reflection arrangements are {$K(\pi,1)$}}}, Ann. of Math. (2) \textbf{181}
  (2015), no.~3, 809--904. \MR{3296817}

\bibitem{Boileau-Friedl-amstud20}
M.~Boileau and S.~Friedl, \DOIhref{10.2307/j.ctvthhdvv.5}{\emph{The profinite
  completion of {$3$}-manifold groups, fiberedness and the {T}hurston norm}},
  What's Next? The Mathematical Legacy of {W}illiam {P}. {T}hurston, Ann. of
  Math. Stud., vol. 205, Princeton Univ. Press, Princeton, NJ, 2020,
  pp.~21--44. \MR{4205634}

\bibitem{BMRS-ann20}
M.~R. Bridson, D.~B. McReynolds, A.~W. Reid, and R.~Spitler,
  \DOIhref{10.4007/annals.2020.192.3.1}{\emph{Absolute profinite rigidity and
  hyperbolic geometry}}, Ann. of Math. (2) \textbf{192} (2020), no.~3,
  679--719. \MR{4172619}

\bibitem{Budur-Wang-ens15}
N.~Budur and B.~Wang, \DOIhref{10.24033/asens.2242}{\emph{Cohomology jump loci
  of quasi-projective varieties}}, Ann. Sci. \'{E}c. Norm. Sup\'{e}r. (4)
  \textbf{48} (2015), no.~1, 227--236. \MR{3335842}

\bibitem{Budur-Wang-asens20}
N.~Budur and B.~Wang, \DOIhref{10.24033/asens.2426}{\emph{Absolute sets and the
  decomposition theorem}}, Ann. Sci. \'{E}c. Norm. Sup\'{e}r. (4) \textbf{53}
  (2020), no.~2, 469--536. \MR{4094563}

\bibitem{CDFSSTY-pre09}
D.~C. Cohen, G.~Denham, M.~J. Falk, H.~Schenck, A.~I. Suciu, H.~Terao, and
  S.~Yuzvinsky, \emph{Complex arrangements: algebra, geometry, topology}, draft
  monograph, 2009.

\bibitem{Cohen-Suciu-cmh97}
D.~C. Cohen and A.~I. Suciu, \DOIhref{10.1007/s000140050017}{\emph{The braid
  monodromy of plane algebraic curves and hyperplane arrangements}}, Comment.
  Math. Helv. \textbf{72} (1997), no.~2, 285--315. \MR{1470093}

\bibitem{Deligne-inv72}
P.~Deligne, \DOIhref{10.1007/BF01406236}{\emph{Les immeubles des groupes de
  tresses g\'{e}n\'{e}ralis\'{e}s}}, Invent. Math. \textbf{17} (1972),
  273--302. \MR{0422673}

\bibitem{Delzant-gafa08}
T.~Delzant, \DOIhref{10.1007/s00039-008-0679-2}{\emph{Trees, valuations and the
  {G}reen--{L}azarsfeld set}}, Geom. Funct. Anal. \textbf{18} (2008), no.~4,
  1236--1250. \MR{2465689}

\bibitem{Delzant-mathann10}
T.~Delzant, \DOIhref{10.1007/s00208-009-0468-8}{\emph{L'invariant de
  {B}ieri--{N}eumann--{S}trebel des groupes fondamentaux des vari\'{e}t\'{e}s
  k\"{a}hl\'{e}riennes}}, Math. Ann. \textbf{348} (2010), no.~1, 119--125.
  \MR{2657436}

\bibitem{Denham-Suciu-plms14}
G.~Denham and A.~I. Suciu, \DOIhref{10.1112/plms/pdt058}{\emph{Multinets,
  parallel connections, and {M}ilnor fibrations of arrangements}}, Proc. Lond.
  Math. Soc. \textbf{108} (2014), no.~6, 1435--1470. \MR{3218315}

\bibitem{Dimca-Papadima-Suciu-duke09}
A.~Dimca, S.~Papadima, and A.~I. Suciu,
  \DOIhref{10.1215/00127094-2009-030}{\emph{Topology and geometry of cohomology
  jump loci}}, Duke Math. J. \textbf{148} (2009), no.~3, 405--457. \MR{2527322}

\bibitem{Falk-Randell-inv85}
M.~Falk and R.~Randell, \DOIhref{10.1007/BF01394780}{\emph{The lower central
  series of a fiber-type arrangement}}, Invent. Math. \textbf{82} (1985),
  no.~1, 77--88. \MR{0808110}

\bibitem{Falk-Randell-conm88}
M.~Falk and R.~Randell, \DOIhref{10.1090/conm/078/975081}{\emph{Pure braid
  groups and products of free groups}}, Braids ({S}anta {C}ruz, {CA}, 1986),
  Contemp. Math., vol.~78, Amer. Math. Soc., Providence, RI, 1988,
  pp.~217--228. \MR{975081}

\bibitem{Falk-Yuzvinsky-comp07}
M.~Falk and S.~Yuzvinsky, \DOIhref{10.1112/S0010437X07002722}{\emph{Multinets,
  resonance varieties, and pencils of plane curves}}, Compositio Math.
  \textbf{143} (2007), no.~4, 1069--1088. \MR{2339840}

\bibitem{Grothendieck-mm70}
A.~Grothendieck, \DOIhref{10.1007/BF01719593}{\emph{Repr\'{e}sentations
  lin\'{e}aires et compactification profinie des groupes discrets}},
  Manuscripta Math. \textbf{2} (1970), 375--396. \MR{0262386}

\bibitem{Grunewald-JaikinZapirain-Zalesskii-duke08}
F.~Grunewald, A.~Jaikin-Zapirain, and P.~A. Zalesskii,
  \DOIhref{10.1215/00127094-2008-031}{\emph{Cohomological goodness and the
  profinite completion of {B}ianchi groups}}, Duke Math. J. \textbf{144}
  (2008), no.~1, 53--72. \MR{2429321}

\bibitem{Guerville-Balle-top16}
B.~Guerville-Ball\'{e}, \DOIhref{10.2140/gt.2016.20.537}{\emph{An arithmetic
  {Z}ariski {$4$}-tuple of twelve lines}}, Geom. Topol. \textbf{20} (2016),
  no.~1, 537--553. \MR{3470721}

\bibitem{Guerville-Balle-mathz22}
B.~Guerville-Ball\'{e}, \DOIhref{10.1007/s00209-021-02953-x}{\emph{The
  loop-linking number of line arrangements}}, Math. Z. \textbf{301} (2022),
  no.~2, 1821--1850. \MR{4418338}

\bibitem{Guerville-Balle-arx24}
B.~Guerville-Ball\'{e}, \emph{On the nonconnectedness of moduli spaces of
  arrangements, {II}: construction of nonarithmetic pairs},
  \href{https://arxiv.org/abs/2409.18022}{arXiv:2409.18022}, 2024.

\bibitem{Hall-plms59}
P.~Hall, \DOIhref{10.1112/plms/s3-9.4.595}{\emph{On the finiteness of certain
  soluble groups}}, Proc. London Math. Soc. (3) \textbf{9} (1959), 595--622.
  \MR{110750}

\bibitem{Hamm-Le-asens73}
H.~A. Hamm and D.~T. L{\^{e}}, \DOIhref{10.24033/asens.1250}{\emph{Un
  th\'{e}or\`{e}me de {Z}ariski du type de {L}efschetz}}, Ann. Sci. \'{E}cole
  Norm. Sup. (4) \textbf{6} (1973), 317--355. \MR{401755}

\bibitem{Hanany-arx26}
L.~Hanany, \emph{Regularity of profinite isomorphisms of hyperbolic
  {$3$}-manifolds}, \href{https://arxiv.org/abs/2607.04530}{arXiv:2607.04530},
  2026.

\bibitem{Hironaka-fourier97}
E.~Hironaka, \DOIhref{10.5802/aif.1573}{\emph{Alexander stratifications of
  character varieties}}, Ann. Inst. Fourier (Grenoble) \textbf{47} (1997),
  no.~2, 555--583. \MR{1450425}

\bibitem{Hughes-Kielak-rmi25}
S.~Hughes and D.~Kielak, \DOIhref{10.4171/rmi/1524}{\emph{Profinite rigidity of
  fibring}}, Rev. Mat. Iberoam. \textbf{41} (2025), no.~4, 1305--1336.
  \MR{4912921}

\bibitem{HLPSV-adv26}
S.~Hughes, C.~Llosa~Isenrich, P.~Py, M.~Stover, and S.~Vidussi,
  \DOIhref{10.1016/j.aim.2026.111343}{\emph{Profinite rigidity of {K}\"{a}hler
  groups: {R}iemann surfaces and subdirect products}}, Adv. Math. \textbf{506}
  (2026), Paper No.~111343. \MR{5147201}

\bibitem{JaikinZapirain-gt20}
A.~Jaikin-Zapirain, \DOIhref{10.2140/gt.2020.24.409}{\emph{Recognition of being
  fibered for compact {$3$}-manifolds}}, Geom. Topol. \textbf{24} (2020),
  no.~1, 409--420. \MR{4080486}

\bibitem{Koberda-Suciu-ccm20}
T.~Koberda and A.~I. Suciu,
  \DOIhref{10.1142/S0219199719500160}{\emph{Residually finite rationally {$p$}
  groups}}, Commun. Contemp. Math. \textbf{22} (2020), no.~3, 1950016 (44
  pages). \MR{4082220}

\bibitem{Liu-invent23}
Y.~Liu, \DOIhref{10.1007/s00222-022-01155-4}{\emph{Finite-volume hyperbolic
  {$3$}-manifolds are almost determined by their finite quotient groups}},
  Invent. Math. \textbf{231} (2023), no.~2, 741--804. \MR{4542705}

\bibitem{Liu-Liu-arx25}
Y.~Liu and Y.~Liu, \emph{Bieri--{N}eumann--{S}trebel--{R}enz invariants and
  tropical varieties of integral homology jump loci},
  \href{https://arxiv.org/abs/2502.16026}{arXiv:2502.16026}, 2025.

\bibitem{Liu-Suciu-arx26}
Y.~Liu and A.~I. Suciu, \emph{Twisted homology jump loci, twisted {A}lexander
  polynomials, and {$\Sigma$}-invariants},
  \href{https://arxiv.org/abs/2605.28595}{arXiv:2605.28595}, 2026.

\bibitem{Marco-gc09}
M.~A. Marco~Buzun\'{a}riz, \DOIhref{10.1007/s00373-009-0863-7}{\emph{A
  description of the resonance variety of a line combinatorics via
  combinatorial pencils}}, Graphs Combin. \textbf{25} (2009), no.~4, 469--488.
  \MR{2575595}

\bibitem{Marco-Rodriguez-icms16}
M.~A. Marco~Buzun\'{a}riz and M.~Rodr\'{\i}guez,
  \DOIhref{10.1007/978-3-319-42432-3\_24}{\emph{{SIROCCO}: a library for
  certified polynomial root continuation}}, Mathematical Software -- {ICMS}
  2016, Lecture Notes in Comput. Sci., vol. 9725, Springer, Cham, 2016,
  pp.~191--197.

\bibitem{Marin-aspm12}
I.~Marin, \DOIhref{10.2969/aspm/06310337}{\emph{Galois actions on complex braid
  groups}}, Galois--{T}eichm\"{u}ller Theory and Arithmetic Geometry, Adv.
  Stud. Pure Math., vol.~63, Math. Soc. Japan, Tokyo, 2012, pp.~337--357.
  \MR{3051247}

\bibitem{Matei-Suciu-aspm00}
D.~Matei and A.~I. Suciu, \DOIhref{10.2969/aspm/02710185}{\emph{Cohomology
  rings and nilpotent quotients of real and complex arrangements}},
  Arrangements---Tokyo 1998 (M.~Falk and H.~Terao, eds.), Adv. Stud. Pure
  Math., vol.~27, Kinokuniya, Tokyo, 2000, pp.~185--215. \MR{1796900}

\bibitem{Nikolov-Segal-abb07}
N.~Nikolov and D.~Segal, \DOIhref{10.4007/annals.2007.165.171}{\emph{On
  finitely generated profinite groups. {I}. {S}trong completeness and uniform
  bounds}}, Ann. of Math. (2) \textbf{165} (2007), no.~1, 171--238.
  \MR{2276769}

\bibitem{Nikolov-Wykowski-mathann26}
N.~Nikolov and J.~Wykowski, \DOIhref{10.1007/s00208-026-03548-9}{\emph{The
  profinite rigidity of torsion-free lamplighter groups}}, Math. Ann.
  \textbf{396} (2026), no.~1, Paper No. 17, 15. \MR{5119484}

\bibitem{Pickel-tams71}
P.~F. Pickel, \DOIhref{10.2307/1995809}{\emph{Finitely generated nilpotent
  groups with isomorphic finite quotients}}, Trans. Amer. Math. Soc.
  \textbf{160} (1971), 327--341. \MR{291287}

\bibitem{Porter-Suciu-ejm20}
R.~D. Porter and A.~I. Suciu,
  \DOIhref{10.1007/s40879-019-00392-x}{\emph{Homology, lower central series,
  and hyperplane arrangements}}, Eur. J. Math. \textbf{6} (2020), no.~3,
  1039--1072. \MR{4151728}

\bibitem{Randell-inv82}
R.~Randell, \DOIhref{10.1007/BF01389187}{\emph{The fundamental group of the
  complement of a union of complex hyperplanes}}, Invent. Math. \textbf{69}
  (1982), 103--108. \MR{0671654}

\bibitem{Reid-lms15}
A.~W. Reid, \emph{Profinite properties of discrete groups}, Groups {S}t
  {A}ndrews 2013, London Math. Soc. Lecture Note Ser., vol. 422, Cambridge
  Univ. Press, Cambridge, 2015, pp.~73--104. \MR{3445488}

\bibitem{Reid-icm18}
A.~W. Reid, \emph{Profinite rigidity}, Proceedings of the {I}nternational
  {C}ongress of {M}athematicians---{R}io de {J}aneiro 2018. {V}ol. {II}.
  {I}nvited lectures, World Sci. Publ., Hackensack, NJ, 2018, pp.~1193--1216.
  \MR{3966805}

\bibitem{Ribes-Zalesskii-10}
L.~Ribes and P.~Zalesskii, \DOIhref{10.1007/978-3-642-01642-4}{\emph{Profinite
  groups}}, second ed., Ergebnisse der Mathematik und ihrer Grenzgebiete. 3.
  Folge. A Series of Modern Surveys in Mathematics, vol.~40, Springer-Verlag,
  Berlin, 2010. \MR{2599132}

\bibitem{Rybnikov-faa11}
G.~Rybnikov, \DOIhref{10.1007/s10688-011-0015-8}{\emph{On the fundamental group
  of the complement of a complex hyperplane arrangement}}, Funct. Anal. Appl.
  \textbf{45} (2011), no.~2, 137--148. \MR{2848779}

\bibitem{Serre-cras64}
J.-P. Serre, \emph{Exemples de vari\'{e}t\'{e}s projectives conjugu\'{e}es non
  hom\'{e}omorphes}, C. R. Acad. Sci. Paris \textbf{258} (1964), 4194--4196.
  \MR{0166197}

\bibitem{Serre-97}
J.-P. Serre, \DOIhref{10.1007/978-3-642-59141-9}{\emph{Galois cohomology}},
  Springer Monographs in Mathematics, Springer-Verlag, Berlin, 1997.
  \MR{1466966}

\bibitem{Simpson-ens93}
C.~Simpson, \DOIhref{10.24033/asens.1675}{\emph{Subspaces of moduli spaces of
  rank one local systems}}, Ann. Sci. \'{E}cole Norm. Sup. \textbf{26} (1993),
  no.~3, 361--401. \MR{1222278}

\bibitem{Stover-pams19}
M.~Stover, \DOIhref{10.1090/proc/14763}{\emph{Lattices in {${\rm PU}(n,1)$}
  that are not profinitely rigid}}, Proc. Amer. Math. Soc. \textbf{147} (2019),
  no.~12, 5055--5062. \MR{4021068}

\bibitem{Stover-jems26}
M.~Stover, \DOIhref{10.4171/JEMS/1536}{\emph{Algebraic fundamental groups of
  fake projective planes}}, J. Eur. Math. Soc. \textbf{28} (2026), no.~8,
  3623--3652. \MR{5062489}

\bibitem{Suciu-conm01}
A.~I. Suciu, \DOIhref{10.1090/conm/276/04510}{\emph{Fundamental groups of line
  arrangements: Enumerative aspects}}, Advances in algebraic geometry motivated
  by physics (Lowell, MA, 2000) (E.~Previato, ed.), Contemp. Math., vol. 276,
  Amer. Math. Soc., Providence, RI, 2001, pp.~43--79. \MR{1837109}

\bibitem{Suciu-topappl02}
A.~I. Suciu, \DOIhref{10.1016/S0166-8641(01)00052-9}{\emph{Translated tori in
  the characteristic varieties of complex hyperplane arrangements}}, Topology
  Appl. \textbf{118} (2002), no.~1-2, 209--223, Arrangements in Boston: a
  Conference on Hyperplane Arrangements (1999). \MR{1877726}

\bibitem{Suciu-toulouse14}
A.~I. Suciu, \DOIhref{10.5802/afst.1412}{\emph{Hyperplane arrangements and
  {M}ilnor fibrations}}, Ann. Fac. Sci. Toulouse Math. \textbf{23} (2014),
  no.~2, 417--481. \MR{3205599}

\bibitem{Suciu-mathann21}
A.~I. Suciu, \DOIhref{10.1007/s00208-021-02172-z}{\emph{Sigma-invariants and
  tropical varieties}}, Math. Ann. \textbf{380} (2021), no.~3-4, 1427--1463.
  \MR{4297190}

\bibitem{Suciu-pisa24}
A.~I. Suciu, \DOIhref{10.2422/2036-2145.202112\_005}{\emph{Alexander invariants
  and cohomology jump loci in group extensions}}, Ann. Sc. Norm. Super. Pisa
  Cl. Sci. (5) \textbf{25} (2024), no.~2, 1085--1154. \MR{4778471}

\bibitem{Suciu-bullroum25}
A.~I. Suciu,
  \href{https://bulletin.ssmr.ro/pdf/68-3/articol_9.pdf}{\emph{Resonance
  varieties and {L}ie algebras of matroids and hyperplane arrangements}}, Bull.
  Math. Soc. Sci. Math. Roumanie (N.S.) \textbf{68} (2025), no.~3, 349--382.
  \MR{4990670}

\bibitem{Suciu-isolated-points}
A.~I. Suciu, \emph{Isolated points and translated tori in characteristic
  varieties}, draft, 2026.

\bibitem{Suciu-linking-classes}
A.~I. Suciu, \emph{Linking classes of line arrangements}, draft, 2026.

\bibitem{Toledo-ihes93}
D.~Toledo, \DOIhref{10.1007/BF02699189}{\emph{Projective varieties with
  non-residually finite fundamental group}}, Inst. Hautes \'{E}tudes Sci. Publ.
  Math. (1993), no.~77, 103--119. \MR{1249171}

\bibitem{Ueki-agt18}
J.~Ueki, \DOIhref{10.2140/agt.2018.18.3013}{\emph{The profinite completions of
  knot groups determine the {A}lexander polynomials}}, Algebr. Geom. Topol.
  \textbf{18} (2018), no.~5, 3013--3030. \MR{3848406}

\bibitem{Wang-mrl16}
B.~Wang, \DOIhref{10.4310/MRL.2016.v23.n2.a12}{\emph{Torsion points on the
  cohomology jump loci of compact {K}\"{a}hler manifolds}}, Math. Res. Lett.
  \textbf{23} (2016), no.~2, 545--563. \MR{3512898}

\bibitem{Xu-arx25}
X.~Xu, \emph{On regularity of profinite isomorphisms between cusped hyperbolic
  {$3$}-manifolds and the {A}-polynomial},
  \href{https://arxiv.org/abs/2506.21105}{arXiv:2506.21105}, 2025.

\end{thebibliography}
\end{document}